\documentclass[12pt,letterpaper,notitlepage]{article}
\usepackage{eurosym}
\usepackage{amsmath}
\usepackage{float}
\usepackage[left=1in,
            right=1in,
            top=1in,
            bottom=1in]{geometry}
\usepackage[footnotesize]{caption}
\usepackage{subcaption}
\usepackage{color}
\usepackage[pdftex]{graphicx}
\usepackage{graphicx}
\usepackage{psfrag}
\usepackage{Here}
\usepackage{multirow}
\usepackage{pgf}
\usepackage{color}
\usepackage{pgf,pgfarrows,pgfnodes,pgfautomata,pgfheaps,pgfshade}
\usepackage[T1]{fontenc}
\usepackage{color}
\usepackage{amsfonts}
\usepackage{amssymb, amsmath, amsthm,bm}
\usepackage{lscape}
\usepackage{natbib}
\usepackage{enumerate}
\usepackage{bbm}
\usepackage{lscape}
\usepackage{diagbox}
\usepackage{comment}
\usepackage{makecell}
\usepackage{titling}
\usepackage{setspace}
\usepackage{rotating}
\usepackage{mathtools}
\usepackage{xr-hyper}
\usepackage{afterpage}
\usepackage{lipsum}
\usepackage{enumitem}
 \usepackage[colorlinks,linkcolor={red},citecolor={blue},urlcolor={red},filecolor={red}]{hyperref}
\allowdisplaybreaks

\numberwithin{equation}{section}

\renewcommand{\arraystretch}{1.25}
\newtheorem{theorem}{Theorem}[section]

\newtheorem{assumption}{Assumption}

\newtheorem{lemma}{Lemma}[section]

\newtheorem{proposition}{Proposition}[section]
\newtheorem{remark}{Remark}

\makeatletter
\newcommand{\specialcell}[1]{\ifmeasuring@#1\else\omit$\displaystyle#1$\ignorespaces\fi}
\makeatother

\title{Bootstrap Inference under General Two-way Clustering with Serially and Spatially Dependent Common Effects\thanks{A previous version of this paper was circulated under the title
\textit{``Projection-Based Wild Bootstrap Under General Two-Way Cluster Dependence with Serial Dependence''}.}}
\author{ {\large \textbf{Ulrich Hounyo}\thanks{%
Department of Economics, University at Albany -- State University of New
York, Albany, NY 12222, United States.}} \and {\large \textbf{Jiahao Lin}%
\thanks{%
Department of Economics, University at Albany -- State University of New
York, Albany, NY 12222, United States.} }}

\makeatother

\begin{document}
\maketitle 
\begin{abstract}
\noindent This paper develops bootstrap procedures for inference in linear regression
models with two-way clustered data. We characterize the estimator's
asymptotic behavior in five mutually exclusive and exhaustive regimes: three
Gaussian and two non-Gaussian. We establish four impossibility results: heterogeneous score components preclude uniform
consistency; uniform consistency also fails in one
non-Gaussian (infeasible) regime; the infeasible regime is not uniformly distinguishable from a
feasible one; and uniform validity over all feasible regimes rules out uniform
conservativeness over the infeasible regime.

To address the feasible regimes, we propose a data-driven regime classifier
and a projection-based wild bootstrap procedure. The procedure delivers
uniformly valid inference across the four feasible regimes while allowing
serial dependence along the second clustering dimension and spatial dependence
along the first. This combination of regime adaptivity and flexible dependence
is new to the two-way clustering literature. Monte Carlo
simulations confirm the accuracy and flexibility of the proposed methods in
settings with complex clustering structures.

\bigskip{}
 \textbf{JEL Classification}: C15, C23, C31, C80

\noindent \medskip{}
 \textbf{Keywords}: Bootstrap, clustered
data, two-way clustering, robust inference, wild bootstrap. 
\end{abstract}

\vspace*{-0.5cm}

\vfill{}

\thispagestyle{empty} \pagebreak{}

\section{Introduction}

Understanding dependence structures is essential for valid statistical
inference. In many empirical settings, the classical assumption of
independence is violated, especially in time series and panel data
contexts. Ignoring such dependence, whether temporal or cross-sectional,
can lead to biased standard errors, inflated Type I error rates, and
invalid inference, as emphasized by Bertrand, Duflo, and Mullainathan~(\citeyear{bertrand2004much})
and Petersen~(\citeyear{petersen2008estimating}).

This paper addresses a key challenge in modern econometrics: conducting
inference under two-way clustering when the common effects may exhibit
arbitrary forms of serial and spatial dependence. This setting is
practically relevant but, to the best of our knowledge, has not been
studied in the existing literature.  We first characterize the asymptotic distribution of the estimator
$\widehat{\bm{\beta}}$ under general two-way clustering structures, yielding five mutually exclusive
and exhaustive regimes, classified by the relative growth of variance
along the two clustering dimensions: three Gaussian and two non-Gaussian. 

Building on the insight of Davezies,
Haultf{\ae}uille, and Guyonvarch~(\citeyear{davezies2025analytic}), we show that
heterogeneous score components impose a fundamental limit on uniformly
consistent inference under two-way clustering.
Moreover, based on Menzel~(\citeyear{menzel2021bootstrap}), who shows that
no inference procedure can achieve uniform consistency across the
full function space, we extend and sharpen this result by isolating
a particularly challenging regime. In this infeasible regime, consistent
inference is still impossible when the DGP is fully unspecified. We
further demonstrate that lack of knowledge of the true variance along
both dimensions is the \emph{only} obstacle to attaining uniform consistency
across the full function space. We also establish that uniform validity
across the four feasible regimes necessarily rules out uniform conservativeness
in the infeasible one.

Next, we develop a data-driven classification procedure that distinguishes
among all the feasible regimes using two informative discriminants.
This enables researchers to tailor inference procedures to the prevailing
asymptotic environment. However, we also show that the infeasible
scenario is not uniformly distinguishable from one feasible regime,
imposing fundamental limits on what can be learned from the data without
stronger assumptions.

Third, we introduce a family of projection-based wild bootstrap (PWB)
procedures and establish their \emph{uniform} consistency across different
regimes.\footnote{The code used to implement the methods in this paper is available
at:\\ \url{https://jiahaoecon.github.io/webpage/research/}.} Among these, the PWB-H variant delivers uniformly valid inference
across all four feasible regimes. The theory uses joint large-$N$, large-$T$ asymptotics,
but imposes no restriction on the relative growth of $N$ and $T$.

When temporal and spatial dependencies are absent, some PWB variants
closely resemble the approach of Menzel~(\citeyear{menzel2021bootstrap}).
Unlike Menzel (\citeyear{menzel2021bootstrap})'s bootstrap, which
mixes wild and i.i.d. components, our method applies the wild bootstrap
to projected scores, which preserves the correlations for uniform
consistency. Moreover, PWB shares conceptual similarities with the
method of Juodis~(\citeyear{juodis2021shock}). Our approach differs
mainly in two key respects. First, the procedure
detects whether the limiting distribution is asymptotically Gaussian
and, based on this diagnostic, to adaptively select the appropriate
regime, thereby achieving simultaneous validity across multiple scenarios.
Consequently, even though PWB employs the tuning thresholds as in
Menzel~\citeyearpar{menzel2021bootstrap} and Juodis~\citeyearpar{juodis2021shock},
it does \emph{not} inherit the boundary non-uniformity induced by
the indifference region that is intrinsic to purely threshold-based
tuning. Second, we introduce a scaling adjustment that corrects the
estimation error of the variance components along both clustering
dimensions, which is crucial for validity in regimes where the contribution
of the noise term is asymptotically negligible. 

Overall, the paper makes six contributions: (i) extends the framework to allow spatial dependence
along the first dimension; (ii) shows heterogeneous score components preclude uniform consistency; (iii) characterizes five
mutually exclusive and exhaustive asymptotic regimes for two-way clustering;
(iv) pinpoints one infeasible regime where uniform consistency fails,
showing that unknown true variance is the sole obstacle to uniform
consistency over the full function space, and further establishes
that uniform validity across the four feasible regimes necessarily
precludes uniform conservativeness in the infeasible regime; (v)
develops a data-driven regime classifier for the four feasible regimes
and shows that the infeasible regime cannot be uniformly distinguished; and 
(vi) proposes a projection-based wild bootstrap family and a hybrid
implementation that is uniformly asymptotically valid across all feasible
regimes. To the best of our knowledge, this level of
generality is new to the two-way clustering literature.

The rest of the paper is organized as follows. Section~\ref{sec:model}
presents the two-way clustering model and the five asymptotic regimes.
Section~\ref{sec:method} introduces the PWB procedures and develops
the theory for bootstrap validity. Section~\ref{sec:simulations}
reports various simulation results under five regimes. Section~\ref{sec:conclusion}
concludes. Technical proofs and additional results are deferred to
the appendix.

\section{Model Setting and Five Asymptotic Regimes}

\label{sec:model}

\subsection{Two-way Clustering with Serially and Spatially Dependent Common Effects}

We consider a linear regression model with two clustering dimensions.
Let $i=1,\dots,N$ index clusters in the first dimension and $t=1,\dots,T$
index clusters in the second dimension. For each intersection $(i,t)$,
suppose that 
\begin{equation}
\bm{y}_{it}=\bm{X}_{it}\bm{\beta}+\bm{u}_{it},
\end{equation}
where $\bm{y}_{it}$ is an $M_{it}\times1$ vector of outcomes, $\bm{X}_{it}$
is an $M_{it}\times K$ matrix of regressors, $\bm{u}_{it}$ is an
$M_{it}\times1$ vector of disturbances, and $\bm{\beta}$ is a $K\times1$
parameter vector. Here, $M_{it}$ denotes the number of observations
in intersection $(i,t)$. This setup allows for unbalanced panels
as well as multiple observations within a given intersection.

Let 
\[
M_{i}=\sum_{t=1}^{T}M_{it},\qquad M_{t}=\sum_{i=1}^{N}M_{it},\qquad M=\sum_{i=1}^{N}\sum_{t=1}^{T}M_{it},
\]
where $M_{i}$ and $M_{t}$ denote the numbers of observations in
cluster $i$ and cluster $t$, respectively, and $M$ is the total
sample size.

Stacking all observations yields 
\begin{equation}
\bm{y}=\bm{X}\bm{\beta}+\bm{u},\label{eq: linear regression model}
\end{equation}
where $\bm{y}$ is an $M\times1$ vector, $\bm{X}$ is an $M\times K$
matrix, and $\bm{u}$ is an $M\times1$ vector. We use $\bm{y}_{i},\bm{X}_{i},\bm{u}_{i}$
to denote the subvectors and submatrices associated with cluster $i$
in the first dimension, and $\bm{y}_{t},\bm{X}_{t},\bm{u}_{t}$ to
denote those associated with cluster $t$ in the second dimension.

The ordinary least squares estimator of $\bm{\beta}$ is 
\begin{equation}
\widehat{\bm{\beta}}=\left(\bm{X}^{\top}\bm{X}\right)^{-1}\bm{X}^{\top}\bm{y}.
\end{equation}
For simplicity, in the main text we focus on the case in which each
intersection contains the same finite number of observations. The
Internet Appendix IC extends the framework to allow for heterogeneous numbers
of observations and missing intersections.

The main feature of the model is its dependence structure, which 
is illustrated in Figure~\ref{fig:RV-generate-multiway-clustering-with-TS}.
Panels~(a)-(c) show the classical independence and one-way clustering dependence structures. Panel~(d)
depicts the conventional two-way clustering setup with no additional
dependence; see, e.g.,  Davezies,
D'Haultf{œ}uille, and Guyonvarch~(\citeyear{davezies2021empirical},
\citeyear{davezies2022marcinkiewicz}), MacKinnon, Nielsen, and Webb~(\citeyear{mackinnon2021wild}),
and Menzel~(\citeyear{menzel2021bootstrap}). Panel~(e) allows for
serial dependence in the time common effect $\{\bm{\xi}_{t}\}_{t=1}^{T}$; see, e.g., Chiang, Hansen, and Sasaki~(\citeyear{chiang2023standard}), Chen and Vogelsang (\citeyear{chen2023fixed}), Hounyo and Lin (\citeyear{hounyo2025jackknife}, \citeyear{hounyo2024wild}). 
Panel~(f) further extends the existing literature by allowing for spatial
dependence in the cross-sectional common effect
$\{\bm{\alpha}_{i}\}_{i=1}^{N}$. This yields a more general two-way
clustering environment with both serial and spatial dependence.

\begin{figure}[t!]
\centering \includegraphics[width=1\textwidth]{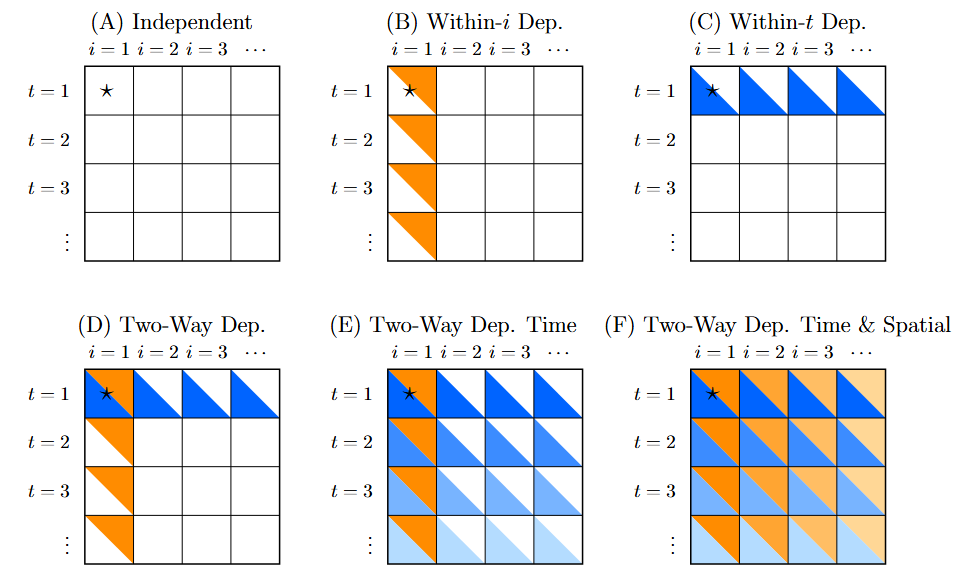} \caption{An illustration of two-way clustering with possibly dependent common
effects. A lighter color indicates weaker dependence. In this figure,
spatial dependence decays as $|i-j|$ increases.}
\label{fig:RV-generate-multiway-clustering-with-TS} 
\end{figure}

Following Conley~\citeyearpar{conley1999gmm}, we assume that the
spatial process $\{\bm{\alpha}_{i}\}_{i=1}^{N}$ is indexed by locations
$\{\mathfrak{s}_{i}\}_{i=1}^{N}\subset\mathcal{H}$, where $\mathcal{H}$
is a regular lattice in $\mathbb{R}^{2}$.\footnote{The restriction to
\(\mathbb R^{2}\) is imposed for expositional simplicity. The arguments can
be extended to \(\mathbb R^{c}\), for any fixed integer \(c>0\), with 
minor modifications.} Let $\mathfrak{d}_{ij}\equiv\|\mathfrak{s}_{i}-\mathfrak{s}_{j}\|$
denote the true Euclidean distance, satisfying $\mathfrak{d}_{ii}=0$,
$\mathfrak{d}_{ij}=\mathfrak{d}_{ji}$, $\mathfrak{d}_{ij}>0$ for
$i\neq j$, and $\mathfrak{d}_{ij}\leq\mathfrak{d}_{ij'}+\mathfrak{d}_{j'i}$.
The spatial
literature also allows $\mathfrak{d}_{ij}$ to be measured with error.
In particular, one may observe $\widetilde{\mathfrak{d}}_{ij}=\mathfrak{d}_{ij}+\varsigma_{ij}$
where $\varsigma_{ij}$ is bounded measurement error; under mild regularity
conditions, replacing $\mathfrak{d}_{ij}$ by $\widetilde{\mathfrak{d}}_{ij}$
 preserves consistency as further discussed in the theory (Conley~\citeyear{conley1999gmm};
see also Conley and Molinari~\citeyear{conley2007spatial}; Kelejian
and Prucha~\citeyear{kelejian2007hac}).

We further 
describe the dependence structure using an Aldous-Hoover-Kallenberg (AHK, Aldous~\citeyear{aldous1981representations};
Hoover~\citeyear{hoover1979relations}; Kallenberg~\citeyear{kallenberg1989representation})
representation, a standard device in the multiway clustering literature.

\begin{assumption}\label{as: AHS representation} There exists a
Borel measurable function $f$ such that 
\begin{equation}
\left(\bm{y}_{it},\bm{X}_{it},\bm{u}_{it}\right)=f\left(\bm{\alpha}_{i},\bm{\xi}_{t},\bm{\varepsilon}_{it}\right),\label{eq: real DGP}
\end{equation}
where $\{\bm{\alpha}_{i}\}$, $\{\bm{\xi}_{t}\}$, and $\{\bm{\varepsilon}_{it}\}$
are mutually independent random elements with uniform marginals on
$[0,1]$. The sequence $\{\bm{\alpha}_{i}\}$ is a strictly stationary spatially dependent process, $\{\bm{\xi}_{t}\}$
is a strictly stationary serially dependent process, and 
$\{\bm{\varepsilon}_{it}\}$ is i.i.d. over $(i,t)$. The function
$f$ is allowed to vary with the sample sizes $N$ and $T$. \end{assumption}

Assumption~\ref{as: AHS representation} allows for dependence beyond the
standard two-way clustering structure. Along the first dimension,
observations may be dependent within a cluster and across nearby clusters
through the spatially correlated common effect \(\bm{\alpha}_{i}\). Along
the second dimension, dependence may arise within a cluster and across
nearby clusters through the serially correlated common effect
\(\bm{\xi}_{t}\). Consequently, observations at intersections \((i,t)\)
and \((i',t')\) may be dependent not only when they share the same cluster
in either dimension, but also when \(i\) and \(i'\) are spatially linked or
when \(t\) and \(t'\) are serially linked.\footnote{All serial and spatial dependence considered in this paper operates,
as is standard, through the common effects \(\bm{\alpha}_i\) and/or
\(\bm{\xi}_t\). Individual-specific
dependence in \(\bm{\varepsilon}_{it}\) would introduce an additional
dependence channel and is outside the scope of the present paper.}

The representation is useful as it allows for a Hoeffding type decomposition
as follows: 
\[
\bm{X}_{it}^{\top}\bm{u}_{it}=\bm{a}_{i}+\bm{d}_{t}+\bm{w}_{it}+E\left(\bm{X}_{it}^{\top}\bm{u}_{it}\right),
\]
where 
\begin{align*}
\bm{a}_{i} & =E(\bm{X}_{it}^{\top}\bm{u}_{it}\mid\bm{\alpha}_{i})-E(\bm{X}_{it}^{\top}\bm{u}_{it}),\\
\bm{d}_{t} & =E(\bm{X}_{it}^{\top}\bm{u}_{it}\mid\bm{\xi}_{t})-E(\bm{X}_{it}^{\top}\bm{u}_{it}),\\
\bm{w}_{it} & =\bm{X}_{it}^{\top}\bm{u}_{it}-E(\bm{X}_{it}^{\top}\bm{u}_{it}\mid\bm{\alpha}_{i})-E(\bm{X}_{it}^{\top}\bm{u}_{it}\mid\bm{\xi}_{t})+E(\bm{X}_{it}^{\top}\bm{u}_{it})=\bm{v}_{it}+\bm{e}_{it},\\
\bm{v}_{it} & =E(\bm{X}_{it}^{\top}\bm{u}_{it}\mid\bm{\alpha}_{i},\bm{\xi}_{t})-E(\bm{X}_{it}^{\top}\bm{u}_{it}\mid\bm{\alpha}_{i})-E(\bm{X}_{it}^{\top}\bm{u}_{it}\mid\bm{\xi}_{t})+E(\bm{X}_{it}^{\top}\bm{u}_{it}),\\
\bm{e}_{it} & =\bm{X}_{it}^{\top}\bm{u}_{it}-E(\bm{X}_{it}^{\top}\bm{u}_{it}\mid\bm{\alpha}_{i},\bm{\xi}_{t}).
\end{align*}
Observe that $E\left(\bm{e}_{it}\mid\bm{\alpha}_{i},\bm{\xi}_{t}\right)=0$.
The application of the tower property of conditional expectation yields
that $E\left(\bm{e}_{it}\mid\bm{a}_{i},\bm{d}_{t},\bm{v}_{it}\right)=0$.
Similarly, it holds that $E\left(\bm{v}_{it}\mid\bm{a}_{i},\bm{d}_{t}\right)=0$.
The decomposition is based on a score expansion tailored to the OLS
setting; for a generic $M$-estimator, one may replace ${\bm{X}}_{it}^{\top}{\bm{u}}_{it}$
with the corresponding score (influence function).

To separate the effect of $\bm{\alpha}_{i}$ and $\bm{\xi}_{t}$ in
$\bm{v}_{it}$, we follow Menzel (\citeyear{menzel2021bootstrap})
and assume a low-rank approximation. Given that any square-integrable
function of $(\bm{\alpha}_{i},\bm{\xi}_{t})$ admits an expansion
in a tensor-product orthonormal basis, each component of $\bm{v}_{it}$,
$v_{it,k}$, admits a spectral decomposition: 
\begin{equation}
v_{it,k}=\sum_{l,l'=1}^{\infty}c_{ll'f,k}\phi_{l,k}\left(\bm{\alpha}_{i}\right)\psi_{l',k}\left(\bm{\xi}_{t}\right)\label{eq: spectral representation}
\end{equation}
under the $L_{2}\left(F_{\alpha\xi}\right)$ norm on the space of
smooth functions of $\left(\alpha,\xi\right)\in\left[0,1\right]^{2}$.
$F_{\alpha\xi}$ is the joint distribution of $\bm{\alpha}_{i},\bm{\xi}_{t}$.
For each $k$, $\left\{ \phi_{l,k}\left(\bm{\alpha}_{i}\right)\right\} _{l}$
and $\left\{ \psi_{l',k}\left(\bm{\xi}_{t}\right)\right\} _{l'}$ are
orthonormal.

\begin{assumption}\label{as: spectral representation} Assume that
there exists a sequence $\left\{ \bar{c}_{ll'}\right\} $ such that
$c_{ll'f,k}\leq\bar{c}_{ll'}$ for each $l$, $l'$ and $k$, and
$\sum_{l,l'=1}^{\infty}\bar{c}_{ll'}^{2}<\infty$. The first three
moments of $\phi_{l,k}\left(\bm{\alpha}_{i}\right)$ and $\psi_{l',k}\left(\bm{\xi}_{t}\right)$
are uniformly bounded by a constant $B>0$ for all $l$, $l'$, and $k$.
\end{assumption} Assumption~\ref{as: spectral representation} restricts
$v_{it,k}$ to be well-approximated
by the leading components of the SVD-type expansion. This is natural in empirical panels where
unit--time dependence is captured by a small number of interactive
effects, as in common-shock/factor-type specifications widely used
in macro--finance and firm-level panels (e.g., interactive fixed
effects, multi-factor return models, and factor-augmented panel regressions).

We now introduce notations that will be used throughout this paper.
Let $\odot$ denote the Hadamard (elementwise) product. For any matrices
$\bm{A}$, write $\mathbf{A}>0$ to indicate that the matrix $\mathbf{A}$
is positive definite. $\lambda_{\max}(\bm{A})$ and $\lambda_{\min}(\bm{A})$
denote the maximum and minimum eigenvalue of $\bm{A}$, respectively.
As is standard in the bootstrap literature, we write $\bm{A}_{n}^{*}\xrightarrow{P^{*}}\bm{A}$
and $\bm{A}_{n}^{*}\xrightarrow{d^{*}}\bm{A}$ to indicate that the
sequence of bootstrap random matrices $\bm{A}_{n}^{*}$ converges
in bootstrap probability and in bootstrap distribution, respectively,
to $\bm{A}$, as $N,T\to\infty$.

\subsection{Five Asymptotic Regimes}

\label{subsec:Limiting-Distribution-five-scenarios}

Define variances in different dimensions as $\bm{\sigma}_{a,f}^{2}=\frac{1}{N}\sum_{i,j=1}^{N}E(\bm{a}_{i}\bm{a}_{j}^{\top})$,
$\bm{\sigma}_{d,f}^{2}=\frac{1}{T}\sum_{t,t'=1}^{T}E\big(\bm{d}_{t}\bm{d}_{t'}^{\top}\big)$,
$\bm{\sigma}_{e,f}^{2}=\frac{1}{NT}\sum_{i,i'=1}^{N}\sum_{t,t'=1}^{T}E\big(\bm{e}_{it}\bm{e}_{i't'}^{\top}\big)$,
and $\bm{\sigma}_{v,f}^{2}=\frac{1}{NT}\sum_{i,i'=1}^{N}\sum_{t,t'=1}^{T}E\big(\bm{v}_{it}\bm{v}_{i't'}^{\top}\big)$.
Moreover, let $\phi_{l}(\bm{\alpha}_{i})\equiv\big(\phi_{l,1}(\bm{\alpha}_{i}),\ldots,\phi_{l,K}(\bm{\alpha}_{i})\big)^{\top}$,
$\psi_{l'}(\bm{\xi}_{t})\equiv\big(\psi_{l',1}(\bm{\xi}_{t}),\ldots,\psi_{l',K}(\bm{\xi}_{t})\big)^{\top}$,
and $\bm{c}_{ll',f}\equiv(c_{ll'f,1},\ldots,c_{ll'f,K})^{\top}$.
Note that the subscript $f$ indicates dependence on the function
$f$ (which may vary with $N$ and $T$); for notational simplicity,
we suppress the explicit $(N,T)$ dependence in the subscript. Assuming
$E\left(\bm{X}_{it}^{\top}\bm{u}_{it}\right)=0$, we can decompose
the mean of scores $\frac{1}{NT}\sum_{i=1}^{N}\sum_{t=1}^{T}\bm{s}_{it}\equiv\frac{1}{NT}\sum_{i=1}^{N}\sum_{t=1}^{T}\bm{X}_{it}^{\top}\bm{u}_{it}$
as follows: 
\begin{align}
\frac{1}{NT}\sum_{i=1}^{N}\sum_{t=1}^{T}\bm{s}_{it}= & \frac{1}{NT}\sum_{i=1}^{N}\sum_{t=1}^{T}\left(\bm{a}_{i}+\bm{d}_{t}+\bm{e}_{it}+\sum_{l,l'=1}^{\infty}\bm{c}_{ll',f}\odot\phi_{l}\left(\bm{\alpha}_{i}\right)\odot\psi_{l'}\left(\bm{\xi}_{t}\right)\right)\nonumber \\
= & \frac{\bm{\sigma}_{a,f}}{\sqrt{N}}\frac{1}{\sqrt{N}}\sum_{i=1}^{N}\bm{\sigma}_{a,f}^{-1}\bm{a}_{i}+\frac{\bm{\sigma}_{d,f}}{\sqrt{T}}\frac{1}{\sqrt{T}}\sum_{t=1}^{T}\bm{\sigma}_{d,f}^{-1}\bm{d}_{t}+\frac{1}{\sqrt{NT}}\frac{\bm{\sigma}_{e,f}}{\sqrt{NT}}\sum_{i=1}^{N}\sum_{t=1}^{T}\bm{\sigma}_{e,f}^{-1}\bm{e}_{it}\nonumber \\
 & +\frac{1}{\sqrt{NT}}\sum_{l,l'=1}^{\infty}\bm{c}_{ll',f}\odot\left(\frac{1}{\sqrt{N}}\sum_{i=1}^{N}\phi_{l}\left(\bm{\alpha}_{i}\right)\right)\odot\left(\frac{1}{\sqrt{T}}\sum_{t=1}^{T}\psi_{l'}\left(\bm{\xi}_{t}\right)\right)\nonumber \\
\equiv & \frac{\bm{\sigma}_{a,f}}{\sqrt{N}}\bm{Z}_{N}^{a}+\frac{\bm{\sigma}_{d,f}}{\sqrt{T}}\bm{Z}_{T}^{d}+\frac{\bm{\sigma}_{e,f}}{\sqrt{NT}}\bm{Z}_{NT}^{e}+\frac{1}{\sqrt{NT}}\sum_{l,l'=1}^{\infty}\bm{c}_{ll',f}\odot\bm{Z}_{N,l}^{\phi}\odot\bm{Z}_{T,l'}^{\psi}.\label{eq:decomposition score}
\end{align}
Here, the correlation $\left\{ Cov\left(\bm{Z}_{N}^{a},\bm{Z}_{N,l}^{\phi}\right)\right\} _{l}$
are not necessarily zero, given that $\bm{Z}_{N}^{a}$ and $\left\{ \bm{Z}_{N,l}^{\phi}\right\} _{l}$
depend on $\left\{ \bm{\alpha}_{i}\right\} _{i}$. The same for $\left\{ Cov\left(\bm{Z}_{T}^{d},\bm{Z}_{T,l'}^{\psi}\right)\right\} _{l'}$.
All remaining components are pairwise uncorrelated. Hence, $\bm{\sigma}_{NT,f}^{2}\equiv Var\left(\frac{1}{NT}\sum_{i=1}^{N}\sum_{t=1}^{T}\bm{s}_{it}\right)=\frac{1}{NT}(T\bm{\sigma}_{a,f}^{2}+N\bm{\sigma}_{d,f}^{2}+\bm{\sigma}_{e,f}^{2}+\bm{\sigma}_{v,f}^{2}).$

Heuristically, the decomposition~\eqref{eq:decomposition score}
implies that for each $k$, the limit law of $\frac{1}{NT}\sum_{i=1}^{N}\sum_{t=1}^{T}{s}_{itk}$
is determined by a collection of (possibly dependent) Gaussian primitives,
with relative magnitudes governed by the variance components $T\sigma_{ak,f}^{2}$,
$N\sigma_{dk,f}^{2}$, and $\sigma_{vk,f}^{2}$, where $\sigma_{ak,f}^{2}$,
$\sigma_{dk,f}^{2}$, and $\sigma_{vk,f}^{2}$ denote the $k$th diagonal
elements of $\bm{\sigma}_{a,f}^{2}$, $\bm{\sigma}_{d,f}^{2}$, and
$\bm{\sigma}_{v,f}^{2}$, respectively. Moreover, denote the (matrix)
square roots of $\sigma_{\bullet k,f}^{2}$ and $\bm{\sigma}_{\bullet,f}^{2}$
by $\sigma_{\bullet k,f}$ and $\bm{\sigma}_{\bullet,f}$, respectively.

\paragraph{(D) Strong clustering in at least one dimension.}

This regime occurs when at least one clustering variance diverges
(D):

\begin{equation}
(\mathrm{D}):\quad T\sigma_{ak,f}^{2}\to\infty\quad\text{or}\quad N\sigma_{dk,f}^{2}\to\infty.\label{eq: both diverge}
\end{equation}
It corresponds to strong clustering along at least one dimension and
generalizes settings such as Condition~(16) of MacKinnon et al. \citeyearpar{mackinnon2021wild},
which keeps $({\sigma}_{ak,f}^{2},{\sigma}_{dk,f}^{2})$ fixed. In
this case, the leading behavior is dominated by $\frac{{\sigma}_{ak,f}}{\sqrt{N}}{Z}_{Nk}^{a}+\frac{{\sigma}_{dk,f}}{\sqrt{T}}{Z}_{Tk}^{d}$.

\paragraph{(V) Vanishing clustering effects.}

Suppose both clustering variances vanish: 
\[
T\sigma_{ak,f}^{2}=o(1),\qquad N\sigma_{dk,f}^{2}=o(1).
\]
The limit then hinges on the interaction component ${\sigma}_{vk,f}^{2}$,
yielding two subcases.

\smallskip{}

\subparagraph{(V\&N) Vanishing clustering but non-Gaussian interaction:}

\begin{equation}
(\mathrm{V\&N}):\quad T\sigma_{ak,f}^{2}=o(1),\quad N\sigma_{dk,f}^{2}=o(1),\quad\sigma_{vk,f}^{2}>0.\label{eq: vanish nongaussian}
\end{equation}
This regime features no (or weak) clustering along either dimension,
while the non-Gaussian interaction component $\frac{1}{\sqrt{NT}}\sum_{l,l'=1}^{\infty}c_{ll'k}\,Z_{Nk,l}^{\phi}\,Z_{Tk,l}^{\psi}$
remains non-negligible. Such behavior may arise in interactive factor
designs, where multiplicative dependence is intrinsic.

\smallskip{}

\subparagraph{(V\&G) Vanishing clustering and Gaussian limit:}

\begin{equation}
(\mathrm{V\&G}):\quad T\sigma_{ak,f}^{2}=o(1),\quad N\sigma_{dk,f}^{2}=o(1),\quad\sigma_{vk,f}^{2}=o(1).\label{eq: vanish gaussian}
\end{equation}
In this case, the leading term reduces to $\frac{{\sigma}_{ek,f}}{\sqrt{NT}}{Z}_{NTk}^{e}$,
which is Gaussian. This corresponds to ``no clustering'' beyond
the intersection level (e.g., it generalizes Condition~(17) of MacKinnon
et al. \citeyearpar{mackinnon2021wild}).

\paragraph{(I) Intermediate clustering strength.}

The intermediate regime arises when no clustering variance diverges,
yet at least one remains non-negligible: 
\[
\max\{T\sigma_{ak,f}^{2},\,N\sigma_{dk,f}^{2}\}\to\varphi_{k}\in(0,\infty).
\]
Again, two subcases are determined by $\sigma_{vk,f}^{2}$.

\smallskip{}

\subparagraph{(I\&N) Intermediate clustering with non-Gaussian interaction:}

\begin{equation}
(\mathrm{I\&N}):\quad\max\{T\sigma_{ak,f}^{2},\,N\sigma_{dk,f}^{2}\}\to\varphi_{k}\in(0,\infty),\qquad\sigma_{vk,f}^{2}>0.\label{eq: converge non-gaussian}
\end{equation}
All terms in~\eqref{eq:decomposition score} may contribute. In fully
unspecified models, the limit may fail to be consistently estimable
due to contamination from $\{{v}_{it,k}\}$: the estimation error
$\widehat{e}_{NT,k}$ can be of order $O_{P}\!\bigl((NT)^{-1/2}{\sigma}_{vk,f}\bigr)$,
which is not negligible relative to the target $O_{P}\!\bigl((NT)^{-1/2}\bigr)$.

\smallskip{}

\subparagraph{(I\&G) Intermediate clustering with Gaussian limit:}

\begin{equation}
(\mathrm{I\&G}):\quad\max\{T\sigma_{ak,f}^{2},\,N\sigma_{dk,f}^{2}\}\to\varphi_{k}\in(0,\infty),\qquad\sigma_{vk,f}^{2}=o(1).\label{eq: converge gaussian}
\end{equation}
Here the interaction contamination vanishes, and a properly designed
procedure can make the estimation error asymptotically negligible.

\begin{remark}[Subsequence reduction and exhaustive regime classification]\label{rem:subsequence-regimes}
Let $\mathcal{B}$ be the class of Borel measurable set of functions
satisfying Assumptions~\ref{as: AHS representation}--\ref{as: same rate},
and allow $f=f_{NT}\in\mathcal{B}$ to vary with $(N,T)$. The induced
parameter vector is compact such that along any sequence $\{f_{NT}\}$
there exists a \emph{convergent} subsequence $\{f_{N_{k}T_{k}}\}$.
Working without loss of generality along an arbitrary convergent subsequence,
the limit necessarily falls into exactly one of the five regimes \eqref{eq: both diverge}-\eqref{eq: converge gaussian}.
Hence, the five regimes are \emph{mutually exclusive and exhaustive}.
For notational simplicity, we keep writing $\{f_{NT}\}$ in place
of the selected subsequence, with the understanding that the proofs
proceed from an arbitrary convergent sequence and use the subsequence
reduction to formally cover all sequences. \end{remark}

\paragraph{Practical relevance of different regimes.}

Among the five regimes, the two Gaussian cases (D) and (V\&G) are
canonical and empirically common, and are largely covered by existing
methods: (D) corresponds to strong clustering in at least one dimension,
whereas (V\&G) corresponds to essentially no clustering beyond the
intersection level. The practical difficulty is that empirical DGPs
need not fall neatly into either extreme. Regime (I\&G) captures the
\emph{transition} between (V\&G) and (D): clustering is present but
not strong enough to behave as in (D), while it is also too strong
to be safely treated as negligible as in (V\&G). In applications,
the boundary separating these regimes is typically unclear as further
demonstrated below, and may be hard to diagnose in finite samples.
Hence, procedures calibrated only for the two endpoints can be sensitive
to local departures from their target regimes, which makes (I\&G)
practically important and motivates inference that is uniformly valid
across the Gaussian continuum (V\&G)--(I\&G)--(D).

Regime (V\&N) is also practically relevant. It arises when main effects
are weak, yet the interaction component remains non-negligible because
common shocks interact with heterogeneous loadings. This structure
can arise in panels with interactive factor features (e.g., $y_{it}=\alpha_{i}\xi_{t}$),
such as asset-pricing panels (common risk shocks with heterogeneous
exposures), firm--time panels (aggregate shocks with heterogeneous
sensitivities), or shift--share type designs (common shocks with
heterogeneous shares). In such environments, treating the limit as
Gaussian may be misleading. Finally, (I\&N) highlights the intrinsic
difficulty of inference when intermediate clustering coexists with
non-vanishing interaction contamination.

\section{Projection-Based Wild Bootstrap (PWB)}

\label{sec:method}

\subsection{Oracle PWB under Known Variance}

For the bootstrap procedure, it is essential to replicate the dependence
structure of the true DGP. We now provide the algorithm procedure
for the projection-based wild bootstrap method. 
\begin{description}
\item [{Algorithm 1.}] \textbf{Projection-Based Wild Bootstrap Algorithm} 

\end{description}
\begin{enumerate}[label=\textbf{Step \arabic*:}, leftmargin=*, itemsep=1ex]
\item Regress $\bm{y}$ on $\bm{X}$ to obtain the regression estimate $\widehat{\bm{\beta}}$,
the residual $\widehat{\bm{u}}_{it}$, the empirical score $\widehat{\bm{s}}_{it}=\bm{X}_{it}^{\top}\widehat{\bm{u}}_{it}$. 
\item Generate the bootstrap score 
\begin{equation}
\bm{s}_{it}^{*b}=\widehat{\bm{a}}_{i}\cdot\eta_{i}^{*b}+\widehat{\bm{d}}_{t}\cdot\eta_{t}^{*b}+\widehat{\bm{w}}_{it}\cdot\eta_{i}^{*b}\eta_{t}^{*b}+\bar{\bm{s}}_{NT}.\label{eq: bootstrap score}
\end{equation}
Here, $\widehat{\bm{a}}_{i}$, $\widehat{\bm{d}}_{t}$, $\widehat{\bm{w}}_{it}$,
$\bar{\bm{s}}_{NT}$ are based on projections of $\widehat{\bm{s}}_{it}$
and serve as sample analogs of $\bm{a}_{i}$, $\bm{d}_{t}$, $\bm{w}_{it}$,
$E\left(\bm{X}_{it}^{\top}\bm{u}_{it}\right)$, respectively. These
components are discussed in further detail below. $b$ is an index
for the bootstrap number. $\eta_{i}^{*b}$ and $\eta_{t}^{*b}$ are bootstrap multipliers discussed below.
\item Compute the bootstrap statistics $\widehat{\bm{\beta}}^{*b}-\widehat{\bm{\beta}}=\left(\bm{X}^{\top}\bm{X}\right)^{-1}\sum_{i}\sum_{t}\bm{s}_{it}^{*b}$. 
\item Repeat Step 2 to Step 3 for $B$ times. The confidence interval for
$\bm\varrho^\top(\widehat{\bm{\beta}}-\bm{\beta}_{0})$ is then constructed by the
empirical distribution of $\left\{  \bm\varrho^\top(\widehat{\bm{\beta}}^{*b}-\widehat{\bm{\beta}})\right\} _{b}$
under the null hypothesis $\mathcal{H}_{0}:\bm\varrho^\top \bm{\beta}=\bm\varrho^\top\bm{\beta}_{0}$, where $\bm\varrho$ is a known unit vector.
\end{enumerate}

The spatially correlated bootstrap multipliers are generated as
\[
{\bm{\eta}}_{N}^{*b}
=
({\eta}_{i}^{*b})_{i=1}^{N}
=
\mathbb{K}_{N}^{1/2}\widetilde{\bm{\eta}}_{N}^{*b},\quad\widetilde{\bm{\eta}}_{N}^{*b}=(\widetilde{\eta}_{i}^{*b})_{i=1}^{N},
\]
where \(\{\widetilde{\eta}_{i}^{*b}\}_{i=1}^{N}\) are i.i.d. Rademacher random variables.
Given a bandwidth \(\mathfrak d_N>0\), let
\(
\mathbb{K}_{N}
=
\bigl[
\mathcal K(\mathfrak d_{ij}/\mathfrak d_N)
\bigr]_{i,j=1}^{N},
\)
where \(\mathcal K(\cdot)\) is a bounded kernel on \(\mathbb R\). We use the
Wendland \(C^2\) kernel,
\(
\mathcal K(u)=(1-|u|)_{+}^{4}(4|u|+1),
\)
which is compactly supported and ensures positive semi-definite. The compact
support is convenient for the large-sample theory, while positive
semi-definiteness ensures that \(\mathbb K_N^{1/2}\) is well defined.
Hence, by construction, the covariance structure of the spatially
correlated bootstrap multipliers is induced by \(\mathbb K_N\); in
particular,
\(
{Cov}^{*}
\bigl(
{\eta}_{i}^{*b},
{\eta}_{j}^{*b}
\bigr)
=
\mathcal K(\mathfrak d_{ij}/\mathfrak d_N).
\) 
Thus, the proposed bootstrap procedure is designed to preserve the
spatial covariance structure of the original data. In the simulations, we set \(\mathfrak d_N\asymp\lfloor N^{1/8}\rfloor\), which
performs well across the designs considered.\footnote{Kim and Sun~\citeyearpar{kim2011spatial} propose an optimal bandwidth
choice for spatially dependent data by modeling spatial dependence as a
linear process. Their bandwidth rule, however, is not directly applicable
under the general mixing framework maintained in this paper. Deriving an
optimal bandwidth choice for the present setting remains an open question
and is left for future research.} This construction is
closely related to the spatially dependent wild bootstrap of
Conley et al.\ (\citeyear{conley2023bootstrap}). 

The serially correlated bootstrap multipliers are generated by a simple
two-state Markov construction. Starting from
\(\eta_{0}^{*b}\), drawn from the Rademacher distribution, define, for
\(t\geq 0\),
\[
\eta_{t+1}^{*b}
=
\begin{cases}
\eta_{t}^{*b}, & \text{with probability } (1+q)/2,\\
-\eta_{t}^{*b}, & \text{with probability } (1-q)/2.
\end{cases}
\]
This construction yields serially dependent Rademacher multipliers, with
\({E}^{*}(\eta_t^{*b})=0\) and
\({Cov}^{*}(\eta_t^{*b},\eta_{t-h}^{*b})=q^{h}\) for
\(h\geq 0\).\footnote{The dependent multipliers $\eta_t^{*b}$ can be applied directly to the wild bootstrap in conventional time series settings; see also the dependent wild bootstrap (DWB) of Shao (\citeyear{shao2010dependent}) and Hounyo (\citeyear{hounyo2023wild}).}
The parameter $q$ quantifies serial dependence. The correlation is
$corr^{*}\left(\eta_{t}^{*b},\eta_{t+\iota}^{*b}\right)=q^{\iota}$,
for each $\iota\geq0$ and each $t\geq0$, which is a form of the
Laplacian kernel $k\left(\iota\right)=\exp\left(-\iota/S_{T}\right)$,
with $S_{T}=\left(\ln q\right)^{-1}$.\footnote{The multipliers $\{\eta_{t}^{*b}\}_{t}$ impose geometric decay, much
stronger than the general $\alpha$-mixing class assumed for the time
process. However, the goal here is only to approximate the long-run
autocovariance structure, and this specification is well-suited for that purpose.} The rule of thumb of selecting $q$ follows the plug-in bandwidth
guideline proposed by Andrews (\citeyear{andrews1991heteroskedasticity})
with kernel $k\left(\iota\right)$: 
\[
\widehat{q}=\exp\left(-\omega^{-1/3}T^{-1/3}\right),
\]
where $\omega$ represents a measure of autocorrelation, defined as
follows: for $k=1,\ldots,K$, let $\widehat{s}_{t,k}$ denote the
$k$-th element of $\widehat{\bm{s}}_{t}$, and let $\widehat{\rho}_{k}$
be the coefficient obtained by regressing $\widehat{s}_{t,k}$ on
$\widehat{s}_{t-1,k}$. Then, $\omega$ is given by: $\omega=\sum_{k=1}^{K}\frac{\widehat{\rho}_{k}^{2}}{(1-\widehat{\rho}_{k})^{4}}\Biggl/\sum_{k=1}^{K}\frac{(1-\widehat{\rho}_{k}^{2})^{2}}{(1-\widehat{\rho}_{k})^{4}}.$

Given the estimated score $\left\{ \widehat{\bm{s}}_{it}\right\} _{i,t}$,
we now construct an empirical analog of the decomposition. Specifically,
we project the empirical score onto different dimensions: 
\begin{equation}
\widehat{\bm{s}}_{it}=\ddot{\bm{a}}_{i}+\ddot{\bm{d}}_{t}+\ddot{\bm{w}}_{it}+\bar{\bm{s}}_{NT},\label{eq: decompose sit}
\end{equation}
where 
\begin{align*}
\ddot{\bm{a}}_{i}= & \frac{1}{T}\sum_{t=1}^{T}\widehat{\bm{s}}_{it}-\frac{1}{NT}\sum_{i=1}^{N}\sum_{t=1}^{T}\widehat{\bm{s}}_{it}\equiv\bar{\bm{s}}_{iT}-\bar{\bm{s}}_{NT},\\
\ddot{\bm{d}}_{t}= & \frac{1}{N}\sum_{i=1}^{N}\widehat{\bm{s}}_{it}-\frac{1}{NT}\sum_{i=1}^{N}\sum_{t=1}^{T}\widehat{\bm{s}}_{it}\equiv\bar{\bm{s}}_{Nt}-\bar{\bm{s}}_{NT},\\
\ddot{\bm{w}}_{it}= & \widehat{\bm{s}}_{it}-\ddot{\bm{a}}_{i}-\ddot{\bm{d}}_{t}-\bar{\bm{s}}_{NT}=\widehat{\bm{s}}_{it}-\bar{\bm{s}}_{iT}-\bar{\bm{s}}_{Nt}+\bar{\bm{s}}_{NT}.
\end{align*}

For the oracle PWB, $\widehat{\bm{a}}_{i}$, $\widehat{\bm{d}}_{t}$,
and $\widehat{\bm{w}}_{it}$ in $(\ref{eq: bootstrap score})$ depend
on the true variances in two dimensions $\bm{\sigma}_{a,f}^{2}$ and
$\bm{\sigma}_{d,f}^{2}$: 
\begin{equation}
\widehat{\bm{a}}_{i}=\bm{\vartheta}_{a}\ddot{\bm{a}}_{i},\quad\widehat{\bm{d}}_{t}=\bm{\vartheta}_{d}\ddot{\bm{d}}_{t},\quad\text{and}\quad\widehat{\bm{w}}_{it}=\ddot{\bm{w}}_{it},\label{eq: Oracle PWB}
\end{equation}
where the scaling terms $\bm{\vartheta}_{a}$ and $\bm{\vartheta}_{d}$
are define as follows: 
\begin{equation}
\bm{\vartheta}_{a}=\bm{\sigma}_{a,f}\left(\frac{1}{N}\sum_{i=1}^{N}\sum_{j=1}^{N}\mathcal{K}\!\left(\frac{\mathfrak{d}_{ij}}{\mathfrak{d}_{N}}\right)\ddot{\bm{a}}_{i}\ddot{\bm{a}}_{j}^{\top}\right)^{-1/2}\label{eq: PWB oracle scaling}
\end{equation}
and 
\begin{equation}
\bm{\vartheta}_{d}=\bm{\sigma}_{d,f}\left(\dfrac{1}{T}\sum\limits _{t=1}^{T}\sum_{\tau=1}^{T}q^{\left|t-\tau\right|}\ddot{\bm{d}}_{t}\ddot{\bm{d}}_{\tau}^{\top}\right)^{-1/2}.\label{eq:PWB oracle scaling 2}
\end{equation}
Here, each scaling term applies a whitening transformation (i.e., the negative square root matrix) with the
true variance, ensuring that the normalized statistic reproduces the
corresponding component in the limiting distribution. This whitening
is needed here because $\ddot{\bm{a}}_{i}$ is distorted by the contribution
of $\bm{v}_{it}$, which may not be separable.

\begin{assumption} \label{as: moment and variance} For some $\delta>0$
and $\zeta>1$, (i) $E(\bm{X}_{it}^{\top}\bm{u}_{it})=0$,
$\bm{Q}\equiv E(\bm{X}_{it}^{\top}\bm{X}_{it})>0$, $E(\left\Vert \bm{X}_{it}\right\Vert ^{8(\zeta+\delta)})\le C_{1}<\infty$,
$E(\left\Vert \bm{u}_{it}\right\Vert ^{8(\zeta+\delta)})\le C_{2}<\infty$,
and uniformly over $(N,T)$ such that the corresponding (possibly
$(N,T)$-dependent) variance satisfies $\sigma_{\bullet,f}>0$, the
$8(\zeta+\delta)$th moments of $\sigma_{a,f}^{-1}a_{ik}$, $\sigma_{d,f}^{-1}d_{tk}$,
$\sigma_{e,f}^{-1}e_{it,k}$, and $\sigma_{v,f}^{-1}v_{it,k}$ are
bounded, for each $k$. (ii) $\lambda_{\min}(\lim_{N,T\to\infty}(T\bm{\sigma}_{a,f}^{2}+N\bm{\sigma}_{d,f}^{2}+\bm{\sigma}_{v,f}^{2}+\bm{\sigma}_{e,f}^{2}))>0$. 
 \end{assumption}
 
Assumptions~\ref{as: moment and variance}~(i) and (ii) impose standard
moment conditions and require the limiting smallest eigenvalue of the
variance of the sum of scores to be bounded away from zero.

\begin{assumption} \label{as: time mixing} For the same
$\zeta$ and $\delta$ as in Assumption \ref{as: moment and variance}, 
(i) $\bm{\xi}_{t}$ is a $\alpha$-mixing sequence
with a mixing coefficient $\alpha(s)$ such that $\alpha(s)=O(s^{-\lambda})$
for a $\lambda>2\zeta/(\zeta-1)$. 
(ii) $q\to1$ as $T\to\infty$, and $(-\ln q)^{-1}=o(T^{1/2})$.  \end{assumption}

Assumption~\ref{as: time mixing} (i) imposes a standard mixing
condition from the time-series literature.\footnote{%
It weakens the
dependence restriction in Chiang et al.\ \citeyearpar{chiang2023standard}
by requiring an \(\alpha\)-mixing condition rather than a
\(\beta\)-mixing condition. Chiang et al. \citeyearpar{chiang2023standard} project the sum of products of scores 
\(\sum_{t,t'}\bm{s}_{it}\bm{s}_{it'}^{\top}\) onto the second dimension, rather than decomposing \(\bm{s}_{it}\) and studying the resulting
product terms. This approach requires a more delicate treatment of
fourth-order summations.}  Assumption~\ref{as: time mixing}~(ii) imposes
a condition on the kernel function $q^{\iota}$. Specifically, under
this assumption, the kernel tends to 1 as $T\to\infty$, and tends
to 0 when $\iota\geq T^{1/2}$ as $T\to\infty$, holding all other
factors constant. 

 Let $\mathcal{F}_{A}\equiv\sigma(\{\bm{\alpha}_{i}:i\in A\})$ for
$A\subset\{1,\dots,N\}$ and define $dist(A,B)\equiv\min\{\mathfrak{d}_{ij}:i\in A,\,j\in B\}$
and the strong mixing coefficient 
\[
\alpha_{d_{1},d_{2}}(r)\equiv\sup\Big\{\big|P(G\cap H)-P(G)P(H)\big|:\;G\in\mathcal{F}_{A},\ H\in\mathcal{F}_{B},|A|\leq d_{1},|B|\leq d_{2},\ dist(A,B)\ge r\Big\}.
\]
\begin{assumption}\label{as:spatial} For the same
$\zeta$ and $\delta$ as in Assumption \ref{as: moment and variance},  (i) The sampling region expands
in two non-opposing directions as $N\to\infty$. (ii) 
$\alpha_{\infty,\infty}(r)^{1-\frac{1}{2(\zeta+\delta)-1}}=o(r^{-4})$.
(iii) $\mathfrak{d}_{N}\to\infty$, as $N\to\infty$, and $\mathfrak{d}_{N}=o(N^{1/6})$.
(iv) If only $\widetilde{\mathfrak{d}}_{ij}={\mathfrak{d}}_{ij}+\varsigma_{ij}$ is observed, assume
$\sup_{i,j}|\varsigma_{ij}|\le C_{\varsigma}<\infty$ a.s., $\{\varsigma_{ij}\}$
are independent of $\{\bm{\alpha}_{i}\}$, $\{\bm{\xi}_{t}\}$, and
$\{\bm{\varepsilon}_{it}\}$. $\big[\mathcal{K}(\widetilde{\mathfrak{d}}_{ij}/\mathfrak{d}_{N})\big]_{i,j=1}^{N}$
is symmetric and positive semi-definite. \end{assumption}

Assumption~\ref{as:spatial} collects standard high-level regularity
conditions used in the spatial dependence and spatial HAC literature,
see, e.g., Conley \citeyearpar{conley1999gmm}. Part (i) is an increasing-domain
condition. Part (ii) imposes a sufficiently fast decay of strong mixing
to deliver a CLT for the sample mean and to control the variance of
the HAC estimator. Part~(iii) restricts the growth rate of the bandwidth
$\mathfrak{d}_{N}$ on the regular lattice $\mathcal{H}$.
In particular, it is chosen so that the maximal neighborhood size
is $O(\mathfrak{d}_{N}^{2})=o(N^{1/3})$.
Part (iv) allows the use of noisy distances, with uniformly bounded
measurement error independent of the latent spatial component.

\begin{assumption} \label{as: same rate} $\lambda_{\max}(\bm \sigma_{NT,f}^2)/\lambda_{\min}(\bm \sigma_{NT,f}^2)=O(1)$. \end{assumption}
Assumption~\ref{as: same rate} requires the aggregate variances of different
score components to be of the same order. Importantly, Assumption~\ref{as: same rate} does not impose coordinatewise
homogeneity within each componentwise variance term. For instance, the
diagonal entries of \(\bm\sigma_{e,f}^2\) may have heterogeneous orders
across coordinates, and the same is allowed for
\(\bm\sigma_{a,f}^2\), \(\bm\sigma_{d,f}^2\), and
\(\bm\sigma_{v,f}^2\). This allows different components of the score vector to fall into
different asymptotic regimes. Related restrictions have also been imposed in the recent two-way clustering literature; see, for example, Assumption 5 in Davezies,
D'Haultf{œ}uille, and Guyonvarch~\citeyearpar{davezies2025analytic}. In particular, their innovative Example 2 illustrates that, in the absence of such a condition, a standard least-squares approximation may fail under two-way clustered dependence. Proposition~\ref{prop: impossibility heter score} below extends this insight by showing that, without a
restriction of this type, the difficulty is not specific to least-squares
approximations but applies to all data-dependent procedures.

We denote the distribution of the original statistic and bootstrap
statistic using $P_{NT,f}$ and $P_{NT,f}^{*}$, respectively. $\left\Vert \cdot\right\Vert _{\infty}$
is the Kolmogorov metric. 

\begin{proposition}[Impossibility due to heterogeneous scores]
\label{prop: impossibility heter score}
Suppose that $\bm{\beta}=\bm{\beta}_0$. Let \(\mathcal D\) denote the collection of all measurable maps of
the observed data \(\{(\bm y_{it},\bm X_{it})\}_{i,t}\) into distribution
functions. Let \(\mathcal B_0\) be a class of DGPs \(f\) satisfying
Assumptions~\ref{as: AHS representation}--\ref{as:spatial}
and for all coordinate \(k\)'s, condition~\eqref{eq: converge non-gaussian} does not hold.

Assume that, for each \(f\in\mathcal B_0\) and $k$, there exists a deterministic
normalizing sequence \(a_{NTk,f}>0\) such that
\(
a_{NTk,f}\bigl(\widehat\beta_{k}-\beta_{k,0}\bigr)
\overset{d}{\to} L_{kf},
\)
where \(L_{kf}\) has a nondegenerate distribution function \(G_{kf}\). Then there
exists \(\varepsilon>0\) and $k$ such that
\[
\liminf_{N,T\to\infty}
\inf_{\widehat D\in\mathcal D}
\sup_{f\in\mathcal B_0}
P_{NT,f}\!\left(
\bigl\|
\widehat D\bigl(\{(\bm y_{it}^{(f)},\bm X_{it}^{(f)})\}_{i,t}\bigr)-G_{kf}
\bigr\|_{\infty}
>\varepsilon
\right)
>0.
\]
\end{proposition}

Proposition~\ref{prop: impossibility heter score} shows that, even after
excluding the infeasible (I\&N) regime, no feasible
data-dependent procedure can uniformly estimate the limiting law of
\(\widehat\beta_k\) when the aggregate score components are allowed to be
heterogeneous (without Assumption \ref{as: same rate}). In particular, uniform estimation may fail when different
coordinates exhibit different stochastic orders, for example when some
coordinates fall in regime (D) while others fall in the remaining regimes.   The difficulty arises because the heterogeneous
coordinates of the score vector are mixed through
\((\bm X^\top \bm X)^{-1}\). This reflects
a fundamental limitation of inference with heterogeneous scores under
two-way clustering.

With Assumption \ref{as: same rate}, each coordinate of $\widehat{\bm{\beta}}-\bm{\beta}_{0}$ converges at the same rate. We define the (infeasible) rate of convergence 
for $\widehat{\bm{\beta}}-\bm{\beta}_{0}$, $r_{NT,f}=\min\{\sqrt{N}\sigma_{a1,f}^{-1},\sqrt{T}\sigma_{d1,f}^{-1},\sqrt{NT}\}$.

\begin{theorem}\label{thm: main}
For the oracle PWB, under the null hypothesis
$\mathcal{H}_{0}:\bm\varrho^\top\bm{\beta}=\bm\varrho^\top\bm{\beta}_{0}$,
\begin{equation}
\Bigl\|
P_{NT,f}^{*}\!\left(
r_{NT,f}\bm\varrho^\top\bigl(\widehat{\bm{\beta}}^{*}-\widehat{\bm{\beta}}\bigr)
\right)
-
P_{NT,f}\!\left(
r_{NT,f}\bm\varrho^\top\bigl(\widehat{\bm{\beta}}-\bm{\beta}_{0}\bigr)
\right)
\Bigr\|_{\infty}
\xrightarrow{P} 0
\label{eq:theorem 1}
\end{equation}
holds uniformly over the entire function space $\mathcal{B}$ satisfying
Assumptions~\ref{as: AHS representation}--\ref{as: same rate}.
\end{theorem} Theorem~\ref{thm: main} shows the uniform consistency
result for the oracle PWB over the entire function space $\mathcal{B}$.\footnote{It is possible that \(\widehat{\bm{\beta}}-\bm{\beta}_{0}\) converges at rate
\(O_P(N^{-1/2})\), while a particular linear combination
\(\bm{\varrho}^{\top}(\widehat{\bm{\beta}}-\bm{\beta}_{0})\) converges at the faster
rate \(O_P((NT)^{-1/2})\). This can occur in extreme cases where two limiting
components are very close but not identical. In such cases, the normalization
\(r_{NT,f}\) remains of order \(\sqrt{N}\), so both original and bootstrap distributions converge to the same degenerate
limit. Consequently, the present theory does not directly cover
inference on hypotheses such as \(\beta_1>\beta_2\) in such extreme cases.}
However, this procedure is generally infeasible in practice, as it
requires knowledge of the DGP, which is typically unspecified. Interestingly,
it implies that the sole obstacle to achieving uniform consistency
is the lack of knowledge of the true variances in two dimensions $\bm{\sigma}_{a,f}^{2}$
and $\bm{\sigma}_{d,f}^{2}$.

\subsection{Feasible PWBs with Estimated Variance}

In practice, the true values of $\bm{\sigma}_{a,f}^{2}$ and $\bm{\sigma}_{d,f}^{2}$
are generally unknown to the researchers, and hence we propose the
feasible PWB method. Define the variance estimator of $\bm{\sigma}_{a,f}^{2}$
and $\bm{\sigma}_{d,f}^{2}$: 
\begin{align}
\widehat{\bm{\sigma}}_{a}^{2}=EVC\left(\frac{1}{N}\sum_{i=1}^{N}\sum_{j=1}^{N}\mathcal{K}\!\left(\frac{\mathfrak{d}_{ij}}{\mathfrak{d}_{N}}\right)\ddot{\bm{a}}_{i}\ddot{\bm{a}}_{j}^{\top}-\frac{1}{NT^{2}}\sum_{i=1}^{N}\sum_{j=1}^{N}\sum_{t=1}^{T}\mathcal{K}\!\left(\frac{\mathfrak{d}_{ij}}{\mathfrak{d}_{N}}\right)\ddot{\bm{w}}_{it}\ddot{\bm{w}}_{jt}^{\top}\right)\label{eq: sigma a estimator}
\end{align}
and 
\begin{align}
\widehat{\bm{\sigma}}_{d}^{2}= & EVC\Biggl(\frac{1}{T}\sum_{t=1}^{T}\sum_{\tau=1}^{T}q^{\left|t-\tau\right|}\ddot{\bm{d}}_{t}\ddot{\bm{d}}_{\tau}^{\top}-\frac{1}{N^{2}T}\sum_{i=1}^{N}\sum_{t=1}^{T}\sum_{\tau=1}^{T}q^{\left|t-\tau\right|}\ddot{\bm{w}}_{it}^{\top}\ddot{\bm{w}}_{it}\Biggl),\label{eq: sigma d estimator}
\end{align}
where $EVC\left(\cdot\right)$ replaces any negative eigenvalue by
zero to ensure positive semi-definiteness.

Lemma~\ref{lemma:limit form-1} demonstrates that the magnitude
of $\widehat{\bm{\sigma}}_{a}^{2}$ can be informative about the true
order of $\bm{\sigma}_{a,f}^{2}$, despite some ambiguity arising
from certain alternative cases. Based on this insight, we define 
\begin{align}
\widehat{\bm{a}}_{i}= & \widehat{\bm{\vartheta}}_{a}\ddot{\bm{a}}_{i},\quad\widehat{\bm{d}}_{t}=\widehat{\bm{\vartheta}}_{d}\ddot{\bm{d}}_{t},\quad\text{and}\quad\widehat{\bm{w}}_{it}=\ddot{\bm{w}}_{it},\label{eq: PWB}
\end{align}
where 
\begin{align}
\widehat{\bm{\vartheta}}_{a} & =\bm{D}_{a}\left(\widehat{\bm{\mu}}_{a}\right)\cdot\widehat{\bm{\sigma}}_{a}\left(\frac{1}{N}\sum_{i=1}^{N}\sum_{j=1}^{N}\mathcal{K}\!\left(\frac{\mathfrak{d}_{ij}}{\mathfrak{d}_{N}}\right)\ddot{\bm{a}}_{i}\ddot{\bm{a}}_{j}^{\top}\right)^{-1/2},\label{eq:PWB DV scaling}\\
\widehat{\bm{\vartheta}}_{d} & =\bm{D}_{d}\left(\widehat{\bm{\mu}}_{d}\right)\cdot\widehat{\bm{\sigma}}_{d}\left(\frac{1}{T}\sum_{t=1}^{T}\sum_{\tau=1}^{T}q^{\left|t-\tau\right|}\ddot{\bm{d}}_{t}^{\top}\ddot{\bm{d}}_{\tau}\right)^{-1/2}.\label{eq:PWB DV scaling 2}
\end{align}
Here, $\bm{D}_{a}(\widehat{\bm{\mu}}_{a})$ denotes the $K\times K$
diagonal matrix with 
\[
\bigl[\bm{D}_{a}(\widehat{\bm{\mu}}_{a})\bigr]_{kk}\;=\;\mathbb{I}\!\left\{ \widehat{\sigma}_{ak}^{2}\ge\widehat{\mu}_{ak}\right\} ,\qquad k=1,\ldots,K.
\]
Define $\bm{D}_{d}(\widehat{\bm{\mu}}_{d})$ analogously by $\bigl[\bm{D}_{d}(\widehat{\bm{\mu}}_{d})\bigr]_{kk}=\mathbb{I}\!\left\{ \widehat{\sigma}_{dk}^{2}\ge\widehat{\mu}_{dk}\right\} $.
Relative to the oracle scalings \(\bm{\nu}_{a}\) and \(\bm{\nu}_{d}\) in
\eqref{eq: PWB oracle scaling} and \eqref{eq:PWB oracle scaling 2}, their
feasible counterparts \(\widehat{\bm{\nu}}_{a}\) and
\(\widehat{\bm{\nu}}_{d}\) also implement a whitening transformation, but replace
the unknown population variances with their sample estimates and incorporate
indicator matrices. These indicators serve as safeguards in regimes where the
variance estimators may be unreliable, preventing spurious rescaling of
components whose variances are poorly estimated.

We propose several variants of the PWB method, each based on different
choices for the tuning parameters $\widehat{\bm{\mu}}_{a}$ and $\widehat{\bm{\mu}}_{d}$.
These parameter choices are designed to help detect whether the variance
components $T\bm{\sigma}_{a,f}^{2}$ and $N\bm{\sigma}_{d,f}^{2}$
diverge, vanish, or remain bounded, which in turn affects the validity
of the bootstrap procedure.

\paragraph{PWB-D (Divergence-sensitive).}

The tuning parameters $\widehat{\bm{\mu}}_{a,D}$ and $\widehat{\bm{\mu}}_{d,D}$
are chosen to diagnose divergence of the scaled variance components
$T\bm{\sigma}_{a,f}^{2}$ and $N\bm{\sigma}_{d,f}^{2}$, respectively.
In particular, their role is to detect whether the data fall into
scenario~(D), i.e., whether condition~\eqref{eq: both diverge}
holds. In practice, we recommend setting 
\begin{equation}
\widehat{\bm{\mu}}_{a,D}=\mathbf{1}_K\cdot{\log T}/{T}\quad\text{and}\quad\widehat{\bm{\mu}}_{d,D}=\mathbf{1}_K\cdot{\log N}/{N}.\label{eq: tuning pwb-d}
\end{equation}

\paragraph{PWB-V (Vanishing-sensitive).}

In contrast, the tuning parameters $\widehat{\bm{\mu}}_{a,V}$ and
$\widehat{\bm{\mu}}_{d,V}$ are selected to discriminate whether the
scaled variances $T\bm{\sigma}_{a,f}^{2}$ and $N\bm{\sigma}_{d,f}^{2}$,
together with $\bm{\sigma}_{v,f}^{2}$, vanish; that is, whether the
(V\&G) condition~\eqref{eq: vanish gaussian} holds. In practice,
we recommend 
\begin{equation}
\widehat{\bm{\mu}}_{a,V}=\mathbf{1}_K\cdot1/(T\log T)\text{ and  }\widehat{\bm{\mu}}_{d,V}=\mathbf{1}_K\cdot1/(N\log N).\label{eq: tuning pwb-v}
\end{equation}

The logarithmic terms in the definitions of \(\widehat{\bm{\mu}}_{a}\)
and \(\widehat{\bm{\mu}}_{d}\) reflect the absence of a sharp finite-sample
boundary between the relevant regimes, a feature intrinsic to the problem
and illustrated further in Section~\ref{sec: DRC}. Such thresholding
rules necessarily generate a shrinking but nonempty indifference region,
within which classification may be unstable. For this reason, procedures
that rely directly on a hard classification may be affected near the
boundary. The specific choices of \(\widehat{\bm{\mu}}_{a}\) and
\(\widehat{\bm{\mu}}_{d}\), however, are not DGP-specific. They are used
only to guide regime classification, and the hybrid procedure recommended below is designed precisely to
avoid relying on a sharp boundary classification; hence these threshold
choices do not affect its uniform validity. The simulation evidence also suggests that its finite-sample performance
is not sensitive to the particular threshold choice, provided that the
thresholding rule can reliably identify the relevant regime.

Define the feasible rate of convergence 
\[
\widehat{r}_{NT,f}=\min\{\sqrt{N}\widehat{\sigma}_{a1}^{-1},\sqrt{T}\widehat{\sigma}_{d1}^{-1},\sqrt{NT}\}.
\]
\begin{theorem}\label{thm: main-2} Suppose Assumptions~\ref{as: AHS representation}--\ref{as: same rate}
hold. Under the null hypothesis $\mathcal{H}_{0}:\bm\varrho^\top\bm{\beta}=\bm\varrho^\top\bm{\beta}_{0}$,
and with ${{r}}_{NT,f}$ replaced by its feasible counterpart $\widehat{{r}}_{NT,f}$,
the convergence result in~\eqref{eq:theorem 1} holds uniformly in
the following cases:

\begin{enumerate}[label=(\alph*)] 
\item \textbf{(PWB-D).} For each $k$, either one of \eqref{eq: vanish nongaussian},
\eqref{eq: vanish gaussian} holds or 
\[
T\sigma_{ak,f}^{2}>2\log T\quad\text{or}\quad N\sigma_{dk,f}^{2}>2\log N.
\]
\item \textbf{(PWB-V).} For each $k$, one of \eqref{eq: both diverge},
\eqref{eq: vanish gaussian}, or \eqref{eq: converge gaussian} holds. 
\end{enumerate}
\end{theorem}

Theorem~\ref{thm: main-2}(a) implies that PWB-D is uniformly valid
when clustering vanishes (regimes V\&G and V\&N), and also under sufficiently
strong clustering in at least one dimension (a condition strictly
stronger than D). The precise boundary of such ``strictly'' strong
clustering region is determined by the chosen tuning thresholds $\widehat{\bm{\mu}}_{a,D}$ and
$\widehat{\bm{\mu}}_{d,D}$.

Theorem~\ref{thm: main-2}(b) shows that PWB-V is uniformly consistent
on (almost) all class of DGPs that yield a Gaussian limit. Importantly, the tuning thresholds do not impose
a similarly strict restriction as PWB-D: in Gaussian scenario,
when $T\bm{\sigma}_{a,f}^{2}$ (or $N\bm{\sigma}_{d,f}^{2}$) vanishes,
its estimator $T\widehat{\bm{\sigma}}_{a}^{2}$ (or $N\widehat{\bm{\sigma}}_{d}^{2}$)
vanishes as well, so any misclassification induced by the indicator
$\widehat{\bm{D}}_{a}(\widehat{\bm{\mu}}_{a,V})$ (or $\widehat{\bm{D}}_{d}(\widehat{\bm{\mu}}_{d,V})$)
is innocuous in \eqref{eq:PWB DV scaling} and \eqref{eq:PWB DV scaling 2}. In fact, this suggests that the indicator matrices,
and the associated tuning parameters $\widehat{\bm{\mu}}_{a,V}$ and
$\widehat{\bm{\mu}}_{d,V}$, are superfluous for the validity of PWB-V.
They are retained for subsequent use and notational clarity.

\paragraph{PWB-H (Hybrid).}

By Theorem~\ref{thm: main-2}, PWB-V is consistent whenever the limiting
distribution is Gaussian, while PWB-D is valid in some settings with
non-Gaussian limiting distributions. However, neither method works
in both of the following scenarios: (I\&G) and
(V\&N). This is because the variance estimator cannot provide information
to distinguish these two scenarios. These observations motivate a
hybrid procedure that attempts to adapt to both underlying structures
simultaneously based on a second factor.

We define a data-driven combination of PWB-D and PWB-V via the tuning
parameters 
\[
\widehat{\bm{\mu}}_{a,H}=\bm{D}^{*}\widehat{\bm{\mu}}_{a,D}+(\bm{I}_{K}-\bm{D}^{*})\widehat{\bm{\mu}}_{a,V}\quad\text{and}\quad\widehat{\bm{\mu}}_{d,H}=\bm{D}^{*}\widehat{\bm{\mu}}_{d,D}+(\bm{I}_{K}-\bm{D}^{*})\widehat{\bm{\mu}}_{d,V},
\]
where $\bm{D}^{*}$ is a diagonal matrix whose $k$-th diagonal element
is constructed based on a Kolmogorov Smirnov (KS) normality test applied
to bootstrap statistics: 
\[
D_{k}^{*}=\mathbb{I}\left\{ P_{{\rm KS}}\left(\widehat{t}_{k}^{*b}\right)_{b=1}^{B}<\kappa\right\} .
\]
Here, $P_{{\rm KS}}\left(\widehat{t}_{k}^{*b}\right)_{b=1}^{B}$ returns
the KS $p$-values computed on each group $\left(\widehat{t}_{k}^{*1},\ldots,\widehat{t}_{k}^{*B}\right)$.
 The standardized version bootstrap
statistic $\widehat{t}_{k}^{*b}=\sum_{i,t}s_{it,k}^{*b}\Biggl/\sqrt{\frac{1}{B-1}\sum_{b=1}^{B}\left(\sum_{i,t}s_{it,k}^{*b}\right)^{2}}$
is computed under the PWB-V procedure. As we show that PWB-V reproduces
the \emph{form} of the asymptotic law: the PWB-V bootstrap statistic
is asymptotically Gaussian if and only if the original statistic is
asymptotically Gaussian. Observe that the components of $\widehat{\mu}_{ak,H}$
and $\widehat{\mu}_{dk,H}$ may differ across $k$, allowing each component
to have its own limiting behavior. Under the non-Gaussian alternative, this KS $p$-value converges to
zero at an exponential rate, which motivates the choice $\kappa=1/B$. 
Since the number of bootstrap replications $B$ can be chosen large
and is not tied to the sample size, this threshold does not induce
the shrinking indifference region associated with the tuning parameters
$\widehat{{\mu}}_{\bullet}$. Proposition \ref{prop: two indicator factors} further shows that the KS diagnostic distinguishes the Gaussian and
non-Gaussian cases with probability approaching one, including transition regimes that allow for drifting parameters. Moreover, such pre-test classifier
does not induce the type of post-selection distortion discussed by Leeb and P{\"o}tscher
\citeyearpar{leeb2008can}. The
simulation results in Table~\ref{tab:average correct rate for DRC varying NT, ols}
are consistent with this theoretical finding.

Note that PWB-H combines the strengths of PWB-V and PWB-D. When $D_{k}^{*}$
indicates a Gaussian limit, PWB-H applies PWB-V, which covers (almost)
all Gaussian regimes, namely D, I\&G, and V\&G. When the rule indicates
a non-Gaussian limit, PWB-H switches to PWB-D, which covers the non-Gaussian
regime V\&N.

More importantly, although PWB-H uses logarithmic tuning thresholds,
it does \emph{not} inherit the boundary non-uniformity that can arise
for PWB-D in certain regimes. This is because PWB-H avoids the thresholds
that may be misleading: in the non-Gaussian region it relies on D-type
thresholds (for which $\widehat{\bm{\mu}}_{a,D}$ remains well behaved
even when $\widehat{\bm{\mu}}_{a,V}$ can be misleading), whereas
in the Gaussian region it adopts the V-type procedure, which is uniformly
valid even when $\widehat{\bm{\mu}}_{a,D}$ may be misleading.

Consequently, as formalized in Theorem~\ref{thm: main-3}, PWB-H
is uniformly asymptotically exact over all regimes except I\&N, i.e.,
when~\eqref{eq: converge non-gaussian} holds, where uniform consistency
is ruled out by Proposition~\ref{prop: impossibility}(a). This supports
PWB-H as a general-purpose inference procedure.

\begin{theorem}\label{thm: main-3} \textbf{(PWB-H).} Suppose Assumptions~\ref{as: AHS representation}--\ref{as: same rate}
hold. Under the null hypothesis $\mathcal{H}_{0}:\bm\varrho^\top\bm{\beta}=\bm\varrho^\top\bm{\beta}_{0}$,
and with ${{r}}_{NT,f}$ replaced by its feasible counterpart $\widehat{{r}}_{NT,f}$,
the convergence result in~\eqref{eq:theorem 1} holds uniformly for
PWB-H, except when condition~\eqref{eq: converge non-gaussian} holds
for some $k$. \end{theorem}

A natural next question is whether one can go one step further: when
uniform validity fails in the I\&N regime, can the procedure at least
be made conservative there. Proposition~\ref{prop: impossibility}(c)
indicates a fundamental tension: any method that is uniformly valid
in all feasible regimes \emph{cannot} be uniformly conservative in the
I\&N regime. Within the two-way clustering framework studied here, this trade-off helps
clarify the robustness of PWB-H: it delivers uniform validity over the feasible
regimes, while the remaining difficulty in I\&N reflects an intrinsic reflects a limitation that cannot be uniformly resolved without sacrificing validity elsewhere.

\subsection{Discussion}

\subsubsection{Dependence Regime Classification}

\label{sec: DRC}

Figure~\ref{fig: variance estimators} (visualizing Proposition~\ref{prop: two indicator factors})
summarizes the relationship between $T\widehat{\sigma}_{ak}^{2}$
and $T\sigma_{ak,f}^{2}$ across Gaussian and non-Gaussian scenarios.
When $D_{k}^{*}=0$ (Gaussian scenario), $T\widehat{\sigma}_{ak}^{2}$
matches the stochastic order of $T\sigma_{ak,f}^{2}$. When $D_{k}^{*}=1$
(non-Gaussian scenario), this order matching fails: under both non-Gaussian
regimes, $T\widehat{\sigma}_{ak}^{2}=O_{P}(1)$. Consequently, these
two non-Gaussian cases may overlap each other.

There exists no clear practical boundary between neighboring regimes, for
instance, between V\&G and I\&G, or between I\&G and D. Accordingly,
we use a $\log T$ or $\log N$ threshold to separate regimes, which
may create an indifference zone near the cutoff. We acknowledge this
limitation, but it is intrinsic: the data cannot, in general, cleanly
distinguish regimes that differ only in such local asymptotic behavior.

\begin{figure}[t!]
\centering \includegraphics[width=0.8\textwidth]{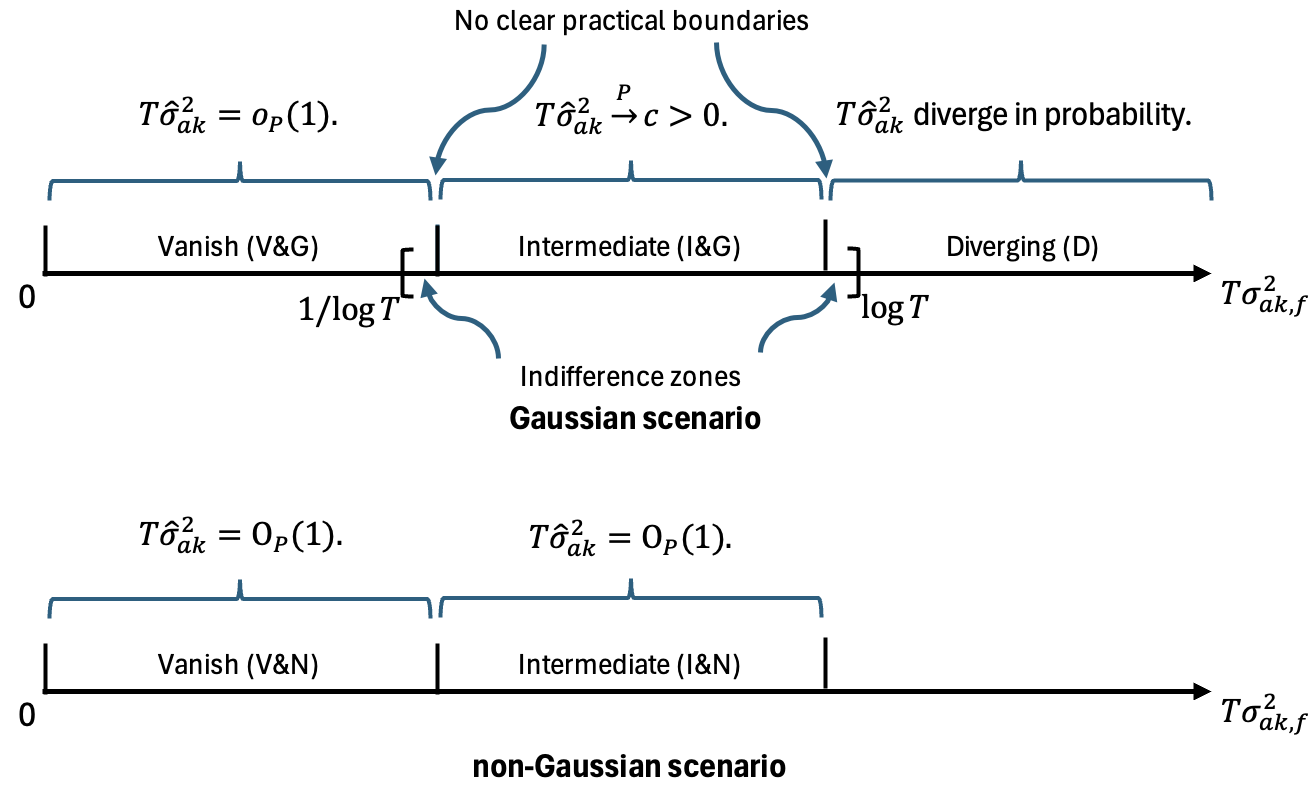} \caption{\textbf{Stochastic order of the variance estimator across true variance
regimes.}}
\label{fig: variance estimators} 
\end{figure}

Notice that relying solely on ``cluster-strength'' diagnostics (building
on the variance estimators) can be misleading, since Gaussian and
non-Gaussian regimes may yield identical diagnostics. Hence, we next
introduce a Dependence Regime Classifier (DRC) to distinguish different
feasible regimes. For each $k$, define the indicators 
\[
\widehat{D}_{D,k}=\max\!\left\{ \bigl[\bm{D}_{a}(\widehat{\bm{\mu}}_{a,D})\bigr]_{kk},\;\bigl[\bm{D}_{d}(\widehat{\bm{\mu}}_{d,D})\bigr]_{kk}\right\} ,\qquad\widehat{D}_{V,k}=\max\!\left\{ \bigl[\bm{D}_{a}(\widehat{\bm{\mu}}_{a,V})\bigr]_{kk},\;\bigl[\bm{D}_{d}(\widehat{\bm{\mu}}_{d,V})\bigr]_{kk}\right\} .
\]

\begin{description}
\item [{\textbf{Algorithm 2.}}] \textbf{Dependence Regime Classifier (DRC)}

\end{description}
\begin{enumerate}[label=\textbf{Step \arabic*:}, leftmargin=*, itemsep=1ex]
\item \textbf{Gaussian vs.\ non-Gaussian.} If $D_{k}^{*}=0$, treat the limiting distribution of $\frac{1}{NT}\sum_{i,t}{s}_{itk}$ as Gaussian
and proceed to Step~2. Otherwise, treat it as non-Gaussian and proceed
to Step~3.
\item \textbf{Gaussian branch.} If $\widehat{D}_{D,k}=1$, classify the
component as (D). Else if $\widehat{D}_{V,k}=0$, classify it as (V\&G).
Otherwise, classify it as lying in the \emph{Gaussian transition region}
between (V\&G) and (D) (i.e., one of V\&G, I\&G, or D). For simulation
reporting, we assign this case to (I\&G).
\item \textbf{Non-Gaussian branch.} Classify it as belonging to the \emph{non-Gaussian
region} (i.e., one of V\&N or I\&N). 
\end{enumerate}
Finally, note that the two non-Gaussian regimes cannot be distinguished
by DRC. One might hope to construct a sharper classifier, but Proposition~\ref{prop: impossibility}(b)
shows that, absent additional information on the DGP, no procedure
can uniformly distinguish these two regimes.

\subsubsection{Comparison of Existing Methods and Key Differences}

Table~\ref{tab: asymptotic property of methods} provides a summary
of the asymptotic properties of the limiting distributions and the
validity of various inference methods across different regimes. AdaWild
denotes the autoregressive double adaptive wild bootstrap introduced
by Juodis~(\citeyear{juodis2021shock}), CHS refers to the
variance estimator proposed by Chiang et al.~(\citeyear{chiang2023standard}),
and MWCB stands for the multiway cluster bootstrap developed by Hounyo
and Lin~(\citeyear{hounyo2024wild}).

Several inference methods have been proposed under the assumption
of no temporal dependence. For example, the variance estimator of
Cameron, Gelbach, and Miller~(\citeyear{cameron2011robust}), the
wild bootstrap procedures of MacKinnon et al. (\citeyear{mackinnon2021wild}),
and several bootstraps of Menzel~(\citeyear{menzel2021bootstrap})
are all designed without explicitly accounting for autocorrelation.
Nonetheless, their properties can be understood within the general
framework developed here. Particularly when $\{\bm{\xi}_{t}\}_{t}$
do not capture serial dependence, the behavior of Cameron et al.~(\citeyear{cameron2011robust})
estimator and MacKinnon et al.~(\citeyear{mackinnon2021wild}) wild
bootstrap methods are similar to the CHS variance estimator and MWCB,
respectively, and Menzel~(\citeyear{menzel2021bootstrap}) bootstrap
procedure with (without) model selection shares the same asymptotic
validity as our PWB-D (PWB-V) method.

The idea of decomposing the score into three components prior to bootstrapping
also appears in Menzel~\citeyearpar{menzel2021bootstrap} and Juodis~\citeyearpar{juodis2021shock}.
We summarize several key differences between these approaches and
our PWB framework (under no temporal dependence): 
\begin{enumerate}
\item Menzel~\citeyearpar{menzel2021bootstrap} employs i.i.d.\ resampling
with wild weights, whereas Juodis~\citeyearpar{juodis2021shock}
and PWB use wild weights to reproduce dependence. This further complicates
the bootstrap joint CLT because the resulting bootstrap components
are no longer independent.
\item Our data-dependent rescaling (introduced to maintain validity in I\&N)
and the thresholding indicators are conceptually related to the bootstrap
with model selection in Menzel~\citeyearpar{menzel2021bootstrap}.
In Juodis~\citeyearpar{juodis2021shock}, thresholding indicators
appear, but without the accompanying rescaling weights.
\item Unlike Menzel~\citeyearpar{menzel2021bootstrap} and Juodis~\citeyearpar{juodis2021shock},
we propose a hybrid bootstrap (PWB-H) that detects whether the
self-normalized statistic is asymptotically Gaussian. The guiding
principle is to deploy the divergence-sensitive and vanishing-sensitive
schemes precisely in the regimes where they are most reliable.
\item Consequently, relative to existing methods, the hybrid bootstrap offers
two main advantages: (i) it is valid on the union of the regimes covered
by the divergence- and vanishing-sensitive schemes, and (ii) it avoids
the boundary non-uniformity induced by the indifference region inherent
in threshold-based tuning. 
\end{enumerate}

\begin{table}
\renewcommand{\arraystretch}{1.5}
 \resizebox{\columnwidth}{!}{%
 {\Huge
\begin{tabular}{llcccccc}
\hline \hline 
\multicolumn{2}{l}{\multirow{2}{*}{\makecell{Limiting behavior of \\$\bm{\sigma}_{a,f}^{2}$, $\bm{\sigma}_{d,f}^{2}$, and $\bm{\sigma}_{v,f}^{2}$}}} & 
\multirow{2}{*}{\makecell{Divergent, $\left(\ref{eq: both diverge}\right)$}} & 
\multicolumn{2}{c}{Vanish}& &
\multicolumn{2}{c}{\makecell[c]{Intermediate}}\\
\cline{4-5}\cline{7-8}
&&& V\&N, \eqref{eq: vanish nongaussian} & 
V\&G, \eqref{eq: vanish gaussian} && 
I\&N, \eqref{eq: converge non-gaussian} & 
I\&G, \eqref{eq: converge gaussian} \\
\hline
\multirow{3}{*}{\makecell[l]{Asymptotic \\
Properties}} & {Limiting form of $\widehat{\bm{\beta}}$}                          & Gaussian & non-Gaussian & Gaussian && non-Gaussian & Gaussian \\
&{Order of $\widehat{\bm{\beta}}-\bm{\beta}$}                  & $O_P\left(\max\left\{\frac{\bm{\sigma}_{a,f}}{\sqrt{N}},\frac{\bm{\sigma}_{d,f}}{\sqrt{T}}\right\}\right)$ & $O_P\left(\frac{1}{\sqrt{NT}}\right)$ & $O_P\left(\frac{1}{\sqrt{NT}}\right)$ & &$O_P\left(\frac{1}{\sqrt{NT}}\right)$ & $O_P\left(\frac{1}{\sqrt{NT}}\right)$ \\
&{Order of $\max\{T\widehat{\bm{\sigma}}_{a}^{2},N\widehat{\bm{\sigma}}_{d}^{2}\}$}   & diverge &  $O_P\left(1\right)$&$o
_P\left(1\right)$    &   &$O_P\left(1\right)$&$O_P\left(1\right)$    \\
\hline
\multirow{7}{*}{\makecell[l]{Methods}} &{CHS CRVE} &$\checkmark$&&$\checkmark$&&&$\checkmark$\\ 
&
MWCB                             & $\checkmark$   &   & $\checkmark$ & &   & $\checkmark$ \\
&AdaWild                           & $\checkmark$   & $\checkmark$    & $\checkmark$  &&   &  \\
&Oracle PWB                         & $\checkmark$   & $\checkmark$    & $\checkmark$ & & $\checkmark$    & $\checkmark$ \\
&PWB-D                           & $\checkmark$   & $\checkmark$    & $\checkmark$ & &   &  \\
&PWB-V                               & $\checkmark$   &     & $\checkmark$ & &   & $\checkmark$ \\
&PWB-H                               & $\checkmark$   & $\checkmark$    & $\checkmark$ & &   & $\checkmark$ \\
\hline \hline 
\end{tabular}}}\caption{\textbf{Asymptotic properties of statistics and various methods.}  A checkmark is put when the method is consistent in the sub-scenario (ignore the indifference region).
}\label{tab: asymptotic property of methods}
\end{table}

\subsubsection{Other Possible Alternatives}

In this paper, the PWB methods are primarily built on the empirical
score $\widehat{\bm{s}}_{it}$ and employ a non-studentized statistic
designed to accommodate a broad range of settings. When perturbing
the score, we cannot directly compute the bootstrap empirical score
$\widehat{\bm{s}}_{it}^{*}$, which is required for forming a studentized
statistic. One possible workaround is 
\begin{equation}
\widehat{\bm{s}}_{it}^{*}=\bm{X}_{it}^{\top}\widehat{\bm{u}}_{it}^{*}=\bm{X}_{it}^{\top}\widehat{\bm{u}}_{it}-\bm{X}_{it}^{\top}\bm{X}_{it}(\widehat{\bm{\beta}}^{*}-\widehat{\bm{\beta}})=\widehat{\bm{s}}_{it}-\bm{X}_{it}^{\top}\bm{X}_{it}(\bm{X}^{\top}\bm{X})^{-1}\sum_{i}\sum_{t}\bm{s}_{it}^{*}.\label{eq: bootstrap empirical score}
\end{equation}
However, this approach is essentially perturbing residuals, as the
first equality in (\ref{eq: bootstrap empirical score}) fixes $\bm{X}_{it}$.

Another alternative is to perturb the score directly, such as $\widehat{\bm{s}}_{i}^{*}=\widehat{\bm{s}}_{i}\eta_{i}^{*}$.
But under Rademacher weights, the resulting bootstrap variance estimator
becomes identical to that of the original $t$-statistic, yielding
no gain relative to the non-studentized approach. We also experimented
with other weight distributions, but found no notable improvement.

Note that we multiply both $\widehat{\bm{a}}_{i}$ and $\widehat{\bm{w}}_{it}$
by the same $\eta_{i}^{*b}$, and both $\widehat{\bm{d}}_{t}$ and
$\widehat{\bm{w}}_{it}$ by the same $\eta_{t}^{*b}$. This design
enables the bootstrap to capture the correlation structures reflected
in $\left\{ Cov\left(\bm{Z}_{N}^{a},\bm{Z}_{N,l}^{\phi}\right)\right\} _{l}$
and $\left\{ Cov\left(\bm{Z}_{T}^{d},\bm{Z}_{T,l}^{\psi}\right)\right\} _{l}$,
which is necessary for the validity of oracle PWB over the entire
parameter function space. By contrast, a pure Efron (or pigeonhole)
bootstrap is ill-suited for reproducing these two underlying correlation
structures of the data, because such resampling tends to overinflate
the variance of $\bm{e}_{it}$.

\section{Simulation Results}

\label{sec:simulations} In this section, we examine the performance
of various methods and report the most relevant results. The additional
 results, including results for less favorable alternatives, heteroskedasticity framework, varying levels of dependence, nonseparable panel models, are provided in the Internet Appendix ID.\footnote{The nonseparable panel DGP is discussed in Fern{\'a}ndez-Val, Freeman, and Weidner \citeyearpar{fernandez2021low} and Chen, Fern{\'a}ndez-Val, Weidner \citeyearpar{chen2021nonlinear}.)} The overall pattern mirrors the findings reported in the main text.
The main goal of our simulation is to support the theoretical results.

We generate data based on the linear model: 
\begin{align}
 & y_{it}=\beta_{1}+\sum_{k=2}^{K}\beta_{k}X_{it,k}+u_{it},\label{eq: simulation DGP}\\
 & X_{it,k}=f_{1k}(\alpha_{i,k}^{x},\xi_{t,k}^{x},\varepsilon_{it,k}^{x}),\text{ and}\\
 & u_{it}=f_{2}(\alpha_{i}^{u},\xi_{t}^{u},\varepsilon_{it}^{u}).
\end{align}
We assume that $(\alpha_{i,k}^{x},\alpha_{i}^{u},\xi_{t,k}^{x},\xi_{t}^{u},\varepsilon_{it,k}^{x},\varepsilon_{it}^{u})$
are mutually independent random variables.  $(\varepsilon_{it,k}^{x},\varepsilon_{it}^{u})$ are standard Gaussian distributed, independent across $i$ and $t$, and $k$. The latent
components $\left(\xi_{t,k}^{x},\xi_{t}^{u}\right)$ are serially
dependent over $t$, following an AR(1) procedure: 
\begin{equation}
\xi_{t}=\rho\xi_{t-1}+\widetilde{\xi}_{t},\text{ where \ensuremath{\widetilde{\xi}_{t}} are independent draws from \ensuremath{\mathcal{N}(0,1-\rho^{2}).}}\label{eq: simulation DGP xi time}
\end{equation}
Such an AR(1) process with $\rho<1$ satisfies Assumption~\ref{as: time mixing}.

The latent
components $\left(\alpha_{i,k}^{x},\alpha_{i}^{u}\right)$ are generated by adopting the spatial design of Conley and Molinari~\citeyearpar{conley2007spatial},
which specifies a stationary, finite-range moving-average field with
geometrically decaying weights. Let $\{\mathfrak{s}_{i}\}_{i=1}^{N}\subset\mathbb{R}^{2}$
denote the (fixed) spatial locations and define the Euclidean distance
$\mathfrak{d}_{ij}\equiv\|\mathfrak{s}_{i}-\mathfrak{s}_{j}\|$. For
each component $k\in\{1,\dots,K\}$, draw i.i.d.\ Gaussian innovations
$\{z_{j,k}\}_{j=1}^{N}$ with $z_{j,k}\stackrel{\text{i.i.d.}}{\sim}\mathcal{N}(0,\sigma_{z}^{2})$,
independent across $j$ and $k$. Fix a range parameter $m>0$ and
a dependence parameter $\rho_{\mathfrak{d}}$, and define
the $m$-neighborhood $\mathcal{N}_{m}(i)\equiv\{j\le N:\mathfrak{d}_{ij}\le m\}$.
We then generate the spatial effect by the truncated, geometrically
weighted average 
\begin{align}
\alpha_{i,k}\equiv\sum_{j\in\mathcal{N}_{m}(i)}w_{ij}\,z_{j,k},\qquad w_{ij}\equiv\rho_{\mathfrak{d}}^{\,\mathfrak{d}_{ij}}\mathbf{1}\{\mathfrak{d}_{ij}\le m\}.\label{eq:conley_molinari_dgp}
\end{align}
Because the weights are compactly supported, $\alpha_{i,k}$ and $\alpha_{j,k}$
are independent whenever $\mathfrak{d}_{ij}>2m$, so dependence decays
with distance and vanishes beyond a finite cutoff, matching the short-range
dependence implied by our mixing assumptions.
In our baseline implementation, we set $m=5$, 
$\rho_{\mathfrak{d}}=0.10$, and $\sigma_{u}^{2}=1$. We further discuss the effect of varying levels of parameters in the Internet Appendix ID.

We set $\beta_{k}=1$ for all $k=1,\ldots,K$. For the number of regressors
($K$), MacKinnon (\citeyear{mackinnon2021fast}) suggests that performances
of many methods deteriorate with increasing $K$. Choosing a small
value of $K$, such as $K=2$ (a constant term and one regressor),
may yield an overly optimistic assessment. We hence choose $K=5$
and examine the true value of $\beta_{5}$ with various methods.

For functions $f_{1k}$ and $f_{2}$, we consider simulation designs
that cover all five scenarios: 
\begin{flalign}
 & \bullet\ \text{D, (\ref{eq: both diverge}) holds:} & {x_{it,k}=\alpha_{i,k}^{x}+\xi_{t,k}^{x}+\varepsilon_{it,k}^{x}\text{ and }u_{it}=\alpha_{i}^{u}+\xi_{t}^{u}+\varepsilon_{it}^{u},}\label{eq: dgp 1}\\
 & \bullet\ \text{V\&N, (\ref{eq: vanish nongaussian}) holds:} & x_{it,k}=\alpha_{i,k}^{x}\xi_{t,k}^{x}\text{ and }u_{it}=\alpha_{i}^{u}\xi_{t}^{u},\label{eq: dgp 2}\\
 & \bullet\ \text{V\&G, (\ref{eq: vanish gaussian}) holds:} & x_{it,k}=\varepsilon_{it,k}^{x}\text{ and }u_{it}=\varepsilon_{it}^{u},\label{eq: dgp 3}\\
 & \bullet\ \text{I\&N, (\ref{eq: converge non-gaussian}) holds:} & {x_{it,k}=\left(\alpha_{i,k}^{x}+N^{-1/4}\right)\xi_{t,k}^{x}\text{ and }u_{it}=\left(\alpha_{i}^{u}+N^{-1/4}\right)\xi_{t}^{u},}\label{eq: dgp 4}\\
 & \bullet\ \text{I\&G, (\ref{eq: converge gaussian}) holds:} & {x_{it,k}=\left(\varepsilon_{it,k}^{x}+N^{-1/4}\right)\xi_{t,k}^{x}\text{ and }u_{it}=\left(\varepsilon_{it}^{u}+N^{-1/4}\right)\xi_{t}^{u}.}
\label{eq: dgp 5}
\end{flalign}

\begin{figure}[t!]
\centering \begin{subfigure}[t]{0.49\hsize} \subcaption{DGP
(\ref{eq: dgp 1})}\includegraphics[width=1\textwidth]{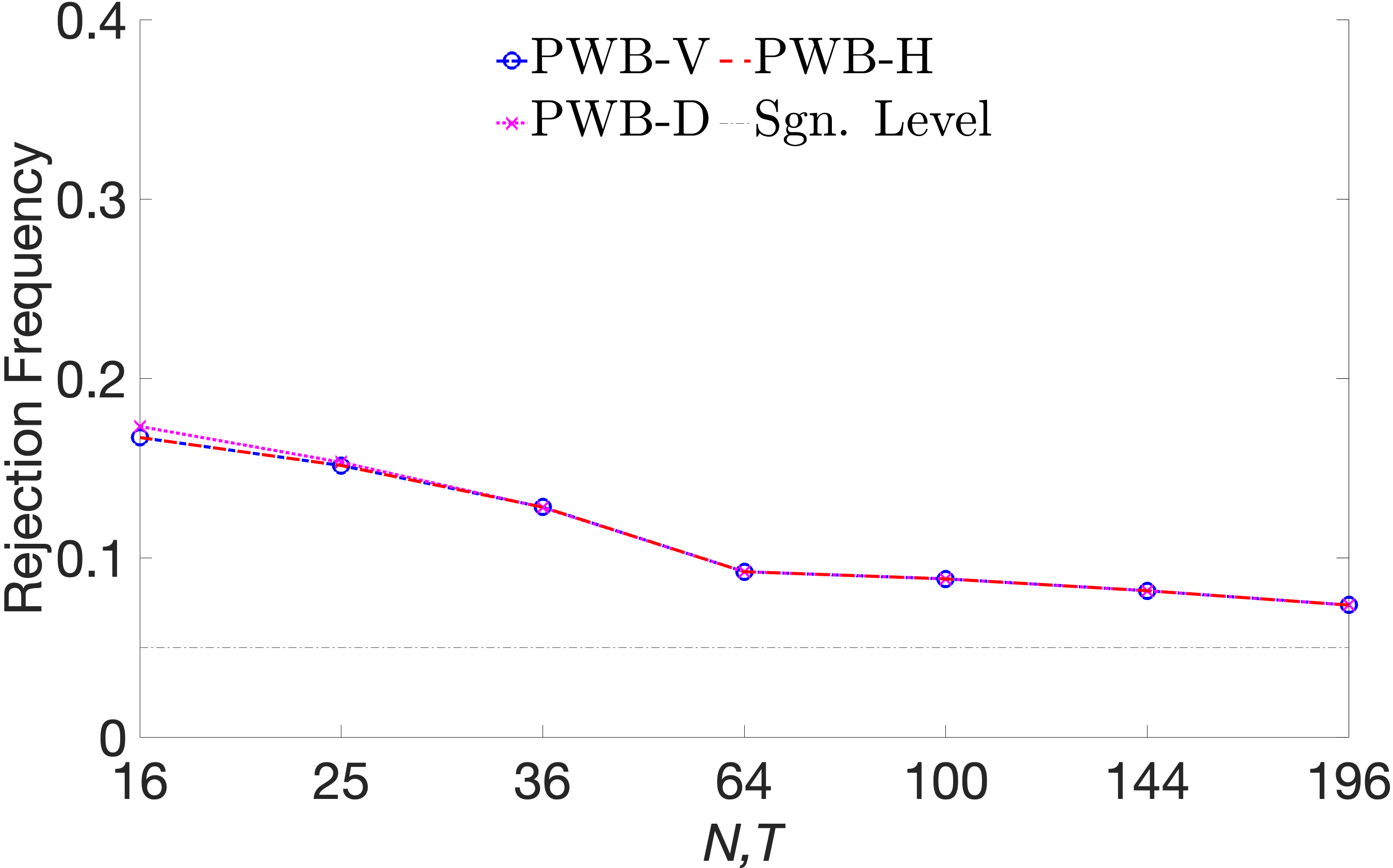}
\end{subfigure}

\begin{subfigure}[t]{0.49\hsize} \subcaption{DGP
(\ref{eq: dgp 2})}\includegraphics[width=1\textwidth]{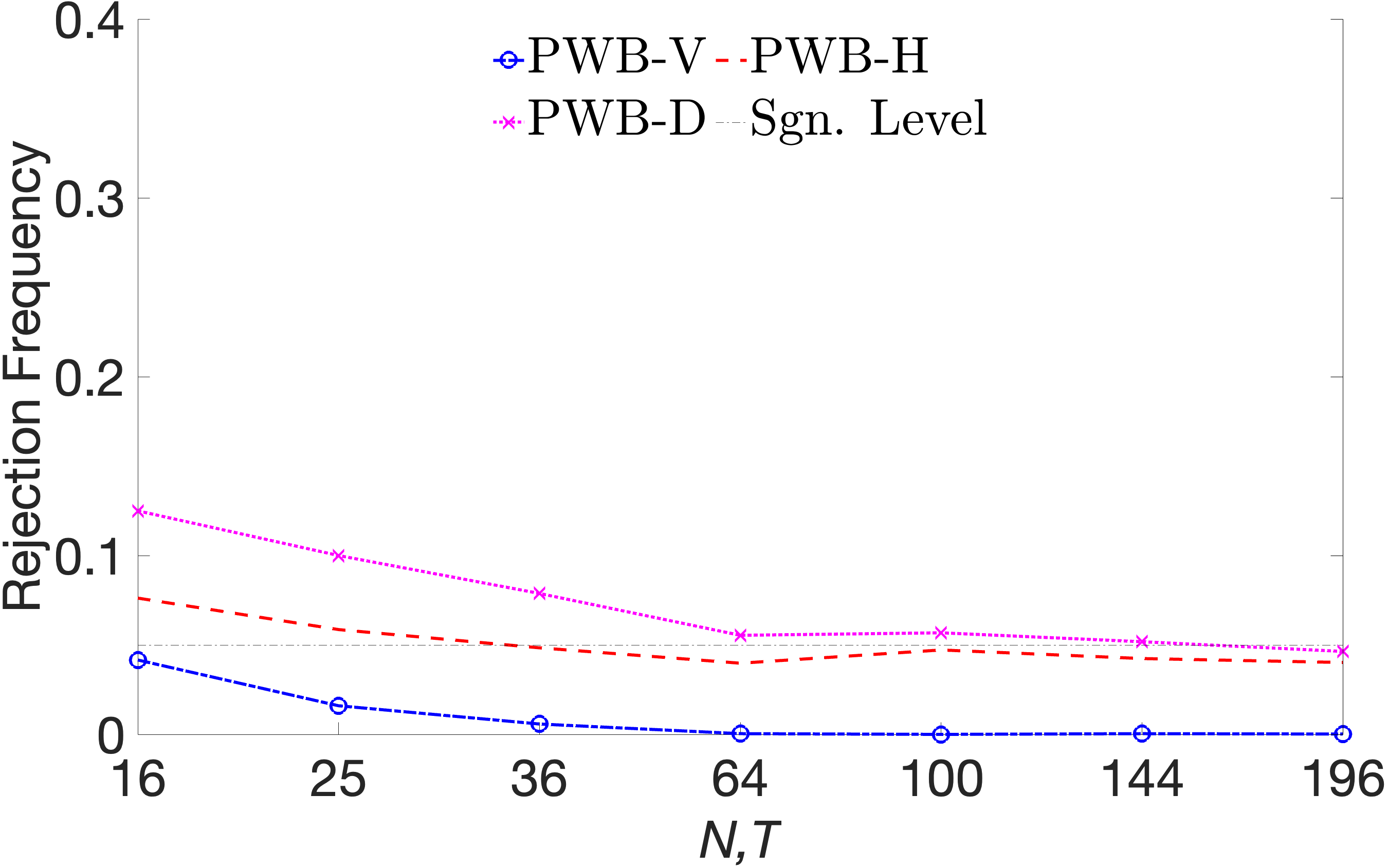}
\end{subfigure} \begin{subfigure}[t]{0.49\hsize} \subcaption{DGP
(\ref{eq: dgp 3})} \includegraphics[width=1\textwidth]{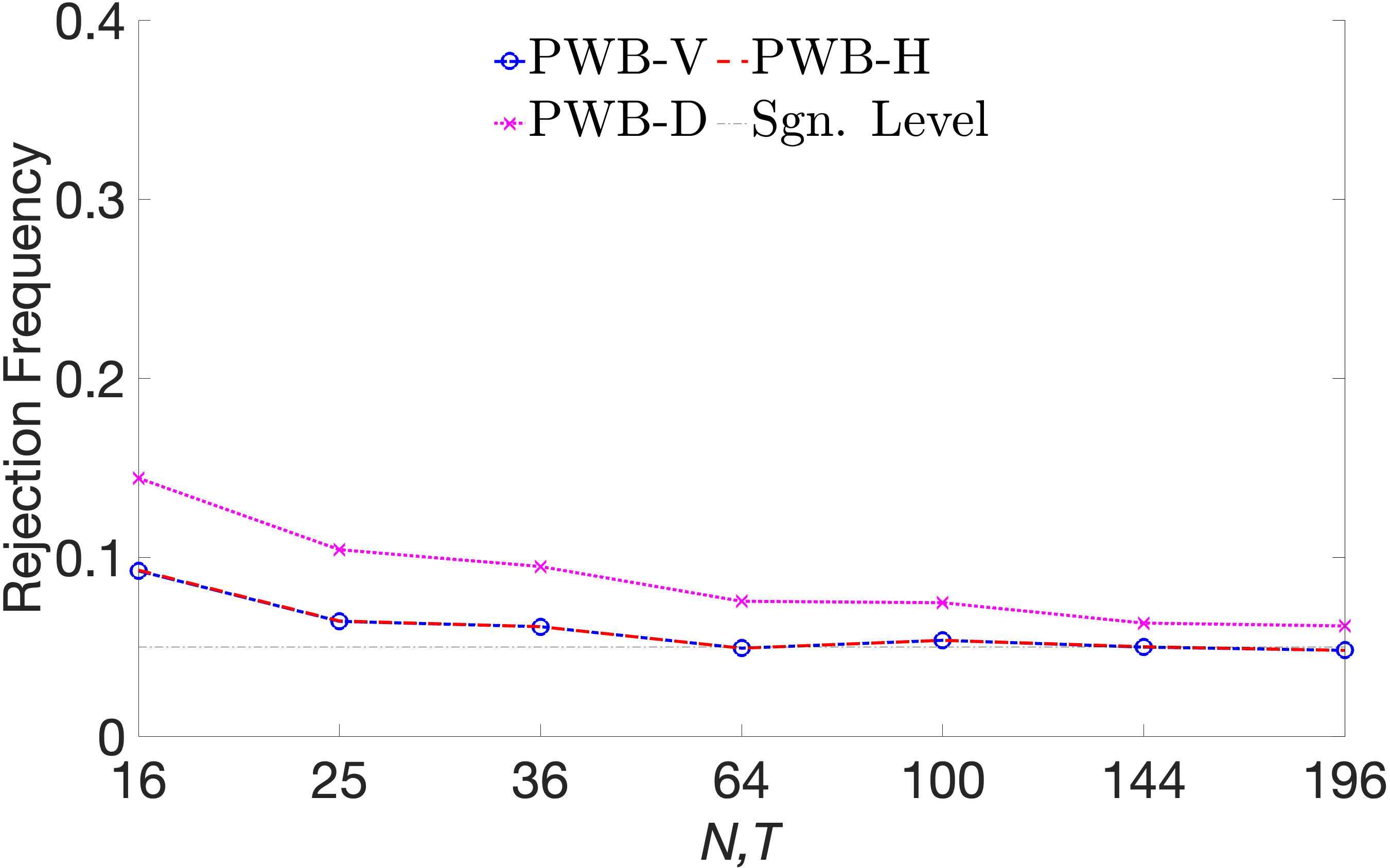}
\end{subfigure}

\begin{subfigure}[t]{0.49\hsize} \subcaption{DGP (\ref{eq: dgp 4})}\includegraphics[width=1\textwidth]{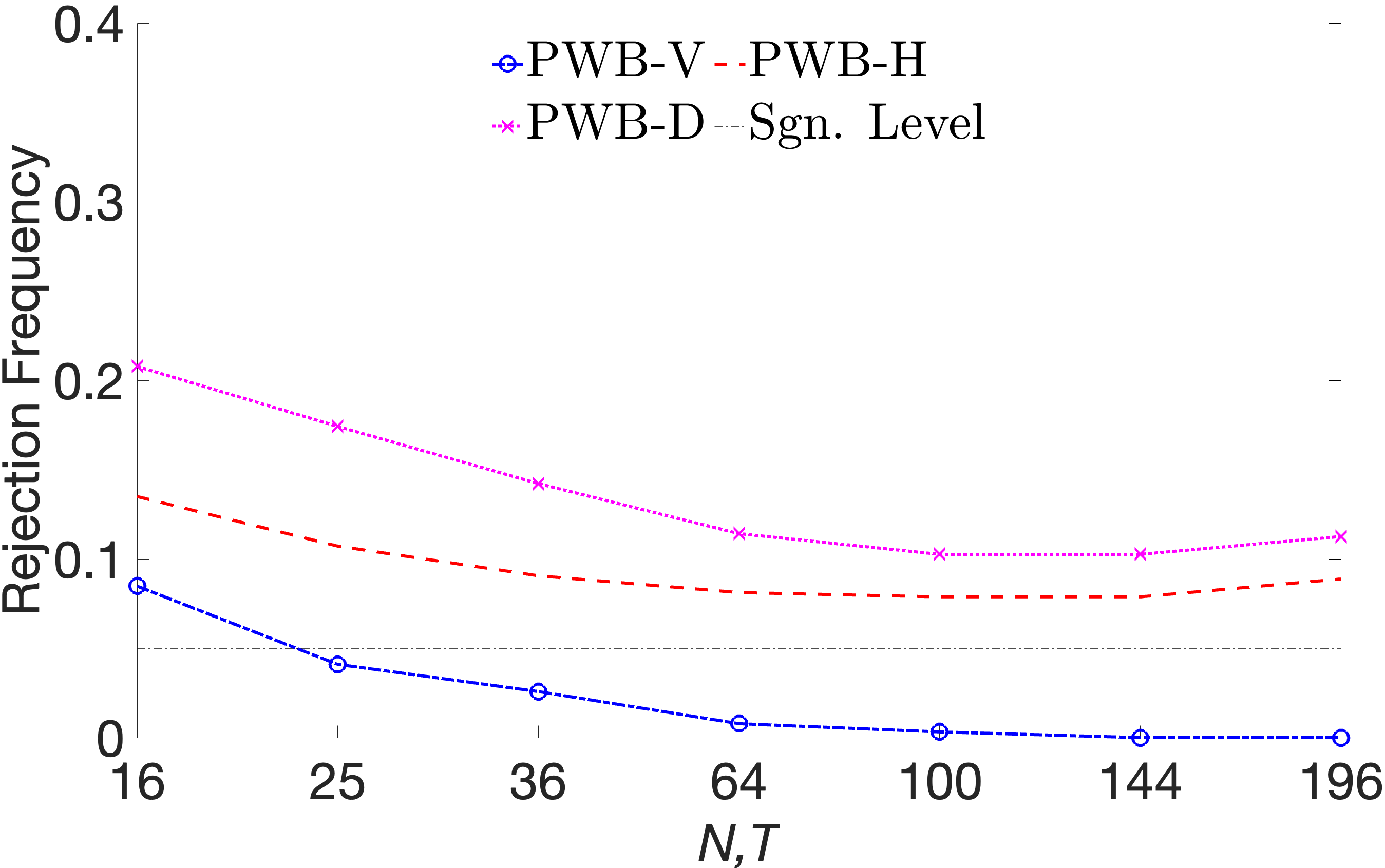}
\end{subfigure} \begin{subfigure}[t]{0.49\hsize} \subcaption{DGP
(\ref{eq: dgp 5})} \includegraphics[width=1\textwidth]{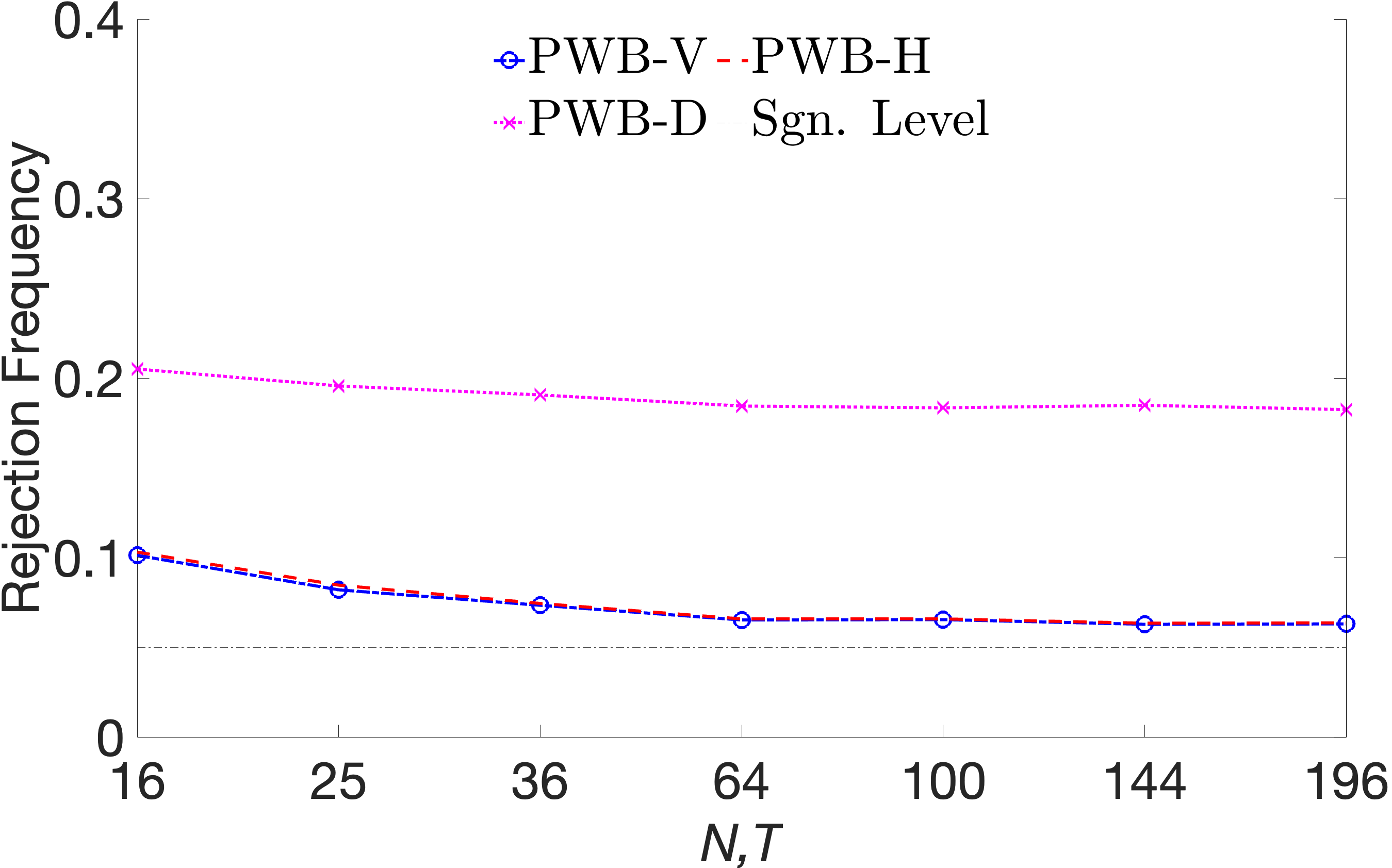}
\end{subfigure}

\caption{\textbf{Rejection Frequency for DGPs (\ref{eq: dgp 2})-(\ref{eq: dgp 5}).}
For each bootstrap
method, $B=999$. Results are based on 5,000 Monte Carlo replicates.
The predetermined significance level is 5\%.}
\label{fig: rej frequency DGPs 2-5} 
\end{figure}

Figure~\ref{tab:average correct rate for DRC varying NT, ols} reports the
corresponding rejection frequencies. Panel~(a) presents the results under
DGP~\eqref{eq: dgp 1}, where there is strong cluster dependence along both
dimensions. As expected, the performance of all methods improves as the
number of clusters increases, which is consistent with the theoretical
prediction that all methods are valid in this regime.

Panel~(b) reports the results under DGP~\eqref{eq: dgp 2}, where the
limiting distribution is non-Gaussian. In line with the theory, PWB-V
does not provide valid inference, whereas the other methods perform well.
Notably, PWB-H performs relatively well even in small samples, although
this may partly reflect finite-sample randomness. Since PWB-H is a hybrid
of PWB-D and PWB-V, its finite-sample performance generally lies between
those of the two benchmark procedures.

Panel~(c) presents the results under DGP~\eqref{eq: dgp 3}, a setting in
which the dependence is primarily driven by intersection-level clustering.
All methods perform adequately in this scenario, again consistent with the
theoretical predictions.

Panel~(d) considers DGP~\eqref{eq: dgp 4}, the most challenging case, in
which the limiting distribution is non-Gaussian and no feasible procedure
can achieve asymptotic validity. The simulation results confirm this theoretical impossibility result, as
none of the methods exhibits further improvement when the number of
clusters becomes sufficiently large. Panel~(e) examines DGP~\eqref{eq: dgp 5}, where the asymptotic
distribution is Gaussian. In this setting, PWB-D is invalid and exhibits
substantial overrejection even when the sample size is large. By contrast,
PWB-V and PWB-H perform well as $N$ and $T$ increase.

Overall, the simulation results across the five scenarios corroborate the
theoretical findings. Among the feasible procedures, PWB-H delivers robust
performance across most scenarios, except in the infeasible regime where no
feasible method can be asymptotically valid. Therefore, in view of both the
theoretical analysis and the simulation evidence, we recommend PWB-H for
practical applications.

\begin{table}[t!]
\centering %
\begin{tabular}{lccccccc}
\hline \hline
$N,T$  & 20  & 30  & 50  & 70  & 100  & 150  & 200 \tabularnewline
\hline 
D, DGP \eqref{eq: dgp 1}  & 0.991  & 0.998  & 0.999  & 0.999  & 0.999  & 0.999  & 0.999\tabularnewline
V\&N, DGP \eqref{eq: dgp 2}  & 0.987  & 0.996  & 0.998 & 0.999  & 0.999  & 0.999 & 0.999 \tabularnewline
V\&G, DGP \eqref{eq: dgp 3} & 0.686  & 0.715  & 0.786  & 0.812  & 0.860  & 0.880  & 0.922\tabularnewline
I\&N, DGP \eqref{eq: dgp 4}  & 0.972  & 0.990  & 0.990  & 0.991  & 0.997 & 0.999  & 0.999 \tabularnewline
I\&G, DGP \eqref{eq: dgp 5}  & 0.538  & 0.631  & 0.735  & 0.810  & 0.855  & 0.933  & 0.964 \tabularnewline
\hline \hline
\end{tabular}\caption{\textbf{Classification accuracy of the dependence-regime classifier
(DRC) across varying \(N\) and \(T\).} Entries report the fraction of
replications in which the DRC selects the population regime; the optimal
classification probability approaches one over distinguishable regimes. For
the two uniformly indistinguishable non-Gaussian regimes, classification is
counted as correct when the DRC selects the non-Gaussian branch. For each bootstrap
method, $B=999$. Results are based on
5,000 Monte Carlo replicates.}
\label{tab:average correct rate for DRC varying NT, ols} 
\end{table}

We also report simulation results for the dependence-regime classifier
(DRC) in Table~\ref{tab:average correct rate for DRC varying NT, ols}.
The entries in the table measure the accuracy of regime classification, not
the empirical size of the subsequent confidence interval; in the population
limit, the relevant classification probability is expected to converge to one
over the distinguishable regimes.
For the two indistinguishable non-Gaussian regimes, we record a classification
as correct whenever the procedure flags the component as non-Gaussian
(i.e., assigns it to either of the two non-Gaussian regimes). The
classification accuracy increases with the sample sizes $N$ and $T$,
and all scenarios exceed 90\% accuracy when $N=T=200$. When $N=T=20$,  V\&N and I\&N demonstrate very high accuracy, indicating the KS diagnostic remains accurate in small samples as long as $B$ is sufficiently large. By contrast, I\&G and V\&G exhibit lower classification
accuracy because they are more likely to be confused when the variance
components are imprecisely estimated in small samples.
 In the unreported results, we also consider
designs with a small numbers of factors ($K=2$), and
obtain similar results.

\section{Conclusion}

\label{sec:conclusion} This paper contributes to the econometrics literature on inference under
two-way clustering with serially and spatially dependent common effects. We
characterize the limiting distribution of the OLS estimator across five
mutually exclusive and exhaustive regimes, determined by the relative
contributions of the two clustering dimensions and the interaction component.
These regimes include both Gaussian and non-Gaussian limits and imply
different requirements for valid inference.

We show that one non-Gaussian regime is intrinsically infeasible: without
additional restrictions on the DGP, no procedure can achieve uniformly
consistent inference in that regime. We further establish two additional
impossibility results. First, the infeasible regime cannot be uniformly
distinguished from one feasible regime. Second, heterogeneous score components
under two-way clustering preclude uniformly consistent inference. Together,
these results identify fundamental limits on what can be learned from the
data in two-way clustered settings.

To address the feasible cases, we propose a family of projection-based wild
bootstrap procedures. The hybrid procedure, PWB-H, combines a data-driven
Gaussianity diagnostic, variance-scaling adjustments, and dependence-adaptive
bootstrap multipliers. It delivers uniformly valid inference across all four
feasible regimes, while the remaining non-Gaussian regime is shown to be
fundamentally beyond the reach of uniformly valid data-driven inference. 
Monte Carlo simulations confirm that PWB-H performs well across a range of
dependence structures and sample sizes.

\section*{Appendix}

\appendix

\section{Serial-Spatial Two-Way CRVE}

 \renewcommand{%
\theequation}{A.\arabic{equation}}
\renewcommand{\thelemma}{A.\arabic{lemma}}
\renewcommand{%
\thetheorem}{A.\arabic{theorem}}

\renewcommand{\thesubsection}{A.\arabic{subsection}}

We also propose a variance estimator for $\widehat{\bm{\beta}}-\bm{\beta}_{0}$
under the general two-way clustering with serial and spatial dependence:
\begin{equation}
\widehat{\bm{V}}=\widehat{\bm{Q}}^{-1}\left(\frac{\widehat{\bm{\sigma}}_{a}^{2}}{N}+\frac{\widehat{\bm{\sigma}}_{d}^{2}}{T}+\frac{\widehat{\bm{\sigma}}_{w}^{2}}{NT}\right)\widehat{\bm{Q}}^{-1},\label{eq: new variance}
\end{equation}
where $\widehat{\bm{Q}}=\frac{1}{NT}(\bm{X}^\top\bm{X})$, $\widehat{\bm{\sigma}}_{a}^{2}$ and $\widehat{\bm{\sigma}}_{d}^{2}$
are defined in (\ref{eq: sigma a estimator}) and (\ref{eq: sigma d estimator}),
respectively. The variance estimator for the intersection $\widehat{\bm{\sigma}}_{w}^{2}$
is given by 
\begin{align*}
\widehat{\bm{\sigma}}_{w}^{2} & =\frac{1}{NT}\sum_{i=1}^{N}\sum_{j=1}^{N}\sum_{t=1}^{T}\sum_{\tau=1}^{T}q^{\left|t-\tau\right|}\mathcal{K}\left(\frac{\mathfrak{d}_{ij}}{\mathfrak{d}_{N}}\right)\ddot{\bm{w}}_{it}\ddot{\bm{w}}_{j\tau}^{\top}.
\end{align*}
It is guaranteed to be positive definite and, when the DGP is clustered
along the intersection dimension, it is less likely to yield a standard
error of zero for the statistic of interest.

 \renewcommand{%
\theequation}{B.\arabic{equation}}
\renewcommand{\thelemma}{B.\arabic{lemma}}
\renewcommand{%
\thetheorem}{B.\arabic{theorem}}

\renewcommand{\thesubsection}{B.\arabic{subsection}}

\section{Proof of Main Theorem}

\subsection{Proof of Theorem \protect\ref{thm: main}}
\begin{proof}
Let $\mathcal{B}$ be the class of DGPs satisfying Assumptions \ref{as: AHS representation}-\ref{as: same rate}.
We prove bootstrap validity (and uniformity over $f\in\mathcal{B}$)
in three steps: (i) marginal CLTs + pairwise covariance, (ii) joint
CLT + truncation, and (iii) continuity of the limit + uniformization
via Lemma~\ref{lemma: uniform convergence}. Throughout, we work
along an arbitrary \emph{convergent} sequence of DGPs $\{f_{NT}\}_{N,T}$
such that the associated singular values $\{\bm{c}_{ll',f}\}_{l,l'}$ and
variance parameters $\bm{\nu}_{NT,f}$, as defined in \eqref{eq:nu-f},
satisfy 
\[
(\{\bm{c}_{ll',f}\}_{l,l'},\bm{\nu}_{NT,f})\;\longrightarrow\;(\{\bm{c}_{ll',0}\}_{l,l'},\bm{\nu}_{0})\qquad\text{as }N,T\to\infty.
\]
Most of the subsequent asymptotic arguments are established along
such convergent sequences. Uniform convergence over the admissible
class of DGPs is further obtained by an application of Lemma~\ref{lemma: uniform convergence}.

\paragraph{Step 1: Bootstrap marginal CLTs.}

\subparagraph{(a) The target limit.}

Define the limiting distribution 
\begin{equation}
\mathcal{\bm{L}}_{0}\!\left(\bm{\nu}_{0},\{\bm{c}_{ll',0}\}_{l,l'=1}^{\infty}\right)=\bm{\nu}_{a,0}\bm{Z}^{a}+\bm{\nu}_{d,0}\bm{Z}^{d}+\bm{\nu}_{e,0}\bm{Z}^{e}+\bm{\nu}_{v,0}\sum_{l,l'=1}^{\infty}\bm{c}_{ll',0}\odot\left(\bm{Z}_{l}^{\phi}\odot\bm{Z}_{l'}^{\psi}\right),\label{eq:L0}
\end{equation}
where $\bm{Z}^{a},\bm{Z}^{d},\bm{Z}^{e}$ are mutually independent
and distributed as $\mathcal{N}(\bm{0},\mathbf{I})$. The collection
$\{(\bm{Z}_{l}^{\phi},\bm{Z}_{l'}^{\psi})\}_{l,l'\ge1}$ is also Gaussian,
with the only potentially nonzero cross-covariances being $\{Cov(\bm{Z}^{a},\bm{Z}_{l}^{\phi})\}_{l\ge1}$
(due to $\{\bm{\alpha}_{i}\}_{i}$) and $\{Cov(\bm{Z}^{d},\bm{Z}_{l'}^{\psi})\}_{l'\ge1}$
(due to $\{\bm{\xi}_{t}\}_{t}$).

Define the parameter sequence 
\begin{equation}
{\bm{\nu}}_{NT,f}\equiv\left({\bm{\nu}}_{NT,af},{\bm{\nu}}_{NT,df},{\bm{\nu}}_{NT,ef},{{\nu}}_{NT,vf},{\bm{\nu}}_{NT,a1f},{\bm{\nu}}_{NT,d1f},\ldots\right)\label{eq:nu-f}
\end{equation}
with 
\begin{equation}
{\bm{\nu}}_{NT,af}=r_{NT,f}\frac{\bm{\sigma}_{a,f}}{\sqrt{N}},\quad{\bm{\nu}}_{NT,df}=r_{NT,f}\frac{\bm{\sigma}_{d,f}}{\sqrt{T}},\quad{\bm{\nu}}_{NT,ef}=r_{NT,f}\frac{\bm{\sigma}_{e,f}}{\sqrt{NT}},\quad{{\nu}}_{NT,vf}=r_{NT,f}\frac{1}{\sqrt{NT}},\label{eq:nu-hat}
\end{equation}
and, for each $l,l'\ge1$, 
\begin{align*}
&\bm{\sigma}_{al,f}=E\!\left(\bm{a}_{i}\phi_{l}\!\left(\bm{\alpha}_{i}\right)^{\top}\right),\quad\bm{\sigma}_{dl',f}=\sum_{\iota=-\infty}^{\infty}E\!\left(\bm{d}_{t}\psi_{l'}\!\left(\bm{\xi}_{t+\iota}\right)^{\top}\right),\\
&\quad{\bm{\nu}}_{NT,alf}=r_{NT,f}\frac{\bm{\sigma}_{al,f}}{\sqrt{N}},\quad{\bm{\nu}}_{NT,dl'f}=r_{NT,f}\frac{\bm{\sigma}_{dl',f}}{\sqrt{T}}.
\end{align*}
Note that the parameter space of $\bm{\nu}_{NT,\bullet f}$ is compact
by construction. In particular, since $|r_{NT,f}|\le\sqrt{NT}$,
we have $\|\bm{\nu}_{NT,ef}\|=\big\|r_{NT,f}\bm{\sigma}_{e,f}/\sqrt{NT}\big\|\le\|\bm{\sigma}_{e,f}\|$,
and $\bm{\sigma}_{e,f}$ is uniformly bounded over $f\in\mathcal{B}$
under Assumption~\ref{as: moment and variance}. Let $\bm{\nu}_{\bullet,0}=\lim_{N,T\to\infty}\bm{\nu}_{NT,\bullet}$.

By Lemma~\ref{lemma: original clt}, along the present convergent
sequence, we have 
\begin{equation}
\Bigl\| P_{NT,f}\!\Bigl(r_{NT,f}\Bigl(\frac{1}{NT}\sum_{i=1}^{N}\sum_{t=1}^{T}\bm{s}_{it}\Bigr)\Bigr)-\mathcal{\bm{L}}_{0}\!\left(\bm{\nu}_{0},\{\bm{c}_{ll',0}\}_{l,l'=1}^{\infty}\right)\Bigr\|_{\infty}\rightarrow0.\label{eq:orig-CLT}
\end{equation}
Moreover, the map $(\bm{\nu}_{0},\{\bm{c}_{ll',0}\})\mapsto\mathcal{\bm{L}}_{0}(\bm{\nu}_{0},\{\bm{c}_{ll',0}\})$
is continuous (Step~3 of Lemma~\ref{lemma: original clt}).

\subparagraph{(b) Reduction using $\widehat{\bm{Q}}^{-1}$.}

Under Assumption~\ref{as: AHS representation}, define 
\[
\bm{a}_{i}^{x}=E\!\left(\bm{X}_{it}^{\top}\bm{X}_{it}\mid\bm{\alpha}_{i}\right)-E\!\left(\bm{X}_{it}^{\top}\bm{X}_{it}\right),\quad\bm{d}_{t}^{x}=E\!\left(\bm{X}_{it}^{\top}\bm{X}_{it}\mid\bm{\xi}_{t}\right)-E\!\left(\bm{X}_{it}^{\top}\bm{X}_{it}\right),
\]
\[
\bm{w}_{it}^{x}=\bm{X}_{it}^{\top}\bm{X}_{it}-\bm{a}_{i}^{x}-\bm{d}_{t}^{x}-E\!\left(\bm{X}_{it}^{\top}\bm{X}_{it}\right),\qquad\widehat{\bm{Q}}=\frac{1}{NT}\sum_{i=1}^{N}\sum_{t=1}^{T}\bm{X}_{it}^{\top}\bm{X}_{it}.
\]
Then 
\begin{equation}
\widehat{\bm{Q}}=\frac{1}{N}\sum_{i=1}^{N}\bm{a}_{i}^{x}+\frac{1}{T}\sum_{t=1}^{T}\bm{d}_{t}^{x}+\frac{1}{NT}\sum_{i=1}^{N}\sum_{t=1}^{T}\bm{w}_{it}^{x}+E\!\left(\bm{X}_{it}^{\top}\bm{X}_{it}\right).\label{eq:Q-decomp}
\end{equation}
By i.i.d.\ of $\{\bm{\alpha}_{i}\}$ and stationarity of $\{\bm{\xi}_{t}\}$,
the variance admits the decomposition 
\begin{align*}
Var(\widehat{\bm{Q}}) & =\frac{1}{N}E(\bm{a}_{i}^{x}\bm{a}_{i}^{x\top})+\frac{1}{T^{2}}\sum_{t=1}^{T}\sum_{t'=1}^{T}E(\bm{d}_{t}^{x}\bm{d}_{t'}^{x\top})+\frac{1}{N^{2}T^{2}}\sum_{i=1}^{N}\sum_{t=1}^{T}\sum_{t'=1}^{T}E(\bm{w}_{it}^{x}\bm{w}_{it'}^{x\top}),\\
 & =\frac{1}{N}E(\bm{a}_{i}^{x}\bm{a}_{i}^{x\top})+\frac{1}{T}\sum_{\iota=-(T-1)}^{T-1}\Bigl(1-\frac{|\iota|}{T}\Bigr)E(\bm{d}_{t}^{x}\bm{d}_{t+\iota}^{x\top})+\frac{1}{NT}\sum_{\iota=-(T-1)}^{T-1}\Bigl(1-\frac{|\iota|}{T}\Bigr)E(\bm{w}_{it}^{x}\bm{w}_{it+\iota}^{x\top}).
\end{align*}

For any $p>1$ and any $\sigma$-field $\mathcal{G}$, conditional
Jensen implies 
\[
E\!\left(\bigl\| E(\bm{X}_{it}^{\top}\bm{X}_{it}\mid\mathcal{G})\bigr\|^{p}\right)\le E\!\left(\|\bm{X}_{it}^{\top}\bm{X}_{it}\|^{p}\right).
\]
Using convexity of $x\mapsto\|x\|^{4(\zeta+\delta)}$ and the inequality
$\|\sum_{j=1}^{4}x_{j}\|^{p}\le C\sum_{j=1}^{4}\|x_{j}\|^{p}$ for
$p\ge1$, we obtain 
\[
E\!\left(\|\bm{w}_{it}^{x}\|^{4(\zeta+\delta)}\right)\le C\,E\!\left(\|\bm{X}_{it}^{\top}\bm{X}_{it}\|^{4(\zeta+\delta)}\right)\le C\,E\!\left(\|\bm{X}_{it}\|^{8(\zeta+\delta)}\right)^{1/2}<\infty.
\]

Hence, by Jensen's inequality and Theorem~14.13.2 of Hansen~(\citeyear{hansen2022econometrics}),
\[
\Bigl\|\sum_{\iota=-\infty}^{\infty}E(\bm{w}_{it}^{x}\bm{w}_{it+\iota}^{x\top})\Bigr\|\le\sum_{\iota=-\infty}^{\infty}E\bigl\| E(\bm{w}_{it}^{x}\bm{w}_{it+\iota}^{x\top}\mid\bm{\alpha}_{i})\bigr\|\le C\,E\!\left(\|\bm{w}_{it}^{x}\|^{4(\zeta+\delta)}\right)^{1/2(\zeta+\delta)}\sum_{\iota=-\infty}^{\infty}\alpha(\iota)^{1-1/2(\zeta+\delta)}<\infty.
\]

By analogous arguments for $\bm{d}_{t}^{x}$, it follows that 
\[
Var(\widehat{\bm{Q}})=\frac{1}{N}E(\bm{a}_{i}^{x}\bm{a}_{i}^{x\top})+\frac{1}{T}\sum_{\iota=-\infty}^{\infty}E(\bm{d}_{t}^{x}\bm{d}_{t+\iota}^{x\top})\,(1+o(1))+\frac{1}{NT}\sum_{\iota=-\infty}^{\infty}E(\bm{w}_{it}^{x}\bm{w}_{it+\iota}^{x\top})\,(1+o(1))=o(1).
\]
Hence, by Chebyshev's inequality and continuous mapping theorem 
\[
\qquad\widehat{\bm{Q}}\overset{P}{\to}\bm{Q}\equiv E(\bm{X}_{it}^{\top}\bm{X}_{it}),\qquad\widehat{\bm{Q}}^{-1}\overset{P}{\to}\bm{Q}^{-1}>0.
\]
Since $\widehat{\bm{Q}}^{-1}$ is applied to both the original and
bootstrap statistics, by Slutsky's Lemma and Cramér-Wold Device it is sufficient to establish
the bootstrap analogue of \eqref{eq:orig-CLT} for the bootstrap score,
namely, 
\begin{equation}
\Bigl\| P_{NT,f}^{*}\!\Bigl(r_{NT,f}\Bigl(\frac{1}{NT}\sum_{i=1}^{N}\sum_{t=1}^{T}\bm{s}_{it}^{*}\Bigr)\Bigr)-\mathcal{\bm{L}}_{0}\!\left(\bm{\nu}_{0},\{\bm{c}_{ll',0}\}_{l,l'=1}^{\infty}\right)\Bigr\|_{\infty}\xrightarrow{P}0.\label{eq:boot-goal}
\end{equation}

\subparagraph{(c) Bootstrap decomposition.}

We can expand $r_{NT,f}\left(\frac{1}{NT}\sum_{i=1}^{N}\sum_{t=1}^{T}\bm{s}_{it}^{*}\right)$
as 
\begin{align}
 & r_{NT,f}\left(\frac{1}{\sqrt{N}}\frac{1}{\sqrt{N}}\sum_{i=1}^{N}\widehat{\bm{a}}_{i}\cdot\eta_{i}^{*b}+\frac{1}{\sqrt{T}}\frac{1}{\sqrt{T}}\sum_{t=1}^{T}\widehat{\bm{d}}_{t}\cdot\eta_{t}^{*b}+\frac{1}{\sqrt{NT}}\frac{1}{\sqrt{NT}}\sum_{i=1}^{N}\sum_{t=1}^{T}\widehat{\bm{w}}_{it}\cdot\eta_{i}^{*b}\eta_{t}^{*b}\right).\label{eq:decompose sit}
\end{align}
Focusing on $\frac{1}{\sqrt{NT}}\sum_{i=1}^{N}\sum_{t=1}^{T}\widehat{\bm{w}}_{it}\cdot\eta_{i}^{*b}\eta_{t}^{*b}$,
we can decompose $\widehat{\bm{w}}_{it}$ into $\widehat{\bm{e}}_{it}$,
$\widehat{\bm{c}}_{it}$, and $\widehat{\bm{v}}_{it}$: 
\begin{align}
\widehat{\bm{e}}_{it}= & \bm{e}_{it}-\bm{e}_{Nt}-\bm{e}_{iT}+\bm{e}_{NT}\nonumber \\
= & \bm{e}_{it}-\frac{1}{N}\sum_{i=1}^{N}\bm{e}_{it}-\frac{1}{T}\sum_{t=1}^{T}\bm{e}_{it}+\frac{1}{NT}\sum_{i=1}^{N}\sum_{t=1}^{T}\bm{e}_{it},\label{eq: decompose eit}\\
\widehat{\bm{c}}_{it}= & \left(\bm{X}_{it}^{\top}-\frac{1}{N}\sum_{i=1}^{N}\bm{X}_{it}^{\top}-\frac{1}{T}\sum_{t=1}^{T}\bm{X}_{it}^{\top}+\frac{1}{NT}\sum_{i=1}^{N}\sum_{t=1}^{T}\bm{X}_{it}^{\top}\right)\left(\widehat{\bm{\beta}}-\bm{\beta}_{0}\right),\nonumber \\
\widehat{\bm{v}}_{it}= & \widehat{\bm{w}}_{it}-\widehat{\bm{c}}_{it}-\widehat{\bm{e}}_{it}\nonumber \\
= & \bm{v}_{it}-\frac{1}{N}\sum_{i=1}^{N}\bm{v}_{it}-\frac{1}{T}\sum_{t=1}^{T}\bm{v}_{it}+\frac{1}{NT}\sum_{i=1}^{N}\sum_{t=1}^{T}\bm{v}_{it}\nonumber \\
= & \sum_{l,l'=1}^{\infty}\bm{c}_{ll',f}\odot\psi_{l'}\left(\bm{\xi}_{t}\right)\odot\phi_{l}\left(\bm{\alpha}_{i}\right)-\frac{1}{T}\sum_{l,l'=1}^{\infty}\sum_{t=1}^{T}\bm{c}_{ll',f}\odot\psi_{l'}\left(\bm{\xi}_{t}\right)\odot\phi_{l}\left(\bm{\alpha}_{i}\right)\nonumber \\
 & -\frac{1}{N}\sum_{l,l'=1}^{\infty}\sum_{i=1}^{N}\bm{c}_{ll',f}\odot\psi_{l'}\left(\bm{\xi}_{t}\right)\odot\phi_{l}\left(\bm{\alpha}_{i}\right)+\frac{1}{NT}\sum_{l,l'=1}^{\infty}\sum_{i=1}^{N}\sum_{t=1}^{T}\bm{c}_{ll',f}\odot\psi_{l'}\left(\bm{\xi}_{t}\right)\odot\phi_{l}\left(\bm{\alpha}_{i}\right)\nonumber \\
= & \sum_{l,l'=1}^{\infty}\bm{c}_{ll',f}\odot\left(\psi_{l'}\left(\bm{\xi}_{t}\right)-\frac{1}{T}\sum_{t=1}^{T}\psi_{l'}\left(\bm{\xi}_{t}\right)\right)\odot\left(\phi_{l}\left(\bm{\alpha}_{i}\right)-\frac{1}{N}\sum_{i=1}^{N}\phi_{l}\left(\bm{\alpha}_{i}\right)\right).\label{eq: decompose vit}
\end{align}
Then, we have the following decomposition: {\small{}{}{} 
\begin{align}
\frac{r_{NT,f}}{NT}\sum_{i=1}^{N}\sum_{t=1}^{T}\widehat{\bm{w}}_{it}\cdot\eta_{i}^{*b}\eta_{t}^{*b}= & \frac{r_{NT,f}}{\sqrt{NT}}\frac{1}{\sqrt{NT}}\sum_{i=1}^{N}\sum_{t=1}^{T}\widehat{\bm{e}}_{it}\cdot\eta_{i}^{*b}\eta_{t}^{*b}+\frac{r_{NT,f}}{\sqrt{NT}}\frac{1}{\sqrt{NT}}\sum_{i=1}^{N}\sum_{t=1}^{T}\widehat{\bm{c}}_{it}\cdot\eta_{i}^{*b}\eta_{t}^{*b}\nonumber \\
 & +\frac{r_{NT,f}}{\sqrt{NT}}\sum_{l,l'=1}^{\infty}\bm{c}_{ll',f}\odot\left(\frac{1}{\sqrt{T}}\sum_{t=1}^{T}\left(\psi_{l'}\left(\bm{\xi}_{t}\right)-\frac{1}{T}\sum_{t=1}^{T}\psi_{l'}\left(\bm{\xi}_{t}\right)\right)\eta_{t}^{*b}\right)\nonumber \\
 & \qquad\odot\left(\frac{1}{\sqrt{N}}\sum_{i=1}^{N}\left(\phi_{l}\left(\bm{\alpha}_{i}\right)-\frac{1}{N}\sum_{i=1}^{N}\phi_{l}\left(\bm{\alpha}_{i}\right)\right)\eta_{i}^{*b}\right).\label{eq:decompose wit}
\end{align}
}{\small\par}

Combining \eqref{eq:decompose sit}-\eqref{eq:decompose wit}, we
can write 
\begin{equation}
r_{NT,f}\Bigl(\frac{1}{NT}\sum_{i=1}^{N}\sum_{t=1}^{T}\bm{s}_{it}^{*}\Bigr)=\widehat{\bm{Z}}_{N}^{a*}+\widehat{\bm{Z}}_{T}^{d*}+\widehat{\bm{Z}}_{NT}^{e*}+\widehat{\bm{Z}}_{NT}^{c*}+\frac{r_{NT,f}}{\sqrt{NT}}\sum_{l,l'=1}^{\infty}\bm{c}_{ll',f}\odot\widehat{\bm{Z}}_{N,l}^{*\phi}\odot\widehat{\bm{Z}}_{T,l'}^{*\psi},\label{eq:boot-score-decomp}
\end{equation}
where 
\begin{align*}
\widehat{\bm{Z}}_{N}^{a*}= & \frac{r_{NT,f}}{\sqrt{N}}\frac{1}{\sqrt{N}}\sum_{i=1}^{N}\widehat{\bm{a}}_{i}\cdot\eta_{i}^{*b}, & \widehat{\bm{Z}}_{T}^{d*}= & \frac{r_{NT,f}}{\sqrt{T}}\frac{1}{\sqrt{T}}\sum_{t=1}^{T}\widehat{\bm{d}}_{t}\cdot\eta_{t}^{*b},\\
\widehat{\bm{Z}}_{NT}^{e*}= & \frac{r_{NT,f}}{\sqrt{NT}}\frac{1}{\sqrt{NT}}\sum_{i=1}^{N}\sum_{t=1}^{T}\widehat{\bm{e}}_{it}\cdot\eta_{i}^{*b}\eta_{t}^{*b}, & \widehat{\bm{Z}}_{NT}^{c*}= & \frac{r_{NT,f}}{\sqrt{NT}}\frac{1}{\sqrt{NT}}\sum_{i=1}^{N}\sum_{t=1}^{T}\widehat{\bm{c}}_{it}\cdot\eta_{i}^{*b}\eta_{t}^{*b},\\
\bm{Z}_{N,l}^{*\phi}= & \frac{1}{\sqrt{N}}\sum_{i=1}^{N}\left(\phi_{l}\left(\bm{\alpha}_{i}\right)-\frac{1}{N}\sum_{i=1}^{N}\phi_{l}\left(\bm{\alpha}_{i}\right)\right)\eta_{i}^{*b}, & \bm{Z}_{T,l'}^{*\psi}= & \frac{1}{\sqrt{T}}\sum_{t=1}^{T}\left(\psi_{l'}\left(\bm{\xi}_{t}\right)-\frac{1}{T}\sum_{t=1}^{T}\psi_{l'}\left(\bm{\xi}_{t}\right)\right)\eta_{t}^{*b}.
\end{align*}

\subparagraph{(d) Bootstrap marginal CLT limits (by regimes).}

We verify that each component matches the corresponding term in \eqref{eq:L0},
including the relevant covariance structure, and we only need to do
so for $l,l'\le L$ with a fixed truncation level $L<\infty$ (the
truncation error can be controlled by square summability of $\{\bm{c}_{ll',f}\}_{l,l'\ge1}$).

Under Assumption~\ref{as: same rate}, if either
\(T\bm{\sigma}_{a,f}^{2}\) or \(N\bm{\sigma}_{d,f}^{2}\) diverges, then
the diagonal components of \(\bm{\sigma}_{a,f}^{2}\) or
\(\bm{\sigma}_{d,f}^{2}\), respectively, are of the same order or negligible after normalization by the rate of convergence. This case
corresponds to Scenario~3 (Diverging) below. Otherwise, the relevant rate
of convergence is \(\sqrt{NT}\), corresponding to Scenario~1 (Vanishing)
or Scenario~2 (Intermediate).

To simplify notation, we  present the proof under the normalization
that the diagonal components of
\(\bm{\sigma}_{a,f}^{2}\), \(\bm{\sigma}_{d,f}^{2}\), and
\(\bm{\sigma}_{v,f}^{2}\) are of the same order within each component.
This restriction is only notational. We further discuss the case where the diagonal entries of
\(\bm{\sigma}_{a,f}^{2}\), \(\bm{\sigma}_{d,f}^{2}\), and
\(\bm{\sigma}_{v,f}^{2}\) are allowed to have heterogeneous orders across
coordinates.

Given that $\bm{\nu}_{f}$ converges, we can separate
into three main scenarios: 1) Vanishing: when $T\bm{\sigma}_{a,f}^{2}=o\left(1\right)$
and $N\bm{\sigma}_{d,f}^{2}=o\left(1\right)$; 2) Intermediate: $T\bm{\sigma}_{a,f}^{2}\to\bm{\varphi}_{a}>0$
or $N\bm{\sigma}_{d,f}^{2}\rightarrow\bm{\varphi}_{d}>0$, and both $T\bm{\sigma}_{a,f}^{2}$
and $N\bm{\sigma}_{d,f}^{2}$ do not diverge; 3) Diverging: $T\bm{\sigma}_{a,f}^{2}$
or $N\bm{\sigma}_{d,f}^{2}$ diverges.

\medskip{}
\underline{\textbf{Scenario 1 (Vanishing):}} $T\bm{\sigma}_{a,f}^{2}=o(1)$
and $N\bm{\sigma}_{d,f}^{2}=o(1)$. Then Lemma~\ref{lemma:limit form-1}
implies the corresponding variance estimator $T\widehat{\bm{\sigma}}_{a}^{2}=O_{P}(1)$
and $N\widehat{\bm{\sigma}}_{d}^{2}=O_{P}(1)$. Hence, $r_{NT,f}^2=NT$
 and $\bm{\nu}_{a,0}=\bm{\nu}_{d,0}=\bm{0}$. For
the bootstrap term $\widehat{\bm{Z}}_{N}^{a*}$, we have 
\begin{align}\label{eq:vanish-a}
E^{*}\bigl\|\widehat{\bm{Z}}_{N}^{a*}\bigr\|^{2} & =\Bigl\|\frac{r_{NT,f}}{N}\,\bm{\sigma}_{a,f}\Bigl(\frac{1}{N}\sum_{j=1}^{N}\ddot{\bm{a}}_{j}\ddot{\bm{a}}_{j}^{\top}\Bigr)^{-1/2}\Bigl(\frac{1}{N}\sum_{i=1}^{N}\ddot{\bm{a}}_{i}\ddot{\bm{a}}_{i}^{\top}\Bigr)\Bigl(\frac{1}{N}\sum_{j=1}^{N}\ddot{\bm{a}}_{j}\ddot{\bm{a}}_{j}^{\top}\Bigr)^{-1/2}\bm{\sigma}_{a,f}r_{NT,f}\Bigr\|\nonumber \\
 & =O_{P}\left(\bigl\| T\bm{\sigma}_{a,f}^{2}\bigr\|\right)+o_{P}(1)=o_{P}(1),
\end{align}
where we used $E^{*}(\eta_{i}^{*b}\eta_{j}^{*b})=\mathbb{I}\{i=j\}$
and the normalization in $\widehat{\bm{a}}_{i}$. Thus $\widehat{\bm{Z}}_{N}^{a*}=o_{P^{*}}(1)$,
and similarly $\widehat{\bm{Z}}_{T}^{d*}=o_{P^{*}}(1)$.

 For each fixed $l,l'\le L$, a similar argument as Lemma \ref{lemma: clt for bootstrap Za} yields that
\begin{equation}
\widehat{\bm{Z}}_{N,l}^{*\phi}\xrightarrow{d^{*}}\bm{Z}_{l}^{\phi},\label{eq:phi-star}
\end{equation}
conditional on the data. For $\widehat{\bm{Z}}_{T,l'}^{*\psi}$, use
the fact that $\{\eta_{t}^{*b}\}$ is a bounded Markov sign process
with $E^{*}(\eta_{t}^{*b}\eta_{t'}^{*b})=q^{|t-t'|}$; the same martingale-CLT
argument as in proof for Lemma \ref{lemma: clt for bootstrap Ze}
yields 
\begin{equation}
\widehat{\bm{Z}}_{T,l'}^{*\psi}\xrightarrow{d^{*}}\bm{Z}_{l'}^{\psi},\label{eq:psi-star}
\end{equation}
conditional on the data.

Moreover, consider the bootstrap $e$-component $\widehat{\bm{Z}}_{NT}^{e*}$.
By Lemma~\ref{lemma: clt for bootstrap Ze}, conditional on the data,
we have 
\[
\bm{\sigma}_{e}^{-1}\frac{1}{\sqrt{NT}}\sum_{i=1}^{N}\sum_{t=1}^{T}\widehat{\bm{e}}_{it}\,\eta_{i}^{*b}\eta_{t}^{*b}\;\overset{d^{*}}{\longrightarrow}\;\mathcal{N}\!\left(\bm{0},\mathbf{I}\right).
\]
Let $\bm{\nu}_{NT,ef}^{*}\;\equiv\;\bigl[Var^{*}(\widehat{\bm{Z}}_{NT}^{e*})\bigr]^{1/2}.$
Then, conditional on the data, 
\begin{equation}
\bm{\nu}_{NT,ef}^{*}\;\xrightarrow{P}\;\bm{\nu}_{e,0},\qquad\widehat{\bm{Z}}_{NT}^{e*}\;\overset{d^{*}}{\longrightarrow}\;\bm{\nu}_{e,0}\bm{Z}^{e}.\label{eq:e-star-limit}
\end{equation}

By a similar argument, one can deduce that $\widehat{\bm{Z}}_{NT}^{c*}$
is negligible given that $\widehat{\bm{\beta}}-\bm{\beta}_{0}=o_{P}(1)$,
\begin{equation}
\widehat{\bm{Z}}_{NT}^{c*}=o_{P^{*}}(1).\label{eq:c-negl}
\end{equation}

Finally, covariance components are negligible: for each fixed $l$,
\begin{equation}
\|\bm{\sigma}_{al,f}\|=\|E(\bm{a}_{i}\phi_{l}(\bm{\alpha}_{i})^{\top})\|\le\sqrt{E\|\bm{a}_{i}\|^{2}\,E\|\phi_{l}(\bm{\alpha}_{i})\|^{2}}=O\!\left(\|\bm{\sigma}_{a,f}^{2}\|^{1/2}\right)=o(T^{-1/2}),\label{eq:vanish-cov}
\end{equation}
hence $\bm{\nu}_{al,0}=\lim_{N,T\to\infty}r_{NT,f}\bm{\sigma}_{al,f}/\sqrt{N}=\bm{0}$,
and similarly $\bm{\nu}_{dl,0}=\bm{0}$. Therefore, in the vanishing
regime the bootstrap components $\widehat{\bm{Z}}_{N}^{a*}$ and $\widehat{\bm{Z}}_{T}^{d*}$,
as well as the covariance terms with $(\widehat{\bm{Z}}_{N,l}^{*\phi},\widehat{\bm{Z}}_{T,l'}^{*\psi})$,
are asymptotically negligible, matching the target law \eqref{eq:L0}
(marginally).

\medskip{}
\underline{\textbf{Scenario 2 (Intermediate):}} without loss of generality,
assume $T\bm{\sigma}_{a,f}^{2}\to\bm{\varphi}_{a}>0$, and neither $T\bm{\sigma}_{a,f}^{2}$
nor $N\bm{\sigma}_{d,f}^{2}$ diverges. Then the term $\bm{\nu}_{a,0}\bm{Z}^{a}$
is non-negligible. We first match its bootstrap marginal distribution
via variance consistency and a conditional CLT. Let $\bm{\nu}_{NT,1f}^{*}=[Var^{*}(\widehat{\bm{Z}}_{N}^{a*})]^{1/2}$.
A direct calculation as \eqref{eq:vanish-a} yields 
\begin{equation}
Var^{*}(\widehat{\bm{Z}}_{N}^{a*})\xrightarrow{P}\bm{\nu}_{a,0}\bm{\nu}_{a,0}^{\top},\qquad\bm{\nu}_{NT,1f}^{*}\xrightarrow{P}\bm{\nu}_{a,0}.\label{eq:interm-var-a}
\end{equation}
Applying Lemma \ref{lemma: clt for bootstrap Za} with Slutsky's Lemma gives 
\begin{equation}
\widehat{\bm{Z}}_{N}^{a*}\xrightarrow{d^{*}}\bm{\nu}_{a,0}\bm{Z}^{a},\label{eq:interm-clt-a}
\end{equation}
conditional on the data. Likewise, with $\bm{\nu}_{NT,df}^{*}=[Var^{*}(\widehat{\bm{Z}}_{T}^{d*})]^{1/2}$,
\begin{equation}
\bm{\nu}_{NT,df}^{*}\xrightarrow{P}\bm{\nu}_{d,0},\qquad\widehat{\bm{Z}}_{T}^{d*}\xrightarrow{d^{*}}\bm{\nu}_{d,0}\bm{Z}^{d},\label{eq:interm-clt-d}
\end{equation}
conditional on the data (the proof uses the dependence in $\eta_{t}^{*b}$
through $E^{*}(\eta_{t}^{*b}\eta_{t'}^{*b})=q^{|t-t'|}$,  Lemma \ref{lemma: convergence in covariance},  and the same argument as proof for Lemma~\ref{lemma: clt for bootstrap Ze}). The remaining marginal components $\widehat{\bm{Z}}_{NT}^{e*}$
and $\{(\widehat{\bm{Z}}_{N,l}^{*\phi},\widehat{\bm{Z}}_{T,l'}^{*\psi})\}_{l,l'\le L}$
will be handled in a similar manner.

Finally, in the intermediate regime we must match the covariance terms
$Cov^{*}(\widehat{\bm{Z}}_{N}^{a*},\widehat{\bm{Z}}_{N,l}^{*\phi})$
and $Cov^{*}(\widehat{\bm{Z}}_{T}^{d*},\widehat{\bm{Z}}_{T,l'}^{*\psi})$.
For the $d$--$\psi$ pair, using $E^{*}(\eta_{t}^{*b}\eta_{t'}^{*b})=q^{|t-t'|}$,
\begin{align}
Cov^{*}\!\left(\widehat{\bm{Z}}_{T}^{d*},\widehat{\bm{Z}}_{T,l'}^{*\psi}\right) & =\frac{r_{NT,f}}{\sqrt{T}}\frac{1}{T}\sum_{t=1}^{T}\sum_{t'=1}^{T}q^{|t-t'|}\widehat{\bm{d}}_{t}\Bigl(\psi_{l'}(\bm{\xi}_{t'})-\frac{1}{T}\sum_{\tau=1}^{T}\psi_{l'}(\bm{\xi}_{\tau})\Bigr)^{\!\top}\xrightarrow{P}\bm{\nu}_{dl',0},\label{eq:cov-dpsi}
\end{align}
where the convergence follows from Lemma~\ref{lemma: convergence in covariance}
and $\bm{\nu}_{dl',0}=\lim_{N,T\to\infty}r_{NT,f}\bm{\sigma}_{dl',f}/\sqrt{T}$.
Similarly, since $\eta_{i}^{*b}$ are i.i.d.\ with mean zero over
$i$, 
\begin{equation}
Cov^{*}\!\left(\widehat{\bm{Z}}_{N}^{a*},\widehat{\bm{Z}}_{N,l}^{*\phi}\right)\xrightarrow{P}\bm{\nu}_{al,0}.\label{eq:cov-aphi}
\end{equation}

\medskip{}
\underline{\textbf{Scenario 3 (Diverging):}} if $T\bm{\sigma}_{a,f}^{2}$
or $N\bm{\sigma}_{d,f}^{2}$ diverges, then $\bm{\nu}_{e,0}=\bm{0}$
and $\bm{\nu}_{v,0}=\bm{0}$ and the limit is Gaussian. For instance,
if $N\bm{\sigma}_{d,f}^{2}$ dominates, then 
\begin{equation}
Var^{*}\!\left(\widehat{\bm{Z}}_{T}^{d*}\right)=\frac{r_{NT,f}}{\sqrt{T}}\Bigl(\frac{1}{T}\sum_{t=1}^{T}\sum_{t'=1}^{T}q^{|t-t'|}\widehat{\bm{d}}_{t}\widehat{\bm{d}}_{t'}^{\top}\Bigr)\frac{r_{NT,f}}{\sqrt{T}}\xrightarrow{P}\bm{\nu}_{d,0}\bm{\nu}_{d,0}^{\top},\label{eq:div-var-d}
\end{equation}
by Lemma~\ref{lemma: convergence in covariance}. Hence, we have $\bm{\nu}_{NT,df}^{*}\xrightarrow{P}\bm{\nu}_{d,0}$
and $\widehat{\bm{Z}}_{T}^{d*}\xrightarrow{d^{*}}\bm{\nu}_{d,0}\bm{Z}^{d}$ by the bootstrap CLT for multipliers $\eta_t^*$ (Lemma \ref{lemma: clt for bootstrap Ze}).
A symmetric argument applies when $T\bm{\sigma}_{a,f}^{2}$ dominates by the bootstrap CLT for multipliers $\eta_i^*$ (Lemma \ref{lemma: clt for bootstrap Za}).

 This completes Step~1: for each fixed truncation level
$L<\infty$, the bootstrap matches the marginal limits of the non-negligible
components and the required pairwise covariance limits (those involving
$a$--$\phi$ and $d$--$\psi$).

\paragraph{Step 2: Joint CLT for the truncated bootstrap vector.}

Fix a truncation level $L<\infty$ and collect the bootstrap components
into 
\[
\widehat{\bm{S}}_{NT}^{*}(L)=\left(\widehat{\bm{S}}_{NT,1}^{*}(L),\widehat{\bm{S}}_{NT,2}^{*}(L)\right)\equiv\Bigl(\left(\widehat{\bm{Z}}_{N}^{a*},\;\widehat{\bm{Z}}_{NT}^{e*},\;\{\widehat{\bm{Z}}_{N,l}^{*\phi}\}_{l=1}^{L}\right),\left(\widehat{\bm{Z}}_{T}^{d*},\;\{\widehat{\bm{Z}}_{T,l'}^{*\psi}\}_{l'=1}^{L}\right)\Bigr).
\]
Given that the $T$-indexed terms $\widehat{\bm{Z}}_{T}^{d*}$ and
$\{\widehat{\bm{Z}}_{T,l'}^{*\psi}\}_{l'\le L}$ do \emph{not} contain
the $i$-multipliers $\{\eta_{i}^{*b}\}_{i}$, we proceed in two stages:
(a) obtain a joint CLT for all components driven by the $i$-multipliers
$\{\eta_{i}^{*b}\}_{i}$ conditional on the $t$-multipliers $\{\eta_{t}^{*b}\}_{t}$;
(b) combine this with the joint CLT for the $t$-multiplier block
using independence between the two multiplier sequences.

\subparagraph{a) Joint CLT for the $i$-multiplier block conditional on $\{\eta_{t}^{*b}\}_{t\le T}$.}

Define the $\sigma$-field $\mathcal{G}_{T}\equiv\sigma(\{\eta_{t}^{*b}\}_{t})$
and focus on the vector of bootstrap terms with $\{\eta_{i}^{*b}\}_{i}$:
$\widehat{\bm{S}}_{NT,1}^{*}(L)$. Rewrite the interaction component
as a triangular-array sum in $i$: 
\[
\widehat{\bm{Z}}_{NT}^{e*}=\frac{r_{NT,f}}{\sqrt{N}}\sum_{i=1}^{N}\eta_{i}^{*b}\,\bm{b}_{iT}^{*},\qquad\bm{b}_{iT}^{*}\equiv\frac{1}{\sqrt{T}}\sum_{t=1}^{T}\widehat{\bm{e}}_{it}\,\eta_{t}^{*b},
\]
and similarly, for each $l\le L$, 
\[
\widehat{\bm{Z}}_{N,l}^{*\phi}=\frac{1}{\sqrt{N}}\sum_{i=1}^{N}\eta_{i}^{*b}\,\bm{u}_{i,l},\qquad\bm{u}_{i,l}\equiv\phi_{l}(\bm{\alpha}_{i})-\frac{1}{N}\sum_{j=1}^{N}\phi_{l}(\bm{\alpha}_{j}),
\]
Therefore, for any conformable deterministic $\bm{\lambda}_{1}$,
conditional on $\mathcal{G}_{T}$ (and the data), $\bm{\lambda}_{1}^{\top}\widehat{\bm{S}}_{NT,1}^{*}(L)$
can be written as a sum of $N$ independent (due to $\eta_{i}^{*}$)
terms. Since $\{\eta_{i}^{*b}\}_{i}$ are i.i.d.\ with $E^{*}(\eta_{i}^{*b})=0$
and $E^{*}((\eta_{i}^{*b})^{2})=1$ and are independent of $\mathcal{G}_{T}$,
a conditional Lindeberg--Feller CLT applies given $\mathcal{G}_{T}$
once we verify variance convergence and a Lindeberg condition.

Define conditional variance of 
\[
\omega_{NT,1}^{2}(\bm{\lambda}_{1};\mathcal{G}_{T})\equiv Var^{*}\!\left(\bm{\lambda}_{1}^{\top}\widehat{\bm{S}}_{NT,1}^{*}(L)\mid\mathcal{G}_{T}\right).
\]
Using the marginal conditional limits for $\widehat{\bm{Z}}_{N}^{a*}$
\eqref{eq:interm-clt-a}, $\widehat{\bm{Z}}_{NT}^{e*}$ \eqref{eq:e-star-limit},
and $\{\widehat{\bm{Z}}_{N,l}^{*\phi}\}_{l\le L}$ \eqref{eq:phi-star},
together with the covariance limits for the shared $i$-multipliers
(the $a$--$\phi$ block) in \eqref{eq:cov-aphi}, we obtain 
\begin{equation}
\omega_{NT,1}^{2}(\bm{\lambda}_{1};\mathcal{G}_{T})\xrightarrow{P^{*}}\omega_{1}^{2}(\bm{\lambda}_{1}),\label{eq:joint clt 1}
\end{equation}
where $\omega_{1}^{2}(\bm{\lambda}_{1})=\bm{\lambda}_{1}^{\top}Var\left(\bm{S}_{1}(L)\right)\bm{\lambda}_{1}$,
with $\bm{S}_{1}(L)\equiv\left\{\bm{\nu}_{a,0}\bm{Z}^{a},\bm{\nu}_{e,0}\bm{Z}^{e},\left\{ \bm{Z}_{l}^{\phi}\right\} _{l=1}^{L}\right\} $.
Moreover, the Lindeberg condition holds given $\mathcal{G}_{T}$,
and hence the bootstrap CLT result implies that 
\begin{equation}
E^{*}\!\left(\exp\!\left\{ i\,\bm{\lambda}_{1}^{\top}\widehat{\bm{S}}_{NT,1}^{*}(L)\right\} \Bigm|\mathcal{G}_{T}\right)=\exp\!\left\{ -\tfrac{1}{2}\,\omega_{NT,1}^{2}(\bm{\lambda}_{1};\mathcal{G}_{T})\right\} +o_{P}(1).\label{eq: characteristic first}
\end{equation}

\subparagraph{b) Joint CLT for the $t$-multiplier block and combination by independence.}

Next, consider the $t$-multiplier block 
\[
\widehat{\bm{S}}_{NT,2}^{*}(L)\equiv\Bigl(\widehat{\bm{Z}}_{T}^{d*},\{\widehat{\bm{Z}}_{T,l'}^{*\psi}\}_{l'=1}^{L}\Bigr),
\]
which depends only on $\{\eta_{t}^{*b}\}_{t}$ (and the data), not
on the $i$-multipliers. By the marginal conditional CLTs \eqref{eq:psi-star},
\eqref{eq:interm-clt-d}, and the covariance limits for the shared
$t$-multipliers (the $d$--$\psi$ block) in \eqref{eq:cov-dpsi},
the vector $\widehat{\bm{S}}_{NT,2}^{*}(L)$ converges conditionally
to $\bm{S}_{2}(L)\equiv\left\{ \bm{\nu}_{d,0}\bm{Z}^{d},\left\{ \bm{Z}_{l'}^{\psi}\right\} _{l'=1}^{L}\right\} $.
By continuous mapping theorem, it follows that 
\begin{equation}
\exp\!\left\{ i\,\bm{\lambda}_{2}^{\top}\widehat{\bm{S}}_{NT,2}^{*}(L)\right\} \overset{d^{*}}{\to}\exp\!\left\{ i\,\bm{\lambda}_{2}^{\top}\bm{S}_{2}(L)\right\} .\label{eq:joint clt 2}
\end{equation}

Finally, note that conditional on the data, the two multiplier sequences
$\{\eta_{i}^{*b}\}_{i}$ and $\{\eta_{t}^{*b}\}_{t}$ are independent.
Let $\bm{\lambda}_{2}$ be conformable deterministic vectors and consider
the characteristic function (conditional on data) of $\bm{\lambda}_{1}^{\top}\widehat{\bm{S}}_{NT,1}^{*}(L)+\bm{\lambda}_{2}^{\top}\widehat{\bm{S}}_{NT,2}^{*}(L)$,
$\phi_{NT}^{*}(\bm{\lambda}_{1},\bm{\lambda}_{2})$. By law of iterated
expectation and \eqref{eq: characteristic first}, we have 
\begin{align}
\phi_{NT}^{*}(\bm{\lambda}_{1},\bm{\lambda}_{2}) & =E^{*}\!\left[\exp\!\left\{ i\,\bm{\lambda}_{2}^{\top}\widehat{\bm{S}}_{NT,2}^{*}(L)\right\} \cdot E^{*}\!\left(\exp\!\left\{ i\,\bm{\lambda}_{1}^{\top}\widehat{\bm{S}}_{NT,1}^{*}(L)\right\} \Bigm|\mathcal{G}_{T}\right)\right],\nonumber \\
 & =E^{*}\!\left[\exp\!\left\{ i\,\bm{\lambda}_{2}^{\top}\widehat{\bm{S}}_{NT,2}^{*}(L)\right\} \cdot\left(\exp\!\left\{ -\tfrac{1}{2}\,\omega_{NT,1}^{2}(\bm{\lambda}_{1};\mathcal{G}_{T})\right\} +o_{P}(1)\right)\right].\label{eq:joint clt 3}
\end{align}

Combining \eqref{eq:joint clt 1} and \eqref{eq:joint clt 2}, by
the continuous mapping theorem, one can deduce 
\[
\exp\!\left\{ i\,\bm{\lambda}_{2}^{\top}\widehat{\bm{S}}_{NT,2}^{*}(L)\right\} \cdot\left(\exp\!\left\{ -\tfrac{1}{2}\,\omega_{NT,1}^{2}(\bm{\lambda}_{1};\mathcal{G}_{T})\right\} +o_{P}(1)\right)\xrightarrow{d^{*}}\exp\!\left\{ i\,\bm{\lambda}_{2}^{\top}\bm{S}_{2}(L)\right\} \exp\!\left\{ -\tfrac{1}{2}\omega_{1}^{2}(\bm{\lambda}_{1})\right\} .
\]
Applying Portmanteau Lemma with \eqref{eq:joint clt 3} gives 
\[
\phi_{NT}^{*}(\bm{\lambda}_{1},\bm{\lambda}_{2})\overset{P}{\to}E\left(\exp\!\left\{ i\,\bm{\lambda}_{2}^{\top}\bm{S}_{2}(L)\right\} \exp\!\left\{ -\tfrac{1}{2}\omega_{1}^{2}(\bm{\lambda}_{1})\right\} \right)=\exp\!\left\{ -\tfrac{1}{2}\omega_{1}^{2}(\bm{\lambda}_{1})-\tfrac{1}{2}\omega_{2}^{2}(\bm{\lambda}_{2})\right\} ,
\]
where $\omega_{2}^{2}(\bm{\lambda}_{2})=\bm{\lambda}_{2}^{\top}Var\left(\bm{S}_{2}(L)\right)\bm{\lambda}_{2}$.
Since this holds for all $(\bm{\lambda}_{1},\bm{\lambda}_{2})$, LÃ©vy's
continuity theorem implies 
\[
\widehat{\bm{S}}_{NT}^{*}(L)\xrightarrow{d^{*}}\left(\bm{S}_{1}(L),\bm{S}_{2}(L)\right),
\]
which is the desired joint conditional CLT for the truncated bootstrap
vector. The application of the Cramér-Wold Device implies that
\begin{equation}
\left\Vert P_{NT,f}^{*}\left(r_{NT,f}\left(\frac{1}{NT}\sum_{i=1}^{N}\sum_{t=1}^{T}\bm{s}_{it}^{*}\right)\right)-\mathcal{\bm{L}}_{L}\left(\bm{\nu}_{0},\left\{ \bm{c}_{ll',f}\right\} _{l,l'=1}^{L}\right)\right\Vert _{\infty}\xrightarrow{P}0,\label{eq:joint-star-L}
\end{equation}

Finally, by square summability of $\{\bm{c}_{ll',0}\}_{l,l'\ge1}$, there
exists $L<\infty$ such that the truncation error is negligible, and
hence 
\begin{equation}
\Bigl\|\mathcal{\bm{L}}_{0}\!\left(\bm{\nu}_{0},\{\bm{c}_{ll',0}\}_{l,l'=1}^{\infty}\right)-\mathcal{\bm{L}}_{L}\!\left(\bm{\nu}_{0},\{\bm{c}_{ll',0}\}_{l,l'=1}^{L}\right)\Bigr\|_{\infty}\xrightarrow{P}0.\label{eq:trunc-L0}
\end{equation}
Combining \eqref{eq:joint-star-L}-\eqref{eq:trunc-L0} with the triangular
inequality yields \eqref{eq:boot-goal}.

\paragraph{Step 3: Continuity of the limit and uniformity over $f$.}

As shown in Step~3 of Lemma~\ref{lemma: original clt}, $\mathcal{\bm{L}}_{0}$
is continuous in $(\bm{\nu}_{0},\{\bm{c}_{ll',0}\}_{l,l'\ge1})$. We have
shown that for any convergent sequence $\bm{\theta}_{NT,f}=(\bm{\nu}_{NT,f},\{\bm{c}_{ll',f}\}_{l,l'\ge1})\to\bm{\theta}_{0}=(\bm{\nu}_{0},\{\bm{c}_{ll',0}\}_{l,l'\ge1})$,
(i) the original CLT \eqref{eq:orig-CLT} holds, (ii) the bootstrap
CLT \eqref{eq:boot-goal} holds, and (iii) the limiting law is continuous.
By Assumptions~\ref{as: spectral representation}--\ref{as: same rate},
the parameter space of $\bm{\theta}_{NT,f}$ is compact, so every
sequence admits a convergent subsequence (Bolzano--Weierstrass).
Therefore Lemma~\ref{lemma: uniform convergence} implies 
\[
\sup_{f\in\mathcal{B}}\Bigl\| P_{NT,f}\!\Bigl(r_{NT,f}\bigl(\widehat{\bm{\beta}}-\bm{\beta}_{0}\bigr)\Bigr)-\mathcal{\bm{L}}_{0}\bigl(\bm{\theta}_{NT}(f)\bigr)\Bigr\|_{\infty}\xrightarrow{P}0,
\]
and 
\[
\sup_{f\in\mathcal{B}}\Bigl\| P_{NT,f}^{*}\!\Bigl(r_{NT,f}\bigl(\widehat{\bm{\beta}}^{*}-\widehat{\bm{\beta}}\bigr)\Bigr)-\mathcal{\bm{L}}_{0}\bigl(\bm{\theta}_{NT}(f)\bigr)\Bigr\|_{\infty}\xrightarrow{P}0.
\]
The application of the  triangular inequality and Cramér-Wold yield
\[
\sup_{f\in\mathcal{B}}\Bigl\| P_{NT,f}\!\Bigl(r_{NT,f}\bigl(\widehat{\bm{\beta}}-\bm{\beta}_{0}\bigr)\Bigr)-P_{NT,f}^{*}\!\Bigl(r_{NT,f}\bigl(\widehat{\bm{\beta}}^{*}-\widehat{\bm{\beta}}\bigr)\Bigr)\Bigr\|_{\infty}\xrightarrow{P}0.
\] If the diagonal entries of
\(\bm{\sigma}_{a,f}^{2}\), \(\bm{\sigma}_{d,f}^{2}\), and
\(\bm{\sigma}_{v,f}^{2}\) are allowed to have heterogeneous orders across
coordinates, the same conclusion follows by applying the joint CLT
 and truncation argument as in Step~2.
\end{proof}

\subsection{Proof of Theorem \protect\ref{thm: main-2}}
\begin{proof}
\textbf{PWB-D.} Without loss of generality, assume that $T\sigma_{ak,f}^{2}$
diverges at a rate no slower than $N\sigma_{dk,f}^{2}$. Recall that
$\{\eta_{i}^{*}\}_{i=1}^{N}$ are i.i.d.\ with $E^{*}(\eta_{i}^{*}\eta_{j}^{*})=\mathbb{I}\{i=j\}$.
Writing 
\[
\widehat{\bm{Z}}_{N}^{a*}=\frac{\widehat{r}_{NT}}{\sqrt{N}}\frac{1}{\sqrt{N}}\sum_{i=1}^{N}\widehat{\bm{\vartheta}}_{a}\ddot{\bm{a}}_{i}\,\eta_{i}^{*},\qquad\widehat{\bm{Q}}_{a}\equiv\frac{1}{N}\sum_{j=1}^{N}\ddot{\bm{a}}_{j}\ddot{\bm{a}}_{j}^{\top},
\]
and using iterated expectation yields 
\begin{align*}
E^{*}\!\left(\widehat{\bm{Z}}_{N}^{a*}\widehat{\bm{Z}}_{N}^{a*\top}\right) & =\frac{\widehat{r}_{NT}}{N}\frac{1}{N}\sum_{i=1}^{N}\widehat{\bm{\vartheta}}_{a}\ddot{\bm{a}}_{i}\ddot{\bm{a}}_{i}^{\top}\widehat{\bm{\vartheta}}_{a}^{\top}\widehat{r}_{NT}\\
 & =\frac{\widehat{r}_{NT}}{N}\,\bm{D}_{a}\!\left(\widehat{\bm{\mu}}_{a}\right)\widehat{\bm{\sigma}}_{a}\,\widehat{\bm{Q}}_{a}^{-1/2}\Bigl(\frac{1}{N}\sum_{i=1}^{N}\ddot{\bm{a}}_{i}\ddot{\bm{a}}_{i}^{\top}\Bigr)\widehat{\bm{Q}}_{a}^{-1/2}\widehat{\bm{\sigma}}_{a}\widehat{r}_{NT}\;+\;o_{P}(1)\\
 & =\frac{\widehat{r}_{NT}}{N}\,\bm{D}_{a}\!\left(\widehat{\bm{\mu}}_{a}\right)\widehat{\bm{\sigma}}_{a}^{2}\widehat{r}_{NT}\;+\;o_{P}(1),
\end{align*}
where the second line uses the definition of $\widehat{\bm{\vartheta}}_{a}$
and the last line uses $\widehat{\bm{Q}}_{a}^{-1/2}\widehat{\bm{Q}}_{a}\widehat{\bm{Q}}_{a}^{-1/2}=\mathbf{I}$.
Hence, the $k$th diagonal element satisfies 
\[
Var^{*}\!\left(\widehat{Z}_{Nk}^{a*}\right)=\,\frac{1}{N}\mathbb{I}\!\left\{ \widehat{\sigma}_{a,k}^{2}\ge\widehat{\mu}_{ak,D}\right\} \widehat{r}_{NT}\widehat{{\sigma}}_{ak}^{2}\widehat{r}_{NT}\;+\;o_{P}(1),\qquad\widehat{\mu}_{ak,D}\equiv\frac{\log T}{T}.
\]
Applying Lemma~\ref{lemma:limit form-1} gives the following cases:

 \emph{(i) D:} when $T\sigma_{ak,f}^{2}>2\log T$, then under Assumption \ref{as: same rate}, we have $\widehat{\sigma}_{ak}\sigma_{ak,f}^{-1}\overset{P}{\to}1$,    $P(\widehat{r}_{NT}^2=N\sigma_{ak,f}^{-2})\to1$, 
and $P(\widehat{\sigma}_{a,k}^{2}>\widehat{\mu}_{ak,D})\to1$, so
the bootstrap $a$-component reproduces the Gaussian limit $(\bm{\nu}_{a,0}\bm{Z}^{a})_{k}$.

\emph{(ii) V\&G:} when $T\sigma_{ak,f}^{2}=o(1)$ and $\sigma_{vk,f}^{2}=o(1)$,
then $(\bm{\nu}_{a,0}\bm{Z}^{a})_{k}$ is negligible in the original limit, $T\widehat{\sigma}_{ak}^{2}=o_{P}(1)$,
and $\widehat{{\sigma}}_{NT,k}^{-2}/NT\overset{P}{\to}1$. Consequently,
\[
E^{*}\!\left|\widehat{Z}_{Nk}^{a*}\right|^{2}=Var^{*}\!\left(\widehat{Z}_{Nk}^{a*}\right)=T\widehat{\sigma}_{ak}^{2}+o_{P}(1)=o_{P}(1),
\]
which implies $\widehat{Z}_{Nk}^{a*}=o_{P^{*}}(1)$, matching the corresponding limit.

\emph{(iii) V\&N:} when $T\sigma_{ak,f}^{2}=o(1)$ and $\sigma_{vk,f}^{2}>0$,
then $(\bm{\nu}_{a,0}\bm{Z}^{a})_{k}$ is again negligible, and Lemma~\ref{lemma:limit form-1}
yields $T\widehat{\sigma}_{ak}^{2}=O_{P}(1)$ such that $P(\widehat{\sigma}_{a,k}^{2}<\widehat{\mu}_{ak,D})\to1$.
Hence $\widehat{Z}_{Nk}^{a*}=o_{P^{*}}(1)$.

The same argument applies to the $d$-component $\widehat{Z}_{Tk}^{d*}$
(with $N$ and $T$ interchanged). Moreover, for indices satisfying
$\widehat{\sigma}_{a,k}^{2}>\widehat{\mu}_{ak,D}$, the bootstrap
preserves cross-$k$ covariance and covariance terms such as $Cov^{*}(\widehat{Z}_{Nk}^{a*},\widehat{Z}_{Nk,l}^{*\phi})$
automatically because the resampling multipliers are common across
components. The remaining steps are identical to those for the oracle
PWB and are omitted.

\textbf{PWB-V.} Without loss of generality, assume $T\sigma_{ak,f}^{2}$
dominates $N\sigma_{dk,f}^{2}$. The same calculation as above gives
\[
Var^{*}\!\left(\widehat{Z}_{Nk}^{a*}\right)=T\,\mathbb{I}\!\left\{ \widehat{\sigma}_{a,k}^{2}\ge\widehat{\mu}_{ak,V}\right\} \widehat{\sigma}_{ak}^{2}\;+\;o_{P}(1),\qquad\widehat{\mu}_{ak,V}\equiv\frac{1}{T\log T}.
\]
Therefore, the application of Lemma~\ref{lemma:limit form-1} gives the following cases:

\emph{(i) D:} if $T\sigma_{ak,f}^{2}\to\infty$, we have $\widehat{\sigma}_{ak}/\sigma_{ak,f}\overset{P}{\to}1$
and $P(\widehat{\sigma}_{a,k}^{2}>\widehat{\mu}_{ak,V})\to1$, so
the bootstrap matches $(\bm{\nu}_{a,0}\bm{Z}^{a})_{k}$.

\emph{(ii) I\&G:} when $T\sigma_{ak,f}^{2}\to\varphi_{ak}\in(0,\infty)$
and $\sigma_{vk,f}^{2}=o(1)$, we have again that $\widehat{\sigma}_{ak}/\sigma_{ak,f}\overset{P}{\to}1$
and $P(\widehat{\sigma}_{a,k}^{2}>\widehat{\mu}_{ak,V})\to1$, so
the bootstrap matches $(\bm{\nu}_{a,0}\bm{Z}^{a})_{k}$.

\emph{(iii) V\&G:} when $T\sigma_{ak,f}^{2}=o(1)$ and $\sigma_{vk,f}^{2}=o(1)$,
we have that $(\bm{\nu}_{a,0}\bm{Z}^{a})_{k}$ is negligible and $T\widehat{\sigma}_{ak}^{2}=o_{P}(1)$,
implying $\widehat{Z}_{Nk}^{a*}=o_{P^{*}}(1)$. In this regime the
thresholding indicator is immaterial for the limiting argument.

The same conclusions hold for $\widehat{Z}_{Tk}^{d*}$. Furthermore,
cross-$k$ covariances are preserved whenever $\widehat{\sigma}_{a,k}^{2}>\widehat{\mu}_{ak,V}$.
In all regimes treated here, the limiting distribution is Gaussian,
completing the proof. 
\end{proof}

\subsection{Proof of Theorem \protect\ref{thm: main-3}}
\begin{proof}
Observe first that the PWB-V procedure is asymptotically valid in
all scenarios under which the limiting distribution of the statistic
is Gaussian, whereas in the remaining scenarios the limiting distribution
is non-Gaussian in probability. In contrast, the PWB-D procedure is
asymptotically valid in the regime where $T\sigma_{ak,f}^{2}=o(1)$
and $\sigma_{vk,f}^{2}>0$, in which case the limiting distribution
of the original statistic is non-Gaussian.

Therefore, the desired result follows if the data-driven selector
$D_{k}^{*}$ satisfies 
\[
P\left(D_{k}^{*}=0\right)\to1\quad\text{whenever the limiting distribution is Gaussian},
\]
and 
\[
P\left(D_{k}^{*}=1\right)\to1\quad\text{whenever the limiting distribution is non-Gaussian}.
\]
Under these conditions, the proposed procedure asymptotically selects
PWB-V in Gaussian regimes and PWB-D in non-Gaussian regimes, thereby
ensuring validity in all four scenarios. It suffices to show that
for PWB-V, the form of limiting distribution of the bootstrap statistic
matches that of the original statistic in all scenarios, which we
now demonstrate.

Without loss of generality, suppose that $T\sigma_{ak,f}^{2}$ diverges
at a rate no slower than $N\sigma_{dk,f}^{2}$. Observe that the only
potential source of non-Gaussianity comes from the cross-product component
\[
\frac{1}{\sqrt{NT}}\sum_{l,l'=1}^{\infty}c_{ll'k,f}\,Z_{Nk,l}^{\phi}Z_{Tk,l'}^{\psi}\;=\;O_{P}\!\left(\frac{\sigma_{vk,f}}{\sqrt{NT}}\right),
\]
appearing in the decomposition~\eqref{eq: decompose sit}. By the
same arguments used in the proof of Theorem~\ref{thm: main}, there
are three regimes in which the limit is Gaussian and PWB-V consistently
reproduces the corresponding limiting law, namely D, I\&G, and V\&G.

It therefore remains to show that PWB-V yields a \emph{non-Gaussian}
bootstrap limit in the two non-Gaussian regimes, I\&N and V\&N. In
either regime, the bootstrap counterpart of the cross-product term,
$\frac{1}{\sqrt{NT}}\sum_{l,l'=1}^{\infty}c_{ll'k,f}\,\widehat{Z}_{Nk,l}^{*\phi}\widehat{Z}_{Tk,l'}^{*\psi},$
is of the same stochastic order as $\frac{1}{\sqrt{NT}}\sum_{l,l'=1}^{\infty}c_{ll'k,f}Z_{Nk,l}^{\phi}Z_{Tk,l'}^{\psi}$
as shown in the proof of Theorem~\ref{thm: main}, while all remaining
terms in the bootstrap expansion are $O_{P}\!\left((NT)^{-1/2}\right)$
and hence asymptotically negligible relative to the leading cross-product
contribution. Consequently, the bootstrap limiting distribution is
non-Gaussian in I\&N and V\&N, as required.
\end{proof}

\singlespacing

 \bibliographystyle{chicago}
\bibliography{multiway_clustering}

\newpage
\title{Internet Appendix for ``Bootstrap Inference under General Two-way Clustering with Serially and Spatially Dependent Common Effects''}
\author{ {\large \textbf{Ulrich Hounyo}} \and {\large \textbf{Jiahao Lin}%
 }}


\maketitle

Appendix IA collects five technical lemmas used in the main text. Appendix IB presents the propositions on discriminant factors and the impossibility result. Appendix IC reports results on heterogeneous intersection sizes, and Appendix ID provides the residual-based PWB approach and additional simulation results.

%


\appendix

\section*{Appendix IA: Technical Lemmas}

\setcounter{equation}{0} \setcounter{assumption}{0}
\setcounter{figure}{0} \setcounter{table}{0}

 \renewcommand{%
\theequation}{IA.\arabic{equation}}
\renewcommand{\thelemma}{IA.\arabic{lemma}}
\renewcommand{\theassumption}{IA.\arabic{assumption}} \renewcommand{%
\thetheorem}{IA.\arabic{theorem}}

\renewcommand{\thetable}{IA.\arabic{table}}
\renewcommand{\thefigure}{IA.\arabic{figure}}

\renewcommand{\thesubsection}{IA.\arabic{subsection}}

\subsection{Uniform Convergence via Sequences}

\begin{lemma}[Uniform Convergence via Sequences]\label{lemma: uniform convergence}
Assume i) $\Theta\subset\mathbb{R}^{d}$ is compact; ii) let $g_{n}:\Theta\to\mathbb{R}^{l}$
be a sequence of functions and $g:\Theta\to\mathbb{R}^{l}$ be a continuous
function; iii) for every sequence $\{\bm{W}_{n}\}\subset\Theta$ such
that $\bm{W}_{n}\to\bm{W}\in\Theta$, we have $g_{n}(\bm{W}_{n})\to g(\bm{W}).$
Then $g_{n}\to g$ uniformly on $\Theta$, i.e., 
\[
\sup_{\bm{W}\in\Theta}\Vert g_{n}(\bm{W})-g(\bm{W})\Vert\to0.
\]
\end{lemma}
\begin{proof}[\textbf{Proof of Lemma \protect\ref{lemma: uniform convergence}}]
Fix $\varepsilon>0$. Suppose, by contradiction, that $g_{n}$ does
not converge uniformly to $g$ on $\Theta$. Then there exist $\varepsilon>0$,
a subsequence $\{n_{k}\}$, and points $\bm{W}_{n_{k}}\in\Theta$
such that 
\[
\Vert g_{n_{k}}(\bm{W}_{n_{k}})-g(\bm{W}_{n_{k}})\Vert>\varepsilon\quad\text{for all }k.
\]

Since $\Theta$ is compact, by the Bolzano-Weierstrass theorem, the
subsequence $\{\bm{W}_{n_{k}}\}$ admits a convergent sub-subsequence,
denoted $\{\bm{W}_{n_{k_{j}}}\}$, such that $\bm{W}_{n_{k_{j}}}\to\bm{W}_{0}\in\Theta.$

By the assumed sequential convergence property, we have $g_{n_{k_{j}}}(\bm{W}_{n_{k_{j}}})\to g(\bm{W}_{0}).$
By continuity of $g$, $g(\bm{W}_{n_{k_{j}}})\to g(\bm{W}_{0}).$
Therefore, 
\[
\Vert g_{n_{k_{j}}}(\bm{W}_{n_{k_{j}}})-g(\bm{W}_{n_{k_{j}}})\Vert\le\Vert g_{n_{k_{j}}}(\bm{W}_{n_{k_{j}}})-g(\bm{W}_{0})\Vert+\Vert g(\bm{W}_{0})-g(\bm{W}_{n_{k_{j}}})\Vert\to0,
\]
which contradicts the definition of the sequence $\{\bm{W}_{n_{k}}\}$.
Hence $g_{n}\to g$ uniformly on $\Theta$. 
\end{proof}

\subsection{Consistency of (Co)variance Estimator}

\begin{lemma}[Consistency of (Co)variance Estimator] \label{lemma: convergence in covariance}
Assume that $\bm{A}_{t}$ and $\bm{B}_{t}$ are $K\times1$ random
vectors. For some $\delta>0$ and $\zeta>1$, $E\left\Vert \bm{A}_{t}\right\Vert ^{4\left(\zeta+\delta\right)}<\infty$,
$E\left\Vert \bm{B}_{t}\right\Vert ^{4\left(\zeta+\delta\right)}<\infty$,
$\left\{ \left(\bm{A}_{t}^{\top},\bm{B}_{t}^{\top}\right)\right\} $
is strictly stationary and is a $\beta$-mixing sequence with a mixing
coefficient $\alpha(s)$ such that $\alpha(s)=O(s^{-\lambda})$ for
a $\lambda>2\zeta/(\zeta-1)$. Let $q=\exp\left(-cT^{-1/3}\right)$
for some $c>0$. Then, as $T\to\infty,$ 
\[
\frac{1}{T}\sum_{t=1}^{T}\sum_{\tau=1}^{T}q^{\left|t-\tau\right|}\bm{A}_{t}\bm{B}_{\tau}^{\top}\overset{P}{\to}\sum_{\iota=-\infty}^{\infty}E\left(\bm{A}_{t}\bm{B}_{t+\iota}^{\top}\right).
\]
\end{lemma} 
\begin{proof}[\textbf{Proof of Lemma \protect\ref{lemma: convergence in covariance}}]

We begin with showing that the target $\sum_{\iota=-\infty}^{\infty}E\left(\bm{A}_{t}\bm{B}_{t+\iota}^{\top}\right)$
is well-defined. Applying Theorem 14.13.2 in Hansen (\citeyear{hansen2022econometrics}),
we can bound the covariance term as follows: 
\begin{align*}
\left\Vert \sum_{\iota=-\infty}^{\infty}E\left(\bm{A}_{t}\bm{B}_{t+\iota}^{\top}\right)\right\Vert  & \leq\sum_{\iota=-\infty}^{\infty}\left\Vert E\left(\bm{A}_{t}\bm{B}_{t+\iota}^{\top}\right)\right\Vert \\
 & \leq8\left(E\left\Vert \bm{A}_{t}\right\Vert ^{4\left(\zeta+\delta\right)}\right)^{1/4\left(\zeta+\delta\right)}\left(E\left\Vert \bm{B}_{t}\right\Vert ^{4\left(\zeta+\delta\right)}\right)^{1/4\left(\zeta+\delta\right)}\sum_{\iota=-\infty}^{\infty}\alpha(\iota)^{1-1/2(\zeta+\delta)}.
\end{align*}
Given that $\sum_{\iota=1}^{\infty}\alpha(\iota)^{1-1/2(\zeta+\delta)}<\infty$
holds under the mixing assumption, one can deduce that $\left\Vert \sum_{\iota=-\infty}^{\infty}E\left(\bm{A}_{t}\bm{B}_{t+\iota}^{\top}\right)\right\Vert <\infty.$

We rewrite 
\begin{equation}
\frac{1}{T}\sum_{t=1}^{T}\sum_{\tau=1}^{T}q^{\left|t-\tau\right|}\bm{A}_{t}\bm{B}_{\tau}^{\top}=\frac{1}{T}\sum_{t=1}^{T}\bm{A}_{t}\bm{B}_{t}^{\top}+\frac{1}{T}\sum_{\iota=1}^{T-1}\sum_{t=1}^{T-\iota}q^{\iota}\left(\bm{A}_{t}\bm{B}_{t+\iota}^{\top}+\bm{A}_{t+\iota}\bm{B}_{t}^{\top}\right).\label{eq:chiang 2}
\end{equation}
and now demonstrate 
\begin{equation}
\left\Vert \frac{1}{T}\sum_{\iota=1}^{T-1}\sum_{t=1}^{T-\iota}q^{\iota}\bm{A}_{t}\bm{B}_{t+\iota}^{\top}-\sum_{\iota=1}^{\infty}E\left(\bm{A}_{t}\bm{B}_{t+\iota}^{\top}\right)\right\Vert =o_{P}\left(1\right).\label{eq:chiang 3}
\end{equation}
Let $\kappa=\sqrt{-\left(\ln q\right)^{-1}T^{1/2}}$, we have 
\begin{align*}
\left\Vert \frac{1}{T}\sum_{\iota=1}^{T-1}\sum_{t=1}^{T-\iota}q^{\iota}\bm{A}_{t}\bm{B}_{t+\iota}^{\top}-\sum_{\iota=1}^{\infty}E\left(\bm{A}_{t}\bm{B}_{t+\iota}^{\top}\right)\right\Vert \leq & \left\Vert \sum_{\iota=1}^{\kappa-1}\frac{1}{T}\sum_{t=1}^{T-\iota}q^{\iota}\left(\bm{A}_{t}\bm{B}_{t+\iota}^{\top}-E\left(\bm{A}_{t}\bm{B}_{t+\iota}^{\top}\right)\right)\right\Vert \\
 & +\left\Vert \sum_{\iota=\kappa}^{T-1}\frac{1}{T}\sum_{t=1}^{T-\iota}q^{\iota}\left(\bm{A}_{t}\bm{B}_{t+\iota}^{\top}-E\left(\bm{A}_{t}\bm{B}_{t+\iota}^{\top}\right)\right)\right\Vert \\
 & +\left\Vert \sum_{\iota=1}^{T-1}\left(1-q^{\iota}\right)\frac{1}{T}\sum_{t=1}^{T-\iota}E\left(\bm{A}_{t}\bm{B}_{t+\iota}^{\top}\right)\right\Vert \\
 & +\left\Vert \sum_{\iota=T}^{\infty}E\left(\bm{A}_{t}\bm{B}_{t+\iota}^{\top}\right)\right\Vert +o\left(1\right).
\end{align*}

Consider the first term, $\left\Vert \sum_{\iota=1}^{\kappa-1}\frac{1}{T}\sum_{t=1}^{T-\iota}q^{\iota}\left(\bm{A}_{t}\bm{B}_{t+\iota}^{\top}-E\left(\bm{A}_{t}\bm{B}_{t+\iota}^{\top}\right)\right)\right\Vert $.
Set 
\[
\bm{H}_{t,\iota}=q^{\iota}\left(\bm{A}_{t}\bm{B}_{t+\iota}^{\top}-E\left(\bm{A}_{t}\bm{B}_{t+\iota}^{\top}\right)\right)
\]
and information set $\mathcal{F}_{-\infty}^{t-\ell}=\sigma\left(\left(\bm{A}_{\tau}\right)_{\tau=-\infty}^{t-\ell},\left(\bm{B}_{\tau}\right)_{\tau=-\infty}^{t-\ell}\right).$
By Theorem 14.2 in Davidson (\citeyear{davidson1994stochastic}),
we have 
\[
E\left(\left\Vert E\left(\bm{H}_{t,\iota}\vert\mathcal{F}_{-\infty}^{t-\ell}\right)\right\Vert ^{2}\right)^{1/2}\leq6\alpha(\ell)^{1/2-1/2\left(\zeta+\delta\right)}\left(E\left\Vert \bm{H}_{t,\iota}\right\Vert ^{2\left(\zeta+\delta\right)}\right)^{1/2\left(\zeta+\delta\right)}.
\]
Given that $\left|q^{\iota}\right|\leq1$, by the application of Cauchy-Schwarz
inequality, we have 
\[
E\left\Vert \bm{H}_{t,\iota}\right\Vert ^{2\left(\zeta+\delta\right)}\leq E\left\Vert \bm{A}_{t}\bm{B}_{t+\iota}^{\top}\right\Vert ^{2\left(\zeta+\delta\right)}\leq\left(E\left\Vert \bm{A}_{t}\right\Vert ^{4\left(\zeta+\delta\right)}\right)^{1/2}\left(E\left\Vert \bm{B}_{t}\right\Vert ^{4\left(\zeta+\delta\right)}\right)^{1/2}<\infty.
\]
Moreover, by the mixing assumption, we have $\sum_{\ell=1}^{\infty}6\alpha(\ell)^{1/2-1/2\left(r+\delta\right)}<\infty$.
Then, applying Lemma A in Hansen (\citeyear{hansen1992consistent})
yields 
\begin{equation}
E\left(\left\Vert \frac{1}{T}\sum_{t=1}^{T-\iota}\bm{H}_{t,\iota}\right\Vert ^{2}\right)^{1/2}\leq36\cdot\sum_{\ell=1}^{\infty}6\alpha(\ell)T^{-1/2}\left(E\left\Vert \bm{H}_{t,\iota}\right\Vert ^{2\left(\zeta+\delta\right)}\right)^{1/2\left(\zeta+\delta\right)}=O\left(T^{-1/2}\right).\label{eq:chiang 1}
\end{equation}
Together, the above result, the application of Markov inequality and
Minkowski inequality imply that 
\begin{align*}
P\left(\left\Vert \sum_{\iota=1}^{\kappa-1}\frac{1}{T}\sum_{t=1}^{T-\iota}\bm{H}_{t,\iota}\right\Vert \text{>}\varepsilon\right)= & O\left(E\left\Vert \sum_{\iota=1}^{\kappa-1}\frac{1}{T}\sum_{t=1}^{T-\iota}\bm{H}_{t,\iota}\right\Vert ^{2}\right)\\
= & O\left(\left[\sum_{\iota=1}^{\kappa-1}\left(E\left\Vert \frac{1}{T}\sum_{t=1}^{T-\iota}\bm{H}_{t,\iota}\right\Vert ^{2}\right)^{1/2}\right]^{2}\right)\\
= & O\left(\kappa^{2}T^{-1}\right).
\end{align*}
Recall that $\kappa=\sqrt{-\left(\ln q\right)^{-1}T^{1/2}}$ and $q=\exp\left(-cT^{-1/3}\right)$,
and hence we have $\kappa^{2}T^{-1}=o\left(1\right).$

For the second term, $\left\Vert \sum_{\iota=\kappa}^{T-1}\frac{1}{T}\sum_{t=1}^{T-\iota}q^{\iota}\left(\bm{A}_{t}\bm{B}_{t+\iota}^{\top}-E\left(\bm{A}_{t}\bm{B}_{t+\iota}^{\top}\right)\right)\right\Vert $.
By Cauchy-Schwarz inequality, one can deduce that $E\left\Vert \bm{A}_{t}\bm{B}_{t+\iota}^{\top}\right\Vert <\infty$.
Moreover, under the mixing assumption, $\ensuremath{\sum_{\iota=\kappa}^{T-1}q^{\iota}=\frac{q^{\kappa}\left(1-q^{T-\kappa}\right)}{1-q}\rightarrow0}$
as $T\to\infty$. Thus, by applying the triangular inequality, we
obtain that 
\begin{align*}
E\left\Vert \sum_{\iota=\kappa}^{T-1}\frac{1}{T}\sum_{t=1}^{T-\iota}q^{\iota}\left(\bm{A}_{t}\bm{B}_{t+\iota}^{\top}-E\left(\bm{A}_{t}\bm{B}_{t+\iota}^{\top}\right)\right)\right\Vert \leq & \sum_{\iota=\kappa}^{T-1}q^{\iota}E\left\Vert \bm{A}_{t}\bm{B}_{t+\iota}^{\top}-E\left(\bm{A}_{t}\bm{B}_{t+\iota}^{\top}\right)\right\Vert \\
= & O\left(\sum_{\iota=\kappa}^{T-1}q^{\iota}\right)\\
= & o\left(1\right).
\end{align*}

Next, consider the third term, $\left\Vert \sum_{\iota=1}^{T-1}\left(1-q^{\iota}\right)\frac{1}{T}\sum_{t=1}^{T-\iota}E\left(\bm{A}_{t}\bm{B}_{t+\iota}^{\top}\right)\right\Vert $.
Recall that we can bound the covariance term such that: 
\[
\left\Vert \sum_{\iota=1}^{T-1}\left(1-q^{\iota}\right)\frac{1}{T}\sum_{t=1}^{T-\iota}E\left(\bm{A}_{t}\bm{B}_{t+\iota}^{\top}\right)\right\Vert =\sum_{\iota=1}^{T-1}\left(1-q^{\iota}\right)\frac{T-\iota}{T}\left\Vert E\left(\bm{A}_{t}\bm{B}_{t+\iota}^{\top}\right)\right\Vert \leq\sum_{\iota=1}^{T-1}\left(1-q^{\iota}\right)\alpha(\iota)^{(1+2\delta)/(4+4\delta)},
\]
with $\sum_{\iota=1}^{\infty}\alpha(\iota)^{(1+2\delta)/(4+4\delta)}<\infty$
under the mixing assumption. Therefore, given that $1-q^{\iota}\rightarrow0$
as $T\rightarrow\infty$ for a fixed $\iota$, the application of
the dominated convergence theorem yields that $\left\Vert \sum_{\iota=1}^{T-1}\left(1-q^{\iota}\right)\frac{1}{T}\sum_{t=1}^{T-\iota}E\left(\bm{A}_{t}\bm{B}_{t+\iota}^{\top}\right)\right\Vert =o\left(1\right)$.

Finally, for the last term, given that $\left\Vert \sum_{\iota=-\infty}^{\infty}E\left(\bm{A}_{t}\bm{B}_{t+\iota}^{\top}\right)\right\Vert <\infty$,
we have directly that $\left\Vert \sum_{\iota=T}^{\infty}\frac{1}{T}\sum_{t=1}^{T-\iota}E\left(\bm{A}_{t}\bm{B}_{t+\iota}^{\top}\right)\right\Vert =o\left(1\right)$,
as $T\to\infty$. Together, the above result establishes that \eqref{eq:chiang 3}
holds. Following similar argument, we can deduce the desirable result. 
\end{proof}

\subsection{CLT for the Original Sample}

\begin{lemma}[CLT for the Original Sample]\label{lemma: original clt}
Assume that Assumptions \ref{as: AHS representation}-\ref{as: same rate}
hold and $(\bm{c}_{l,0},\bm{\nu}_{0})=\lim_{N,T\to\infty}(\bm{c}_{l,f},{\bm{\nu}}_{f})$
exist, where with $\bm{\nu}_{f}$ captures the variance and covariance
terms, as defined in \eqref{eq:nu-f}. Then, it follows that 
\[
\left\Vert P_{NT}\left(r_{NT,f}\left(\frac{1}{NT}\sum_{i=1}^{N}\sum_{t=1}^{T}\bm{s}_{it}\right)\right)-\mathcal{\bm{L}}_{0}\left(\bm{\nu}_{0},\left\{ \bm{c}_{ll',0}\right\} _{l,l'=1}^{\infty}\right)\right\Vert _{\infty}\xrightarrow{P}0,
\]
where the limiting distribution is given by 
\begin{equation}
\mathcal{\bm{L}}_{0}\left(\bm{\nu}_{0},\left\{ \bm{c}_{ll',0}\right\} _{l,l'=1}^{\infty}\right)=\bm{\nu}_{a,0}\bm{Z}^{a}+\bm{\nu}_{d,0}\bm{Z}^{d}+\bm{\nu}_{e,0}\bm{Z}^{e}+\bm{\nu}_{v,0}\sum_{l,l'=1}^{\infty}\bm{c}_{ll',0}\odot\left(\bm{Z}_{l}^{\phi}\odot\bm{Z}_{l'}^{\psi}\right),\label{eq: limiting distribution}
\end{equation}
which is continuous in $\bm{c}_{l,0}$ and $\bm{\nu}_{0}$. \end{lemma} 
\begin{proof}[\textbf{Proof of Lemma \protect\ref{lemma: original clt}}]
We organize the argument in three steps.

\paragraph{Step 1: Marginal CLTs and pairwise covariance convergence.}

We first establish marginal (vector) CLTs for the building blocks
and the convergence of the relevant cross-covariances.

\emph{(i) Marginal CLTs for $\bm{d}$, $\bm{a}$, $\phi$, and $\psi$.}
By Assumption~\ref{as: moment and variance}, 
\[
Var\!\left(\frac{1}{\sqrt{T}}\sum_{t=1}^{T}\bm{\sigma}_{d,f}^{-1}\bm{d}_{t}\right)\xrightarrow{P}\mathbf{I}.
\]
Under the stated stationarity/mixing and moment conditions (Assumption~\ref{as: moment and variance}),
Theorem~14.15 of Hansen \citeyearpar{hansen2022econometrics} yields
\[
\frac{1}{\sqrt{T}}\sum_{t=1}^{T}\bm{\sigma}_{d,f}^{-1}\bm{d}_{t}\xrightarrow{d}\mathcal{N}(\bm{0},\mathbf{I}),\qquad r_{NT,f}\frac{\bm{\sigma}_{d,f}}{\sqrt{T}}\bm{Z}_{T}^{d}\xrightarrow{d}\bm{\nu}_{d,0}\bm{Z}^{d}.
\]

For $\frac{1}{\sqrt{N}}\sum_{i=1}^{N}\bm{\sigma}_{a,f}^{-1}\bm{a}_{i}$, we seek to apply Theorem~2 of Conley~\citeyearpar{conley1999gmm}.
Since the sample mean is a special case of GMM with moment $g_{i}(\theta)=\bm{a}_{i}$
(no unknown parameters), Conley's Assumptions A1--A3 and B4--B5
are satisfied under Assumption \ref{as:spatial} with standing regularity
conditions on the sampling region and the kernel/truncation sequence.
It remains to verify the mixing and moment conditions (B1--B3).

\noindent \smallskip{}
 \emph{Moment condition.} By Jensen and Cauchy--Schwarz (for a generic
constant $C$), 
\[
E\|\bm{a}_{i}\|^{4(\zeta+\delta)}\;\le\;C\,E\|\bm{X}_{it}^{\top}\bm{u}_{it}\|^{4(\zeta+\delta)}\;\le\;C\Big(E\|\bm{X}_{it}\|^{8(\zeta+\delta)}\Big)^{1/2}\Big(E\|\bm{u}_{it}\|^{8(\zeta+\delta)}\Big)^{1/2}\;<\;\infty.
\]
Let $\delta'\equiv4(\zeta+\delta-1)>0$. Then $E\|\bm{a}_{i}\|^{4+\delta'}<\infty$,
which is Conley's moment requirement.

\noindent \smallskip{}
 \emph{Mixing conditions.} By Assumption \ref{as:spatial} with $\delta'\equiv4(\zeta+\delta-1)$,
we have $\alpha_{\infty,\infty}(r)^{\delta'/(2+\delta')}=o(r^{-4}).$
Then $\sum_{r\ge1}r\,\alpha_{\infty,\infty}(r)^{\delta'/(2+\delta')}<\infty$
because $r\,\alpha_{\infty,\infty}(r)^{\delta'/(2+\delta')}=o(r^{-3})$
and $\sum_{r\ge1}r^{-3}<\infty$. Since $\alpha_{1,1}(r)\le\alpha_{\infty,\infty}(r)$,
this implies Conley's B3: $\sum_{r=1}^{\infty}r\,\alpha_{1,1}(r)^{\delta'/(2+\delta')}<\infty.$
Moreover, $\alpha_{\infty,\infty}(r)=o\!\left(r^{-4(2+\delta')/\delta'}\right)$,
hence for $d_{1}+d_{2}\leq4$, 
\[
\sum_{r=1}^{\infty}r\,\alpha_{d_{1},d_{2}}(r)\;\le\;\sum_{r=1}^{\infty}r\,\alpha_{\infty,\infty}(r)\;=\;o\!\left(\sum_{r=1}^{\infty}r^{-(8+3\delta')/\delta'}\right)\;<\;\infty,
\]
because $(8+3\delta')/\delta'>1$, which establishes Conley's B1.
Finally, 
\[
\alpha_{1,\infty}(r)\le\alpha_{\infty,\infty}(r)=o\!\left(r^{-4(2+\delta')/\delta'}\right)=o(r^{-2}),
\]
which is Conley's B2. Therefore, applying Theorem~2 of Conley~\citeyearpar{conley1999gmm}  with Slutsky's Lemma
yields that  
\[
\frac{1}{\sqrt{N}}\sum_{i=1}^{N}\bm{\sigma}_{a,f}^{-1}\bm{a}_{i}\;\overset{d}{\to}\;\mathcal{N}(\bm{0},\mathbf{I}).
\]
The same argument gives
the marginal CLTs for $\bm{Z}_{N,l}^{\phi}$ and
$\bm{Z}_{T,l'}^{\psi}$.

\emph{(ii) Marginal CLT for $\bm{e}$.} Recall 
\[
\bm{Z}_{NT}^{e}=\frac{1}{\sqrt{NT}}\sum_{i=1}^{N}\sum_{t=1}^{T}\bm{\sigma}_{e,f}^{-1}\bm{e}_{it},\qquad\bm{e}_{it}=\bm{X}_{it}^{\top}\bm{u}_{it}-E(\bm{X}_{it}^{\top}\bm{u}_{it}\mid\bm{\alpha}_{i},\bm{\xi}_{t}),
\]
and define $\mathcal{F}_{NT}=\sigma(\{\bm{\alpha}_{i}\}_{i=1}^{N},\{\bm{\xi}_{t}\}_{t=1}^{T})$,
so that $E(\bm{e}_{it}\mid\mathcal{F}_{NT})=\bm{0}$. Let 
\[
\bm{V}_{NT}=\frac{1}{NT}\sum_{i=1}^{N}\sum_{t=1}^{T}\sum_{t'=1}^{T}E\!\left(\bm{\sigma}_{e,f}^{-1}\bm{e}_{it}\bm{e}_{it'}^{\top}\bm{\sigma}_{e,f}^{-1}\mid\mathcal{F}_{NT}\right).
\]
By Assumption~\ref{as: AHS representation}, conditional on $\mathcal{F}_{NT}$
the array $\{\bm{\sigma}_{e,f}^{-1}\bm{e}_{it}\}_{i,t}$ is independent
with mean zero and finite $(2+\delta)$-moment for some $\delta>0$.
Hence the conditional Lyapunov condition holds: 
\[
(NT)^{-(1+\delta/2)}\sum_{i,t}E\!\left(\|\bm{\sigma}_{e,f}^{-1}\bm{e}_{it}\|^{2+\delta}\mid\mathcal{F}_{NT}\right)=(NT)^{-\delta/2}E\!\left(\|\bm{\sigma}_{e,f}^{-1}\bm{e}_{it}\|^{2+\delta}\mid\mathcal{F}_{NT}\right)=o_{P}(1),
\]
where $E(\|\bm{\sigma}_{e,f}^{-1}\bm{e}_{it}\|^{2+\delta}\mid\mathcal{F}_{NT})=O_{P}(1)$
follows from $E\|\bm{\sigma}_{e,f}^{-1}\bm{e}_{it}\|^{2+\delta}<\infty$
and iterated expectation. Therefore, if $\bm{V}_{NT}\xrightarrow{P}\mathbf{I}$,
we have 
\[
\bm{V}_{NT}^{-1/2}\frac{1}{\sqrt{NT}}\sum_{i,t}\bm{\sigma}_{e,f}^{-1}\bm{e}_{it}\,\Big|\,\mathcal{F}_{NT}\xrightarrow{d}\mathcal{N}(\bm{0},\mathbf{I}).
\]
It remains to show $\bm{V}_{NT}\xrightarrow{P}\mathbf{I}$. By a law
of large numbers over $i$ (conditional on $\{\bm{\xi}_{t}\}$) and
then rewriting the double sum over $(t,t')$ by lags, 
\begin{align*}
\bm{V}_{NT} & =\frac{1}{T}\sum_{t=1}^{T}\sum_{t'=1}^{T}\bm{\sigma}_{e,f}^{-1}E(\bm{e}_{it}\bm{e}_{it'}^{\top}\mid\bm{\xi}_{t},\bm{\xi}_{t'})\bm{\sigma}_{e,f}^{-1}+o_{P}(1)\\
 & =\frac{1}{T}\sum_{t=1}^{T}\sum_{\iota=-(T-1)}^{T-1}\Bigl(1-\frac{|\iota|}{T}\Bigr)\bm{\sigma}_{e,f}^{-1}E(\bm{e}_{it}\bm{e}_{i,t+\iota}^{\top}\mid\bm{\xi}_{t},\bm{\xi}_{t+\iota})\bm{\sigma}_{e,f}^{-1}+o_{P}(1).
\end{align*}
By the same truncation/dominated-convergence argument as in Lemma~\ref{lemma: convergence in covariance}
(using absolute summability of lag-covariances and stationarity of
$\{\bm{\xi}_{t}\}$), 
\[
\bm{V}_{NT}=\sum_{\iota=-\infty}^{\infty}E\!\left(\bm{\sigma}_{e,f}^{-1}E(\bm{e}_{it}\bm{e}_{i,t+\iota}^{\top}\mid\bm{\xi}_{t})\bm{\sigma}_{e,f}^{-1}\right)+o_{P}(1)=\mathbf{I}+o_{P}(1),
\]
where the last equality uses the definition/normalization of $\bm{\sigma}_{e,f}$
in Assumption~\ref{as: moment and variance}. Consequently, 
\[
\frac{1}{\sqrt{NT}}\sum_{i,t}\bm{\sigma}_{e,f}^{-1}\bm{e}_{it}\xrightarrow{d}\mathcal{N}(\bm{0},\mathbf{I}),\qquad r_{NT,f}\frac{\bm{\sigma}_{e,f}}{\sqrt{NT}}\bm{Z}_{NT}^{e}\xrightarrow{d}\bm{\nu}_{e,0}\bm{Z}^{e}.
\]

\emph{(iii) Pairwise covariance convergence.} For example, 
\[
\frac{1}{\sqrt{T}}\sum_{t=1}^{T}\bm{d}_{t}\left(\frac{1}{\sqrt{T}}\sum_{t=1}^{T}\psi_{l'}(\bm{\xi}_{t})\right)^{\!\top}=\frac{1}{T}\sum_{t=1}^{T}\sum_{\iota=-(T-1)}^{T-1}\Bigl(1-\frac{|\iota|}{T}\Bigr)\bm{d}_{t}\psi_{l'}(\bm{\xi}_{t+\iota})^{\top}.
\]
Hence, by Lemma~\ref{lemma: convergence in covariance}, 
\[
\frac{1}{\sqrt{T}}\sum_{t=1}^{T}\bm{d}_{t}\left(\frac{1}{\sqrt{T}}\sum_{t=1}^{T}\psi_{l'}(\bm{\xi}_{t})\right)^{\!\top}\overset{P}{\to}\bm{\sigma}_{dl',f}.
\]
Analogous covariance limits hold for the other required pairs, e.g.\ $(\bm{Z}^{a},\bm{Z}_{l}^{\phi})$
with $\bm{\sigma}_{al,f}$, etc.

\paragraph{Step 2: Joint CLT and truncation.}

Fix $L<\infty$ and define the block vectors $\bm{S}_{NT}^{e}=\bm{\sigma}_{e,f}\bm{Z}_{NT}^{e}/\sqrt{NT}$,
\[
\bm{S}_{N}^{a}=\Bigl((r_{NT,f}\bm{\sigma}_{a,f}\bm{Z}_{N}^{a}/\sqrt{N})^{\top},\{\bm{Z}_{N,l}^{\phi\top}\}_{l=1}^{L}\Bigr)^{\top},\quad\bm{S}_{T}^{d}=\Bigl((r_{NT,f}\bm{\sigma}_{d,f}\bm{Z}_{T}^{d}/\sqrt{T})^{\top},\{\bm{Z}_{T,l'}^{\psi\top}\}_{l'=1}^{L}\Bigr)^{\top}.
\]
By Step~1, each block is asymptotically Gaussian: 
\[
\bm{S}_{N}^{a}\overset{d}{\to}\mathcal{N}(\bm{0},\bm{\sigma}_{S_{a}}^{2}),\qquad\bm{S}_{T}^{d}\overset{d}{\to}\mathcal{N}(\bm{0},\bm{\sigma}_{S_{d}}^{2}),\qquad\bm{S}_{NT}^{e}\overset{d}{\to}\mathcal{N}(\bm{0},\bm{\sigma}_{S_{e}}^{2}).
\]
To obtain joint convergence and asymptotic independence, consider
the joint characteristic function 
\[
\phi_{NT}(\bm{u},\bm{v},\bm{w})=E\exp\{i\bm{u}^{\top}\bm{S}_{N}^{a}+i\bm{v}^{\top}\bm{S}_{T}^{d}+i\bm{w}^{\top}\bm{S}_{NT}^{e}\}.
\]
Let $\mathcal{F}_{NT}=\sigma(\{\bm{\alpha}_{i}\},\{\bm{\xi}_{t}\})$
as above. Since $\bm{S}_{N}^{a}$ and $\bm{S}_{T}^{d}$ are $\mathcal{F}_{NT}$-measurable,
\[
\phi_{NT}(\bm{u},\bm{v},\bm{w})=E\Bigl[\exp\{i\bm{u}^{\top}\bm{S}_{N}^{a}+i\bm{v}^{\top}\bm{S}_{T}^{d}\}\,E(\exp\{i\bm{w}^{\top}\bm{S}_{NT}^{e}\}\mid\mathcal{F}_{NT})\Bigr].
\]
By the conditional CLT in Step~1(ii), 
\[
E(\exp\{i\bm{w}^{\top}\bm{S}_{NT}^{e}\}\mid\mathcal{F}_{NT})\xrightarrow{P}\exp\!\left(-\frac{1}{2}\bm{w}^{\top}\bm{\sigma}_{S_{e}}^{2}\bm{w}\right),\qquad\left|E(\exp\{i\bm{w}^{\top}\bm{S}_{NT}^{e}\}\mid\mathcal{F}_{NT})\right|\le1,
\]
and dominated convergence implies 
\[
\phi_{NT}(\bm{u},\bm{v},\bm{w})\to\exp\!\left(-\frac{1}{2}\bm{w}^{\top}\bm{\sigma}_{S_{e}}^{2}\bm{w}\right)\cdot\lim_{N,T\to\infty}E\exp\{i\bm{u}^{\top}\bm{S}_{N}^{a}+i\bm{v}^{\top}\bm{S}_{T}^{d}\}.
\]
Under $\sigma(\{\bm{\alpha}_{i}\})\perp\sigma(\{\bm{\xi}_{t}\})$,
the vectors $\bm{S}_{N}^{a}$ and $\bm{S}_{T}^{d}$ are independent
for each $(N,T)$, hence 
\[
E\exp\{i\bm{u}^{\top}\bm{S}_{N}^{a}+i\bm{v}^{\top}\bm{S}_{T}^{d}\}=E\exp\{i\bm{u}^{\top}\bm{S}_{N}^{a}\}\cdot E\exp\{i\bm{v}^{\top}\bm{S}_{T}^{d}\}.
\]
Using the marginal Gaussian limits, 
\[
E\exp\{i\bm{u}^{\top}\bm{S}_{N}^{a}\}\to\exp\!\left(-\frac{1}{2}\bm{u}^{\top}\bm{\sigma}_{S_{a}}^{2}\bm{u}\right),\qquad E\exp\{i\bm{v}^{\top}\bm{S}_{T}^{d}\}\to\exp\!\left(-\frac{1}{2}\bm{v}^{\top}\bm{\sigma}_{S_{d}}^{2}\bm{v}\right),
\]
so one can deduce $\phi_{NT}(\bm{u},\bm{v},\bm{w})\to\exp\!\left(-\frac{1}{2}\bm{u}^{\top}\bm{\sigma}_{S_{a}}^{2}\bm{u}-\frac{1}{2}\bm{v}^{\top}\bm{\sigma}_{S_{d}}^{2}\bm{v}-\frac{1}{2}\bm{w}^{\top}\bm{\sigma}_{S_{e}}^{2}\bm{w}\right).$
It follows that 
\[
(\bm{S}_{N}^{a},\bm{S}_{T}^{d},\bm{S}_{NT}^{e})\overset{d}{\to}\mathcal{N}\!\left(\bm{0},Diag(\bm{\sigma}_{S_{a}}^{2},\bm{\sigma}_{S_{d}}^{2},\bm{\sigma}_{S_{e}}^{2})\right).
\]

This joint CLT implies, for fixed $L<\infty$, 
\[
\Bigl\| P_{NT}\!\Bigl(r_{NT,f}\Bigl(\frac{1}{NT}\sum_{i=1}^{N}\sum_{t=1}^{T}\bm{s}_{it}\Bigr)\Bigr)-\mathcal{L}_{L}(\bm{\nu}_{0},\{\bm{c}_{ll',f}\}_{l,l'=1}^{L})\Bigr\|_{\infty}\xrightarrow{P}0,
\]
where 
\[
\mathcal{L}_{L}(\bm{\nu}_{0},\{\bm{c}_{ll',f}\}_{l,l'=1}^{L})=\bm{\nu}_{a,0}\bm{Z}^{a}+\bm{\nu}_{d,0}\bm{Z}^{d}+\bm{\nu}_{e,0}\bm{Z}^{e}+\bm{\nu}_{v,0}\sum_{l,l'=1}^{L}\bm{c}_{ll',f}\odot(\bm{Z}_{l}^{\phi}\odot\bm{Z}_{l'}^{\psi}).
\]
By square summability of $\{\bm{c}_{ll',f}\}_{l,l'\ge1}$, for any
$\varepsilon>0$ there exists $L$ such that 
\[
\Bigl\|\mathcal{L}_{L}(\bm{\nu}_{0},\{\bm{c}_{ll',f}\}_{l,l'=1}^{L})-\mathcal{L}_{0}(\bm{\nu}_{0},\{\bm{c}_{ll',f}\}_{l,l'=1}^{\infty})\Bigr\|_{\infty}\xrightarrow{P}0,
\]
where $\mathcal{L}_{0}$ denotes the infinite-series limit. Combining
the last two displays yields the desired limit approximation.

\paragraph{Step 3: Continuity of the limiting law.}

It remains to verify that the limiting distribution is continuous.
The limit is a measurable function of a collection of jointly Gaussian
components. The only potentially nonzero dependence in the limit arises
between $\bm{\nu}_{a,0}\bm{Z}^{a}$ and $\{\bm{Z}_{l}^{\phi}\}_{l\ge1}$
(and similarly between $\bm{\nu}_{d,0}\bm{Z}^{d}$ and $\{\bm{Z}_{l'}^{\psi}\}_{l'\ge1}$)
in the regime where both terms are non-negligible. Since the limit
is jointly Gaussian, we can orthogonalize these pairs: there exist
Gaussian vectors $\bm{Z}_{\perp}^{\phi}$ and $\bm{Z}_{\perp}^{\psi}$
such that 
\[
\bm{\nu}_{a,0}\bm{Z}^{a}=\bm{Z}_{\perp}^{\phi}+\bm{\nu}_{v,0}\sum_{l=1}^{\infty}\bm{c}_{l,a}\odot\bm{Z}_{l}^{\phi},\qquad\bm{\nu}_{d,0}\bm{Z}^{d}=\bm{Z}_{\perp}^{\psi}+\bm{\nu}_{v,0}\sum_{l'=1}^{\infty}\bm{c}_{l',d}\odot\bm{Z}_{l'}^{\psi},
\]
where the collection 
\[
\Bigl(\bm{Z}_{\perp}^{\phi},\{\bm{Z}_{l}^{\phi}\}_{l\ge1},\bm{Z}_{\perp}^{\psi},\{\bm{Z}_{l'}^{\psi}\}_{l'\ge1},\bm{Z}^{e}\Bigr)
\]
are mutually independent. We can rewrite as 
\begin{align*}
&\mathcal{\bm{L}}_{0}\left(\bm{\nu}_{0},\left\{ \bm{c}_{ll',0}\right\} _{l,l'=1}^{\infty}\right)\\
= & \left(\bm{Z}_{\perp}^{\phi}+\bm{\nu}_{v,0}\sum_{l=1}^{\infty}\bm{c}_{l,a}\odot\bm{Z}_{l}^{\phi}\right)+\left(\bm{Z}_{\perp}^{\psi}+\bm{\nu}_{v,0}\sum_{l'=1}^{\infty}\bm{c}_{l',d}\odot\bm{Z}_{l'}^{\psi}\right)+\bm{\nu}_{e,0}\bm{Z}^{e}+\bm{\nu}_{v,0}\sum_{l,l'=1}^{\infty}\bm{c}_{ll',0}\odot\left(\bm{Z}_{l}^{\phi}\odot\bm{Z}_{l'}^{\psi}\right)\\
= & \bm{Z}_{\perp}^{\phi}+\bm{Z}_{\perp}^{\psi}+\bm{\nu}_{e,0}\bm{Z}^{e}+\bm{\nu}_{v,0}\sum_{l=1}^{\infty}\bm{c}_{l,a}\odot\bm{Z}_{l}^{\phi}+\bm{\nu}_{v,0}\sum_{l'=1}^{\infty}\left(\bm{c}_{l',d}+\sum_{l=1}^{\infty}\bm{c}_{ll',0}\odot\bm{Z}_{l}^{\phi}\right)\odot\bm{Z}_{l'}^{\psi}.
\end{align*}
Consequently, conditional on the $\left\{ \bm{Z}_{l}^{\phi}\right\} _{l}$,
$\mathcal{L}_{0}\vert\left\{ \bm{Z}_{l}^{\phi}\right\} _{l}$ is Gaussian
distributed with $Var\left(\mathcal{L}_{0}\vert\left\{ \bm{Z}_{l}^{\phi}\right\} _{l}\right)>0$.
Hence, by the law of total probability, the distribution of $\mathcal{L}_{0}$
has no point masses and is therefore continuous. This completes the
proof. 
\end{proof}

\subsection{Order of Variance Estimators}

\begin{lemma}[Order of Variance Estimators] \label{lemma:limit form-1}
Assume Assumptions \ref{as: AHS representation}-\ref{as: same rate}
hold. Then, it holds that 
\begin{align}
\widehat{\bm{\sigma}}_{a}^{2} & =\bm{\sigma}_{a,f}^{2}+O_{P}\left(T^{-1}\bm{\sigma}_{v,f}^{2}+T^{-1}\bm{\sigma}_{d,f}^{2}\right)+o_{P}\left(\bm{\sigma}_{a,f}^{2}+T^{-1}\bm{\sigma}_{e,f}^{2}\right),\label{eq: sigma a order}\\
\widehat{\bm{\sigma}}_{d}^{2} & =\bm{\sigma}_{d,f}^{2}+O_{P}\left(N^{-1}\bm{\sigma}_{v,f}^{2}+N^{-1}\bm{\sigma}_{a,f}^{2}\right)+o_{P}\left(\bm{\sigma}_{d,f}^{2}+N^{-1}\bm{\sigma}_{e,f}^{2}\right).\label{eq: sigma d order}
\end{align}

\end{lemma} 
\begin{proof}[\textbf{Proof of Lemma \protect\ref{lemma:limit form-1}}]
Define 
\[
\bm{\sigma}_{NT,f}^{2}\equiv Var\left(\frac{1}{NT}\sum_{i=1}^{N}\sum_{t=1}^{T}\bm{s}_{it}\right)=N^{-1}\bm{\sigma}_{a,f}^{2}+T^{-1}\bm{\sigma}_{d,f}^{2}+(NT)^{-1}(\bm{\sigma}_{v,f}^{2}+\bm{\sigma}_{e,f}^{2}),
\]
and hence by construction, we have 
\begin{equation}
N^{-1}\bm{\sigma}_{a,f}^{2}+T^{-1}\bm{\sigma}_{d,f}^{2}+N^{-1}T^{-1}\left(\bm{\sigma}_{e,f}^{2}+\bm{\sigma}_{v,f}^{2}\right)=O_{P}\left(\bm{\sigma}_{NT,f}^{2}\right).\label{eq: sigma NT}
\end{equation}

To proceed, consider the decomposition of $\ddot{\bm{d}}_{t}$. We
have: 
\begin{align}
\ddot{\bm{d}}_{t} & =\frac{1}{N}\sum_{i=1}^{N}\bm{s}_{it}-\frac{1}{NT}\sum_{i=1}^{N}\sum_{t=1}^{T}\bm{s}_{it}+\left(\frac{1}{N}\sum_{i=1}^{N}\bm{X}_{it}\bm{X}_{it}^{\top}-\frac{1}{NT}\sum_{i=1}^{N}\sum_{t=1}^{T}\bm{X}_{it}\bm{X}_{it}^{\top}\right)\left(\bm{\beta}-\widehat{\bm{\beta}}\right)\nonumber \\
 & =\bm{d}_{t}+\frac{1}{N}\sum_{i=1}^{N}\bm{w}_{it}-\frac{1}{T}\sum_{t=1}^{T}\left(\bm{d}_{t}+\frac{1}{N}\sum_{i=1}^{N}\bm{w}_{it}\right)+\left(\frac{1}{N}\sum_{i=1}^{N}\bm{X}_{it}\bm{X}_{it}^{\top}-\frac{1}{NT}\sum_{i=1}^{N}\sum_{t=1}^{T}\bm{X}_{it}\bm{X}_{it}^{\top}\right)\left(\bm{\beta}-\widehat{\bm{\beta}}\right),\label{eq: dit 3}
\end{align}

First, we focus on the main term of interest, showing that the autocovariance
component matches the true autovariance: 
\begin{align}
 & \frac{1}{T}\sum_{\iota=1}^{T-1}q^{\iota}\sum_{t=1}^{T-\iota}\left(\bm{d}_{t+\iota}+\frac{1}{N}\sum_{i=1}^{N}\bm{w}_{it+\iota}\right)\left(\bm{d}_{t}+\frac{1}{N}\sum_{i=1}^{N}\bm{w}_{it}\right)^{\top}\nonumber \\
=\, & \sum_{\iota=1}^{\infty}E(\bm{d}_{t+\iota}\bm{d}_{t}^{\top})+N^{-1}\sum_{\iota=1}^{\infty}E(\bm{w}_{it+\iota}\bm{w}_{it}^{\top})+O_{P}\left(N^{-1}\sum_{\iota=1}^{\infty}E(\bm{v}_{it+\iota}\bm{v}_{it}^{\top})+\bm{\sigma}_{NT,f}^{2}\right)+o_{P}\left(\bm{\sigma}_{d,f}^{2}+N^{-1}\bm{\sigma}_{e,f}^{2}\right).\label{eq: dw term}
\end{align}
To establish this, we decompose the left-hand side of \eqref{eq: dw term}
as: 
\begin{align}
 & \frac{1}{T}\sum_{\iota=1}^{T-1}q^{\iota}\sum_{t=1}^{T-\iota}\left(\bm{d}_{t+\iota}+\frac{1}{N}\sum_{i=1}^{N}\bm{w}_{it+\iota}\right)\left(\bm{d}_{t}+\frac{1}{N}\sum_{i=1}^{N}\bm{w}_{it}\right)^{\top}\nonumber \\
=\, & \sum_{\iota=1}^{T-1}q^{\iota}\frac{1}{T}\sum_{t=1}^{T-\iota}\left(\bm{d}_{t+\iota}+\frac{1}{N}\sum_{i=1}^{N}\bm{e}_{it+\iota}\right)\left(\bm{d}_{t}+\frac{1}{N}\sum_{i=1}^{N}\bm{e}_{it}\right)^{\top}\nonumber \\
 & +\sum_{\iota=1}^{T-1}q^{\iota}\frac{1}{T}\sum_{t=1}^{T-\iota}\left[\left(\bm{d}_{t+\iota}+\frac{1}{N}\sum_{i=1}^{N}\bm{e}_{it+\iota}\right)\left(\frac{1}{N}\sum_{i=1}^{N}\bm{v}_{it}\right)^{\top}+\left(\frac{1}{N}\sum_{i=1}^{N}\bm{v}_{it+\iota}\right)\left(\bm{d}_{t}+\frac{1}{N}\sum_{i=1}^{N}\bm{e}_{it}\right)^{\top}\right]\nonumber \\
 & +\sum_{\iota=1}^{T-1}q^{\iota}\frac{1}{T}\sum_{t=1}^{T-\iota}\left(\frac{1}{N}\sum_{i=1}^{N}\bm{v}_{it+\iota}\right)\left(\frac{1}{N}\sum_{i=1}^{N}\bm{v}_{it}\right)^{\top}\nonumber \\
\equiv\, & (1)+(2)+(3).\label{eq: term 123}
\end{align}

We begin with term (1), aiming to show: 
\begin{align}
 & \left\Vert \sum_{\iota=1}^{T-1}q^{\iota}\frac{1}{T}\sum_{t=1}^{T-\iota}\left(\bm{d}_{t+\iota}+\frac{1}{N}\sum_{i=1}^{N}\bm{e}_{it+\iota}\right)\left(\bm{d}_{t}+\frac{1}{N}\sum_{i=1}^{N}\bm{e}_{it}\right)^{\top}-\sum_{\iota=1}^{\infty}E\left(\bm{d}_{t+\iota}\bm{d}_{t}^{\top}+N^{-1}\bm{e}_{it+\iota}\bm{e}_{it}^{\top}\right)\right\Vert \nonumber \\
 & =o_{P}\left(\bm{\sigma}_{d,f}^{2}+N^{-1}\bm{\sigma}_{e,f}^{2}\right).\label{eq: term 1 oP}
\end{align}
Let $c=\sqrt{-\left(\ln q\right)^{-1}T^{1/2}}$. Under Assumption
\ref{as: moment and variance} (iv), we have $c\to\infty$ and $\sum_{\iota=c}^{T-1}q^{\iota}\to0$
as $T\to\infty$. Define $\check{\bm{d}}_{t}=\bm{\sigma}_{d,f}^{-1}\bm{d}_{t}$.
Then, by triangular inequality and quadratic inequality, we can bound:
\begin{align*}
\left\Vert \sum_{\iota=1}^{T-1}q^{\iota}\frac{1}{T}\sum_{t=1}^{T-\iota}\bm{d}_{t+\iota}\bm{d}_{t}^{\top}-\sum_{\iota=1}^{\infty}E\left(\bm{d}_{t+\iota}\bm{d}_{t}^{\top}\right)\right\Vert \leq & \left\Vert \bm{\sigma}_{d,f}^{2}\right\Vert \left\Vert \sum_{\iota=1}^{c}q^{\iota}\frac{1}{T}\sum_{t=1}^{T-\iota}\left(\check{\bm{d}}_{t+\iota}\check{\bm{d}}_{t}^{\top}-E\left(\check{\bm{d}}_{t+\iota}\check{\bm{d}}_{t}^{\top}\right)\right)\right\Vert \\
 & +\left\Vert \bm{\sigma}_{d,f}^{2}\right\Vert \left\Vert \sum_{\iota=c}^{T-1}q^{\iota}\frac{1}{T}\sum_{t=1}^{T-\iota}\left(\check{\bm{d}}_{t+\iota}\check{\bm{d}}_{t}^{\top}-E\left(\check{\bm{d}}_{t+\iota}\check{\bm{d}}_{t}^{\top}\right)\right)\right\Vert \\
 & +\left\Vert \bm{\sigma}_{d,f}^{2}\right\Vert \left\Vert \sum_{\iota=1}^{T-1}\left(1-q^{\iota}\right)\frac{1}{T}\sum_{t=1}^{T-\iota}E\left(\check{\bm{d}}_{t+\iota}\check{\bm{d}}_{t}^{\top}\right)\right\Vert \\
 & +\left\Vert \bm{\sigma}_{d,f}^{2}\right\Vert \left\Vert \sum_{\iota=T}^{\infty}\frac{1}{T}\sum_{t=1}^{T-\iota}E\left(\check{\bm{d}}_{t+\iota}\check{\bm{d}}_{t}^{\top}\right)\right\Vert \\
\equiv & (4)+(5)+(6)+(7).
\end{align*}
It follows from the standard arguments in Newey and West~(\citeyear{newey1986simple},
cf. equation (9)) that terms (4) and (7) are of orders $o_{P}(\bm{\sigma}_{d,f}^{2})$
and $o(\bm{\sigma}_{d,f}^{2})$, respectively. Term (5) is bounded
by: 
\[
(5)=O_{P}\left(\bm{\sigma}_{d,f}^{2}\sum_{\iota=c}^{T-1}q^{\iota}\right)=o_{P}\left(\bm{\sigma}_{d,f}^{2}\right).
\]

For term (6), applying Jensen's inequality and Corollary 6.17 of White
(\citeyear{white1984asymptotic}), the autocovariances are summable
with decay: 
\[
E\left(\check{\bm{d}}_{t+\iota}\check{\bm{d}}_{t}^{\top}\right)\leq C\alpha(\iota)^{(1+2\delta)/(4+4\delta)},
\]
where $\sum_{\iota=1}^{\infty}\alpha(\iota)^{(1+2\delta)/(4+4\delta)}<\infty$,
for some $C<\infty$. Since $q^{\iota}\to1$ as $T\to\infty$, dominated
convergence implies $(6)=o(\bm{\sigma}_{d,f}^{2})$. Therefore, we
conclude: 
\begin{equation}
\left\Vert \sum_{\iota=1}^{T-1}q^{\iota}\frac{1}{T}\sum_{t=1}^{T-\iota}\bm{d}_{t+\iota}\bm{d}_{t}^{\top}-\sum_{\iota=1}^{\infty}E\left(\bm{d}_{t+\iota}\bm{d}_{t}^{\top}\right)\right\Vert =o_{P}(\bm{\sigma}_{d,f}^{2}).\label{eq: dd converge mixing}
\end{equation}

By similar arguments and the mutual independence of $\bm{d}_{t}$
and $\bm{e}_{it}$, the full result in \eqref{eq: term 1 oP} follows.

We now analyze term (3) from \eqref{eq: term 123} and apply the property
of the Hadamard product: 
\begin{align}
 & \sum_{\iota=1}^{T-1}q^{\iota}\frac{1}{T}\sum_{t=1}^{T-\iota}\left(\frac{1}{N}\sum_{i=1}^{N}\bm{v}_{it+\iota}\right)\left(\frac{1}{N}\sum_{i=1}^{N}\bm{v}_{it}\right)^{\top}\nonumber \\
=\, & \sum_{\iota=1}^{T-1}q^{\iota}\frac{1}{T}\sum_{t=1}^{T-\iota}\frac{1}{N}\sum_{i=1}^{N}\sum_{l_1,l_1'=1}^{\infty}\left(\bm{c}_{l_1l_1',f}\odot\phi_{l_1}(\bm{\alpha}_{i})\odot\psi_{l_1'}(\bm{\xi}_{t+\iota})\right)\cdot\frac{1}{N}\sum_{i'=1}^{N}\sum_{l_2,l_2'=1}^{\infty}\left(\bm{c}_{l_2l_2',f}\odot\phi_{l_2}(\bm{\alpha}_{i'})\odot\psi_{l_2'}(\bm{\xi}_{t})\right)^{\top}\nonumber \\
=\, & \frac{1}{N}\sum_{l_1,l_1'=1}^{\infty}\bm{c}_{l_1l_1',f}\bm{c}_{l_1l_1',f}^{\top}\odot\left(\frac{1}{\sqrt{N}}\sum_{i=1}^{N}\phi_{l_1}(\bm{\alpha}_{i})\right)\left(\frac{1}{\sqrt{N}}\sum_{i=1}^{N}\phi_{l_1}(\bm{\alpha}_{i})\right)^{\top}\odot\sum_{\iota=1}^{T-1}q^{\iota}\frac{1}{T}\sum_{t=1}^{T-\iota}\psi_{l_1'}(\bm{\xi}_{t+\iota})\psi_{l_1'}(\bm{\xi}_{t})^{\top}\nonumber \\
 & +\frac{1}{N}\sum_{(l_2,l_2')\ne (l_1,l_1')}\bm{c}_{l_1l_1',f}\bm{c}_{l_2l_2',f}^{\top}\odot\left(\frac{1}{\sqrt{N}}\sum_{i=1}^{N}\phi_{l_1}(\bm{\alpha}_{i})\right)\left(\frac{1}{\sqrt{N}}\sum_{i=1}^{N}\phi_{l_2}(\bm{\alpha}_{i})\right)^{\top}\odot\sum_{\iota=1}^{T-1}q^{\iota}\frac{1}{T}\sum_{t=1}^{T-\iota}\psi_{l_1'}(\bm{\xi}_{t+\iota})\psi_{l_2'}(\bm{\xi}_{t})^{\top}.\label{eq: term 3}
\end{align}

Recall that 
\begin{align*}
\sum_{\iota=1}^{\infty}E\left(\bm{v}_{it}\bm{v}_{i,t+\iota}^{\top}\right)&=\sum_{l,l'=1}^{\infty}\bm{c}_{ll',f}\bm{c}_{ll',f}^{\top}\odot E\left(\phi_{l}\left(\bm{\alpha}_{i}\right)\phi_{l}\left(\bm{\alpha}_{i}\right)^{\top}\right)\odot\sum_{\iota=1}^{\infty}E\left(\psi_{l'}\left(\bm{\xi}_{t}\right)\psi_{l'}\left(\bm{\xi}_{t+\iota}\right)^{\top}\right)\\
&\equiv\sum_{l,l'=1}^{\infty}\bm{c}_{ll',f}\bm{c}_{ll',f}^{\top}\odot\bm{\sigma}_{\phi l}^{2}\odot\bm{\sigma}_{\psi l'+}^{2}.
\end{align*}
In the first term of \eqref{eq: term 3}, we have 
\[
\sum_{\iota=1}^{T-1}q^{\iota}\frac{1}{T}\sum_{t=1}^{T-\iota}\psi_{l'}(\bm{\xi}_{t+\iota})\psi_{l'}(\bm{\xi}_{t})^{\top}\xrightarrow{P}\bm{\sigma}_{\psi l'+}^{2},
\]
However, the random quadratic form 
\[
\left(\frac{1}{\sqrt{N}}\sum_{i=1}^{N}\phi_{l}(\bm{\alpha}_{i})\right)\left(\frac{1}{\sqrt{N}}\sum_{i=1}^{N}\phi_{l}(\bm{\alpha}_{i})\right)^{\top}\overset{P}{\to}\bm{Z}_{l}^{\phi}\bm{Z}_{l}^{\phi\top}\sim\mathcal{W}_{K}(\bm{\sigma}_{\phi l}^{2},1),
\]
a rank-one Wishart distribution. This persistent randomness is the
primary source of estimation noise, which prevents us from a more
precise estimation.

The second term in \eqref{eq: term 3} is negligible, because the
functions $\{\phi_{l}\}$ and $\{\psi_{l'}\}$ are orthogonal in $l$ and $l'$.
As such, term (3) has a distribution close to 
\begin{equation}
\frac{1}{N}\sum_{l,l'=1}^{\infty}\bm{c}_{ll',f}\bm{c}_{ll',f}^{\top}\odot\bm{Z}_{l}^{\phi}\bm{Z}_{l}^{\phi\top}\odot\bm{\sigma}_{\psi l'+}^{2},\label{eq:term 3-2}
\end{equation}
and hence we conclude that it satisfies: 
\begin{equation}
(3)=O_{P}(N^{-1}\bm{\sigma}_{v,f}^{2}).\label{eq: term 3-1}
\end{equation}

Next, consider term (2) from \eqref{eq: term 123}, which is a cross
term involving $\bm{d}_{t}$ and $\bm{v}_{it}$. By an argument similar
to that for term (3), we obtain: 
\[
(2)=o_{P}(N^{-1/2}\bm{\sigma}_{v,f}(\bm{\sigma}_{d,f}^{2}+N^{-1}\bm{\sigma}_{e,f}^{2})^{1/2}).
\]
This term is of smaller order than the leading terms in \eqref{eq: dw term},
and thus negligible. Putting together results for terms (1), (2),
and (3) in \eqref{eq: term 123}, we conclude that \eqref{eq: dw term}
holds.

Having established \eqref{eq: dw term}, we now apply similar reasoning
to the sample variance and autocovariance terms of $\ddot{\bm{d}}_{t}$.
Specifically, we consider the HAC-type estimator: 
\begin{align}
 & \frac{1}{T}\sum_{t=1}^{T}\ddot{\bm{d}}_{t}\ddot{\bm{d}}_{t}^{\top}+\sum_{\iota=1}^{T-1}q^{\iota}\frac{1}{T}\sum_{t=1}^{T-\iota}\left(\ddot{\bm{d}}_{t}\ddot{\bm{d}}_{t+\iota}^{\top}+\ddot{\bm{d}}_{t+\iota}\ddot{\bm{d}}_{t}^{\top}\right)\nonumber \\
=\, & \bm{\sigma}_{d,f}^{2}+N^{-1}\bm{\sigma}_{w,f}^{2}+O_{P}\left(N^{-1}\bm{\sigma}_{v,f}^{2}+\bm{\sigma}_{NT,f}^{2}\right)+o_{P}\left(\bm{\sigma}_{d,f}^{2}+N^{-1}\bm{\sigma}_{e,f}^{2}\right).\label{eq: term dd}
\end{align}

It remains to establish the analogous result for $\widehat{\bm{\sigma}}_{w}^{2}$.
We aim to show: 
\begin{align}
 & \frac{1}{N^{2}T}\sum_{i=1}^{N}\sum_{t=1}^{T}\ddot{\bm{w}}_{it}\ddot{\bm{w}}_{it}^{\top}+\sum_{\iota=1}^{T-1}q^{\iota}\frac{1}{N^{2}T}\sum_{i=1}^{N}\sum_{t=1}^{T-\iota}\left(\ddot{\bm{w}}_{it}\ddot{\bm{w}}_{it+\iota}^{\top}+\ddot{\bm{w}}_{it+\iota}\ddot{\bm{w}}_{it}^{\top}\right)\nonumber \\
= & N^{-1}\bm{\sigma}_{w,f}^{2}+O_{P}\left(N^{-1}\bm{\sigma}_{NT,f}^{2}\right)+o_{P}\left(N^{-1}\bm{\sigma}_{v,f}^{2}+N^{-1}\bm{\sigma}_{e,f}^{2}\right).\label{eq: term ww}
\end{align}
Note that $\ddot{\bm{w}}_{it}=\bm{v}_{it}+\bm{e}_{it}+O_{P}(\bm{\sigma}_{NT,f})$.
Applying CLT, we derive: 
\begin{align*}
\frac{1}{N^{2}T}\sum_{i=1}^{N}\sum_{t=1}^{T}\ddot{\bm{w}}_{it}\ddot{\bm{w}}_{it}^{\top} & =\frac{1}{N^{2}T}\sum_{i=1}^{N}\sum_{t=1}^{T}\left(\bm{e}_{it}\bm{e}_{it}^{\top}+\bm{v}_{it}\bm{v}_{it}^{\top}\right)+O_{P}\left(N^{-1}\bm{\sigma}_{NT,f}^{2}\right)\\
 & =N^{-1}E(\bm{w}_{it}^{\top}\bm{w}_{it})+O_{P}\left(N^{-1}\left(\bm{\sigma}_{NT,f}^{2}+N^{-1/2}T^{-1/2}\bm{\sigma}_{e,f}^{2}+(N^{-1/2}+T^{-1/2})\bm{\sigma}_{v,f}^{2}\right)\right).
\end{align*}
Similarly, the autocovariance component satisfies: 
\begin{align*}
\sum_{\iota=1}^{T-1}q^{\iota}\frac{1}{N^{2}T}\sum_{i=1}^{N}\sum_{t=1}^{T-\iota}\ddot{\bm{w}}_{it+\iota}\ddot{\bm{w}}_{it}^{\top}=\, & N^{-1}\sum_{\iota=1}^{\infty}E(\bm{w}_{it+\iota}\bm{w}_{it}^{\top})\\
 & +O_{P}\left(N^{-1}\left(\bm{\sigma}_{NT,f}^{2}+N^{-1/2}T^{-1/2}\bm{\sigma}_{e,f}^{2}+(N^{-1/2}+T^{-1/2})\bm{\sigma}_{v,f}^{2}\right)\right).
\end{align*}
Combining these gives \eqref{eq: term ww}.

For the rest terms in \eqref{eq: dit 3}. Notice that we have 
\[
\mathrm{Var}\left(\frac{1}{T}\sum_{t=1}^{T}\bm{d}_{t}\right)=O\left(T^{-1}\bm{\sigma}_{d,f}^{2}\right)\quad\text{and}\quad\mathrm{Var}\left(\frac{1}{NT}\sum_{i=1}^{N}\sum_{t=1}^{T}\bm{w}_{it}\right)=O\left(\frac{1}{NT}(\bm{\sigma}_{e,f}^{2}+\bm{\sigma}_{v,f}^{2})\right).
\]
following the proof of Chiang et al. (\citeyear{chiang2023standard},
cf. (A.3) and (A.4) in the proof of Theorem 1). Together, the above
results, and $\bm{\beta}-\widehat{\bm{\beta}}=O_{P}(\bm{\sigma}_{NT,f})$,
and the application of Chebyshev's inequality and Cauchy Schwarz inequality
yield that 
\begin{equation}
\frac{1}{T}\sum_{\iota=1}^{T-1}q^{\iota}\sum_{t=1}^{T-\iota}\ddot{\bm{d}_{t}}\ddot{\bm{d}_{t}}^{\top}=\frac{1}{T}\sum_{\iota=1}^{T-1}q^{\iota}\sum_{t=1}^{T-\iota}\left(\bm{d}_{t+\iota}+\frac{1}{N}\sum_{i=1}^{N}\bm{w}_{it+\iota}\right)\left(\bm{d}_{t}+\frac{1}{N}\sum_{i=1}^{N}\bm{w}_{it}\right)^{\top}+O_{P}(\bm{\sigma}_{NT,f}^{2}).\label{eq: dit 2}
\end{equation}
Hence, by \eqref{eq: term dd}-\eqref{eq: dit 2}, we can conclude
that the estimator satisfies: 
\[
\widehat{\bm{\sigma}}_{d}^{2}=\bm{\sigma}_{d,f}^{2}+O_{P}\left(N^{-1}\bm{\sigma}_{v,f}^{2}+\bm{\sigma}_{NT,f}^{2}\right)+o_{P}\left(\bm{\sigma}_{d,f}^{2}+N^{-1}\bm{\sigma}_{e,f}^{2}\right),
\]
as desired in \eqref{eq: sigma d order}. 

We now that shows
that \eqref{eq: sigma a order} also holds. By an analogous argument as above, the key part is to
show 
\begin{equation}
\left\Vert \frac{1}{N}\sum_{i=1}^{N}\sum_{j=1}^{N}K\!\left(\frac{\mathfrak{d}_{ij}}{\mathfrak{d}_{N}}\right)\bm{a}_{i}\bm{a}_{j}^{\top}-\frac{1}{N}\sum_{i=1}^{N}\sum_{j=1}^{N}E\!\left(\bm{a}_{i}\bm{a}_{j}^{\top}\right)\right\Vert =o_{P}(1).\label{eq:key_kern_conv}
\end{equation}
Rewrite as 
\[
\frac{1}{N}\sum_{i=1}^{N}\sum_{j=1}^{N}K\!\left(\frac{\mathfrak{d}_{ij}}{\mathfrak{d}_{N}}\right)\Big(\bm{a}_{i}\bm{a}_{j}^{\top}-E(\bm{a}_{i}\bm{a}_{j}^{\top})\Big)\;+\;\frac{1}{N}\sum_{i=1}^{N}\sum_{j=1}^{N}\Bigg(K\!\left(\frac{\mathfrak{d}_{ij}}{\mathfrak{d}_{N}}\right)-1\Bigg)E(\bm{a}_{i}\bm{a}_{j}^{\top})\equiv\mathcal{T}_{1N}+\mathcal{T}_{2N}.
\]

\noindent \smallskip{}
 \emph{For Term $\mathcal{T}_{1N}$.} Let 
\[
A_{i}\equiv\sum_{j=1}^{N}K\!\left(\frac{\mathfrak{d}_{ij}}{\mathfrak{d}_{N}}\right)\Big(\bm{a}_{i}\bm{a}_{j}^{\top}-E(\bm{a}_{i}\bm{a}_{j}^{\top})\Big),\qquad\mathcal{T}_{1N}=\frac{1}{N}\sum_{i=1}^{N}A_{i}.
\]
Since $K(\mathfrak{d}_{ij}/\mathfrak{d}_{N})\neq0$ implies $\mathfrak{d}_{ij}\le\mathfrak{d}_{N}$,
given the lattice $\mathcal{H}\subset\mathbb{R}^{2}$, we define the
maximum number of neighbors in distance $\mathfrak{d}_{N}$: 
\[
m_{N}\equiv\sup_{i\le N}\#\big\{ j\le N:\mathfrak{d}_{ij}\le\mathfrak{d}_{N}\big\}=O(\mathfrak{d}_{N}^{2}),
\]
so each $A_{i}$ contains at most $O(m_{N})$ nonzero elements. Then
\[
Var(\mathcal{T}_{1N})=\frac{1}{N^{2}}\sum_{i_{1}=1}^{N}\sum_{i_{2}=1}^{N}Cov(A_{i_{1}},A_{i_{2}})\le\frac{1}{N}\sup_{i_{1}}\sum_{i_{2}=1}^{N}|Cov(A_{i_{1}},A_{i_{2}})|.
\]
Using Davydov's inequality for strong mixing random fields, 
\[
|Cov(A_{i_{1}},A_{i_{2}})|\le C\|A_{i_{1}}\|_{2(\zeta+\delta)}\|A_{i_{2}}\|_{2(\zeta+\delta)}\alpha_{\infty,\infty}\!\big((\mathfrak{d}_{i_{1}i_{2}}-2\mathfrak{d}_{N})_{+}\big)^{1-1/(\zeta+\delta)}.
\]
Moreover, by Minkowski and the moment bound $E\|\bm{a}_{i}\|^{4(\zeta+\delta)}<\infty$,
\[
\|A_{i}\|_{2(\zeta+\delta)}\le C\sum_{j:\,\mathfrak{d}_{ij}\le\mathfrak{d}_{N}}\|\bm{a}_{i}\bm{a}_{j}^{\top}-E(\bm{a}_{i}\bm{a}_{j}^{\top})\|_{2(\zeta+\delta)}\le C\,m_{N}.
\]
Hence, we have 
\begin{align*}
\sum_{i_{2}=1}^{N}|Cov(A_{i_{1}},A_{i_{2}})| & \le Cm_{N}^{2}\sup_{i_{1}}\left[\#\{i_{2}:\mathfrak{d}_{i_{1}i_{2}}\le2\mathfrak{d}_{N}\}+\sum_{i_{2}:\,\mathfrak{d}_{i_{1}i_{2}}>2\mathfrak{d}_{N}}\alpha_{\infty,\infty}\!\big((\mathfrak{d}_{i_{1}i_{2}}-2\mathfrak{d}_{N})_{+}\big)^{1-1/(\zeta+\delta)}\right]
\end{align*}
By the same counting argument $\sup_{i_{1}\le N}\#\{i_{2}:\mathfrak{d}_{i_{1}i_{2}}\le2\mathfrak{d}_{N}\}=O(m_{N})$
and the summability of $\sum_{r\geq1}\alpha_{\infty,\infty}(r)^{1-1/(\zeta+\delta)}<\infty$,
we have 
\[
Var(\mathcal{T}_{1N})=O\!\left(\frac{m_{N}^{3}}{N}\right)=O\!\left(\frac{\mathfrak{d}_{N}^{6}}{N}\right)=o(1),\qquad\text{so}\qquad\mathcal{T}_{1N}=o_{P}(1).
\]

\noindent \smallskip{}
 \emph{Term $\mathcal{T}_{2N}$.} Decompose 
\[
\|\mathcal{T}_{2N}\|\le\frac{1}{N}\sum_{i=1}^{N}\left\Vert \sum_{j:\,\mathfrak{d}_{ij}\le\mathfrak{d}_{N}}\Big(K(\mathfrak{d}_{ij}/\mathfrak{d}_{N})-1\Big)E(\bm{a}_{i}\bm{a}_{j}^{\top})\right\Vert +\frac{1}{N}\sum_{i=1}^{N}\left\Vert \sum_{j:\,\mathfrak{d}_{ij}>\mathfrak{d}_{N}}E(\bm{a}_{i}\bm{a}_{j}^{\top})\right\Vert .
\]
For the first piece, $K(\mathfrak{d}_{ij}/\mathfrak{d}_{N})-1\to0$
pointwise for each fixed $\mathfrak{d}_{ij}$, and Davydov's inequality
gives 
\[
\left\Vert \sum_{j:\,\mathfrak{d}_{ij}\le\mathfrak{d}_{N}}E(\bm{a}_{i}\bm{a}_{j}^{\top})\right\Vert \le C\,\sum_{j:\,\mathfrak{d}_{ij}\le\mathfrak{d}_{N}}\alpha_{\infty,\infty}(\mathfrak{d}_{ij})^{1-1/(\zeta+\delta)}<\infty.
\]
Thus, by dominated convergence, 
\[
\frac{1}{N}\sum_{i=1}^{N}\left\Vert \sum_{j:\,\mathfrak{d}_{ij}\le\mathfrak{d}_{N}}\Big(K(\mathfrak{d}_{ij}/\mathfrak{d}_{N})-1\Big)E(\bm{a}_{i}\bm{a}_{j}^{\top})\right\Vert \to0.
\]
For the second piece, 
\[
\frac{1}{N}\sum_{i=1}^{N}\left\Vert \sum_{j:\,\mathfrak{d}_{ij}>\mathfrak{d}_{N}}E(\bm{a}_{i}\bm{a}_{j}^{\top})\right\Vert \le C\sum_{j:\,\mathfrak{d}_{ij}>\mathfrak{d}_{N}}\alpha_{\infty,\infty}(\mathfrak{d}_{ij})^{1-1/(\zeta+\delta)}\to0
\]
as $\mathfrak{d}_{N}\to\infty$ and the tail of a summable series
vanishes. Hence $\mathcal{T}_{2N}=o(1)$, and \eqref{eq:key_kern_conv}
follows. Finally, combining the above consistency with the
same algebraic decomposition used for $\widehat{\bm \sigma}_{d}^2$
yields $\eqref{eq: sigma a order}$.

\end{proof}

\subsection{Bootstrap CLT for $\frac{1}{\sqrt{N}}\sum_{i=1}^{N}\ddot{\bm{a}}_{i}\eta_{i}^{*b}$.}

\begin{lemma} \label{lemma: clt for bootstrap Za} If Assumptions
\ref{as: AHS representation}-\ref{as: same rate} hold
and $\bm{\sigma}_{a}^{2}=\lim_{N,T\to\infty}\bm{\sigma}_{a,f}^{2}>0$,
then 
\[
\frac{1}{\sqrt{N}}\sum_{i=1}^{N}\ddot{\bm{a}}_{i}\eta_{i}^{*b}\overset{d^{*}}{\to}\mathcal{N}(0,\bm{\sigma}_{a}^{2}).
\]
\end{lemma} 
\begin{proof}[\textbf{Proof of Lemma \protect\ref{lemma: clt for bootstrap Za}}]

Define the $K\times N$ matrix of
(centered) score vectors $\widehat{\bm{A}}_{N}\;\equiv\;(\ddot{\bm{a}}_{1},\ldots,\ddot{\bm{a}}_{N}).$
Recall that the kernel weight matrix $\mathbb{K}_{N}\;\equiv\;\Big(\mathcal{K}(\mathfrak{d}_{ij}/\mathfrak{d}_{N})\Big)_{1\le i,j\le N}$
is symmetric and positive semidefinite. Hence, it admits the spectral
decomposition 
\[
\mathbb{K}_{N}=\bm{\Phi}_{N}\bm{\Lambda}_{N}\bm{\Phi}_{N}^{\top},\qquad\bm{\Lambda}_{N}=diag(\lambda_{1},\ldots,\lambda_{N}),\ \lambda_{i}\ge0,\qquad\bm{\Phi}_{N}=(\bm{\phi}_{1},\ldots,\bm{\phi}_{N})\ \text{orthonormal}.
\]
Recall that ${\bm{\eta}}_{N}^{*b}\;\equiv\;\mathbb{K}_{N}^{1/2}\widetilde{\bm{\eta}}_{N}^{*b}\;=\;\bm{\Phi}_{N}\bm{\Lambda}_{N}^{1/2}\widetilde{\bm{\eta}}_{N}^{*b}.$
Then the bootstrap statistic can be written as the triangular array
\begin{align}
\frac{1}{\sqrt{N}}\sum_{i=1}^{N}\ddot{\bm{a}}_{i}\,{\eta}_{i}^{*b} & =\frac{1}{\sqrt{N}}\widehat{\bm{A}}_{N}\,{\bm{\eta}}_{N}^{*b}=\frac{1}{\sqrt{N}}\widehat{\bm{A}}_{N}\,\bm{\Phi}_{N}\bm{\Lambda}_{N}^{1/2}\widetilde{\bm{\eta}}_{N}^{*b}\equiv\sum_{i=1}^{N}\widehat{\bm{b}}_{iN}\,\widetilde{\eta}_{i}^{*b},\label{eq: bootstrap sum}
\end{align}
where $\widehat{\bm{b}}_{iN}\;\equiv\;\frac{1}{\sqrt{N}}\widehat{\bm{A}}_{N}\bm{\phi}_{i}\,\lambda_{i}^{1/2}\in\mathbb{R}^{K}.$

Fix any unit vector $\bm{\gamma}\in\mathbb{R}^{K}$ and consider $S_{N}^{*b}\;\equiv\;\bm{\gamma}^{\top}\frac{1}{\sqrt{N}}\sum_{i=1}^{N}\ddot{\bm{a}}_{i}\,{\eta}_{i}^{*b}=\sum_{i=1}^{N}\widehat{c}_{iN}\,\widetilde{\eta}_{i}^{*b}$,
with $\widehat{c}_{iN}\equiv\bm{\gamma}^{\top}\widehat{\bm{b}}_{iN}.$
Expanding the fourth moment yields 
\[
E^{*}\!\big[(S_{N}^{*b})^{4}\big]=\sum_{i_{1},i_{2},i_{3},i_{4}=1}^{N}\widehat{c}_{i_{1}N}\widehat{c}_{i_{2}N}\widehat{c}_{i_{3}N}\widehat{c}_{i_{4}N}\,E^{*}\!\big(\widetilde\eta_{i_{1}}^{*b}\widetilde\eta_{i_{2}}^{*b}\widetilde\eta_{i_{3}}^{*b}\widetilde\eta_{i_{4}}^{*b}\big).
\]
Because $\{\widetilde\eta_{i}^{*b}\}_{i\le N}$ are i.i.d.\ and centered, the
expectation $E^{*}(\widetilde\eta_{i_{1}}^{*b}\widetilde\eta_{i_{2}}^{*b}\widetilde\eta_{i_{3}}^{*b}\widetilde\eta_{i_{4}}^{*b})=0$
unless each index appears an even number of times. Thus the only nonzero
contributions come from: (i) one index appearing four times, and (ii)
two distinct indices each appearing twice. Hence, 
\begin{align*}
E^{*}\!\big[(S_{N}^{*b})^{4}\big] & =E^{*}\!\big[(\widetilde\eta_{1}^{*b})^{4}\big]\sum_{i=1}^{N}\widehat{c}_{iN}^{4}\;+\;3\sum_{1\le i\neq j\le N}\widehat{c}_{iN}^{2}\widehat{c}_{jN}^{2},
\end{align*}
where the factor $3$ corresponds to the three distinct pairings of
$(i,j)$ across four positions. Moreover, 
\[
\big(E^{*}\!\big[(S_{N}^{*b})^{2}\big]\big)^{2}=\left(\sum_{i=1}^{N}\widehat{c}_{iN}^{2}Var^{*}(\widetilde\eta_{i}^{*b})\right)^{2}=\left(\sum_{i=1}^{N}\widehat{c}_{iN}^{2}\right)^{2}=\sum_{i=1}^{N}\widehat{c}_{iN}^{4}+\sum_{1\le i\neq j\le N}\widehat{c}_{iN}^{2}\widehat{c}_{jN}^{2}.
\]
Therefore, 
\[
\frac{E^{*}\!\big[(S_{N}^{*b})^{4}\big]}{\big(E^{*}\!\big[(S_{N}^{*b})^{2}\big]\big)^{2}}=3+\Big(E^{*}\!\big[(\widetilde\eta_{1}^{*b})^{4}\big]-3\Big)\,\frac{\sum_{i=1}^{N}\widehat{c}_{iN}^{4}}{\left(\sum_{i=1}^{N}\widehat{c}_{iN}^{2}\right)^{2}}.
\]
Since $\widehat{c}_{iN}=(1/\sqrt{N})\widehat{\bm{A}}_{N}\bm{\phi}_{i}\lambda_{i}^{1/2}$
share the same $\widehat{\bm{A}}_{N}$ and $\{||\bm{\phi}_{i}||,\lambda_{i}\}$
are uniformly bounded, the weights are of comparable magnitude, and
$\frac{\sum_{i=1}^{N}\widehat{c}_{iN}^{4}}{\left(\sum_{i=1}^{N}\widehat{c}_{iN}^{2}\right)^{2}}=o_{P}(1).$
Consequently, 
\[
\frac{E^{*}\!\big[(S_{N}^{*b})^{4}\big]}{\big(E^{*}\!\big[(S_{N}^{*b})^{2}\big]\big)^{2}}=3+o_{P}(1).
\]
By Theorem~1 of De Jong~\citeyearpar{de1990central} and the Cramér-Wold
device, the conditional CLT for triangular arrays implies 
\[
\left[Var^{*}\!\left(\frac{1}{\sqrt{N}}\sum_{i=1}^{N}\ddot{\bm{a}}_{i}\,{\eta}_{i}^{*b}\right)\right]^{-1/2}\frac{1}{\sqrt{N}}\sum_{i=1}^{N}\ddot{\bm{a}}_{i}\,{\eta}_{i}^{*b}\;\overset{d^{*}}{\to}\;\mathcal{N}(\bm{0},\mathbf{I}_{K}).
\]

Given that $E^{*}(\widetilde{\bm{\eta}}_{N}^{*b}\widetilde{\bm{\eta}}_{N}^{*b\top})=\mathbf{I}_{N}$
and \eqref{eq: bootstrap sum}, one can deduce that 
\begin{align*}
Var^{*}\!\left(\frac{1}{\sqrt{N}}\sum_{i=1}^{N}\ddot{\bm{a}}_{i}\,{\eta}_{i}^{*b}\right) & =\frac{1}{N}\widehat{\bm{A}}_{N}\bm{\Phi}_{N}\bm{\Lambda}_{N}^{1/2}E^{*}(\widetilde{\bm{\eta}}_{N}^{*b}\widetilde{\bm{\eta}}_{N}^{*b\top})\bm{\Lambda}_{N}^{1/2}\bm{\Phi}_{N}^{\top}\widehat{\bm{A}}_{N}^{\top}\\
 & =\frac{1}{N}\widehat{\bm{A}}_{N}\mathbb{K}_{N}\widehat{\bm{A}}_{N}^{\top}\\
 & =\frac{1}{N}\sum_{i=1}^{N}\sum_{j=1}^{N}\ddot{\bm{a}}_{i}\ddot{\bm{a}}_{j}^{\top}\,\mathcal{K}\!\left(\frac{\mathfrak{d}_{ij}}{\mathfrak{d}_{N}}\right).
\end{align*}
Hence, by Lemma \ref{lemma:limit form-1}, $Var^{*}\!\left(\frac{1}{\sqrt{N}}\sum_{i=1}^{N}\ddot{\bm{a}}_{i}\,{\eta}_{i}^{*b}\right)\overset{P}{\to}\bm{\sigma}_{a}^{2}$
when $\bm{\sigma}_{a}^{2}$ is nonsingular. Therefore, Slutsky's lemma
implies 
\[
\frac{1}{\sqrt{N}}\sum_{i=1}^{N}\ddot{\bm{a}}_{i}\,{\eta}_{i}^{*b}\;\overset{d^{*}}{\to}\;\mathcal{N}(\bm{0},\bm{\sigma}_{a}^{2}).
\]

Finally, if only noisy distances $\widetilde{\mathfrak{d}}_{ij}$
is applied. Because the perturbations $\{\varsigma_{ij}\}$ are uniformly
bounded and independent of the underlying spatial process by Assumption
\ref{as:spatial}(iv), they do not alter the asymptotic contribution
of the kernel weights: any single error affects at most a negligible
fraction of pairs, while the aggregate effect averages out in the
summation. Consequently, all preceding limit results continue to hold
when $\mathfrak{d}_{ij}$ is replaced by $\widetilde{\mathfrak{d}}_{ij}$.
\end{proof}

\subsection{Bootstrap CLT for $\frac{1}{\sqrt{NT}}\sum_{i=1}^{N}\sum_{t=1}^{T}\widehat{\mathbf{e}}_{it}\eta_{i}^{*b}\eta_{t}^{*b}$.}

\begin{lemma} \label{lemma: clt for bootstrap Ze} If Assumptions
\ref{as: AHS representation}-\ref{as: same rate} hold
and $\bm{\sigma}_{e}^{2}=\lim_{N,T\to\infty}\bm{\sigma}_{e,f}^{2}>0$,
then 
\[
\frac{1}{\sqrt{NT}}\sum_{i=1}^{N}\sum_{t=1}^{T}\widehat{\bm{e}}_{it}\cdot\eta_{i}^{*b}\eta_{t}^{*b}\overset{d^{*}}{\to}\mathcal{N}(0,\bm{\sigma}_{e}^{2}),
\]
where $\widehat{\bm{e}}_{it}=\bm{e}_{it}-\frac{1}{N}\sum_{i=1}^{N}\bm{e}_{it}-\frac{1}{T}\sum_{t=1}^{T}\bm{e}_{it}+\frac{1}{NT}\sum_{i=1}^{N}\sum_{t=1}^{T}\bm{e}_{it}$.
\end{lemma} 
\begin{proof}[\textbf{Proof of Lemma \protect\ref{lemma: clt for bootstrap Ze}}]
Write $\eta_{t}^{*b}=\prod_{\iota=1}^{t}\kappa_{\iota}^{b}\eta_{0}^{*b}$
with i.i.d.\ $\kappa_{\iota}^{b}\in\{-1,1\}$ satisfying $P(\kappa_{\iota}^{b}=1)=(q+1)/2$
and $E^{*}(\kappa_{\iota}^{b})=q$. Then $E^{*}(\kappa_{\iota}^{b}-q)=0$
and $\eta_{t}^{*b}$ is bounded by one, so we decompose 
\begin{align*}
\eta_{t}^{*b} & =\eta_{0}^{*}\left(\prod_{\iota=1}^{t-1}\kappa_{\iota}^{b}\left(\kappa_{t}^{b}-q\right)+q\prod_{\iota=1}^{t-1}\kappa_{\iota}^{b}\right).
\end{align*}
Applying a similar argument, we can decompose $\frac{1}{\sqrt{NT}}\sum_{i=1}^{N}\sum_{t=1}^{T}\widehat{\bm{e}}_{it}\cdot\eta_{i}^{*b}\eta_{t}^{*b}$
as follows: 
\begin{align}
\frac{1}{\sqrt{NT}}\sum_{i=1}^{N}\sum_{t=1}^{T}\widehat{\bm{e}}_{it}\cdot\eta_{i}^{*b}\eta_{t}^{*b}= & \sum_{t=1}^{T-c+1}\bm{B}_{c,t}^{*b}\cdot\left(\kappa_{t}^{b}-q\right)\left(1+o_{P^{*}}\left(1\right)\right)\nonumber \\
 & +\sum_{t=c+1}^{T}\bm{R}_{c,t}^{*b}\cdot q^{c}+\sum_{t=1}^{c}\frac{1}{\sqrt{NT}}\sum_{i=1}^{N}\widehat{\bm{e}}_{it}\cdot\eta_{i}^{*b}\eta_{0}^{*b}q^{t},\label{eq: eit 1-1}
\end{align}
where 
\begin{align*}
\bm{B}_{c,t}^{*b}= & \frac{1}{\sqrt{NT}}\sum_{i=1}^{N}\sum_{\tau=1}^{c}\widehat{\bm{e}}_{i,s+\tau-1}\cdot\eta_{i}^{*b}\eta_{0}^{*b}q^{c-1}\prod_{\iota=1}^{t-1}\kappa_{\iota}^{b},\\
\bm{R}_{c,t}^{*b}= & \frac{1}{\sqrt{NT}}\sum_{i=1}^{N}\widehat{\bm{e}}_{it}\cdot\eta_{i}^{*b}\eta_{0}^{*b}\prod_{\iota=1}^{t-c}\kappa_{\iota}^{b}.
\end{align*}
Choosing 
\[
c=\sqrt{-\left(\ln q\right)^{-1}T^{1/2}},\qquad\mathcal{F}_{t}=\sigma\Bigl(\{\eta_{i}^{*b}\}_{i=1}^{N},\{\kappa_{\iota}^{b}\}_{\iota=1}^{t},\eta_{0}^{*b}\Bigr).
\]
Then $\bm{B}_{c,t}^{*b}$ is $\mathcal{F}_{t-1}$-measurable and $E^{*}\!\left(\bm{B}_{c,t}^{*b}(\kappa_{t}^{b}-q)\mid\mathcal{F}_{t-1}\right)=\bm{0}.$
Hence $\{\bm{B}_{c,t}^{*b}(\kappa_{t}^{b}-q),\mathcal{F}_{t}\}_{t\le T-c+1}$
is a martingale difference array. Let 
\[
\bm{V}_{B}^{*}=Var^{*}\Bigl(\sum_{t=1}^{T-c+1}\bm{B}_{c,t}^{*b}(\kappa_{t}^{b}-q)\Bigr).
\]
Because all multipliers are bounded by one and $c\to\infty$, the
Lindeberg-type condition required in Theorem~2.3 of McLeish (\citeyear{mcleish1974dependent})
holds: 
\[
\max_{1\le t\le T-c+1}\Bigl\|\bm{V}_{B}^{*-1/2}\bm{B}_{c,t}^{*b}(\kappa_{t}^{b}-q)\Bigr\|\xrightarrow{P^{*}}0\quad\text{and}\quad\bm{V}_{B}^{*}\ \text{is nonsingular w.p.a.1.}
\]
Therefore, conditional on the data, 
\begin{equation}
\sum_{t=1}^{T-c+1}\bm{V}_{B}^{*-1/2}\bm{B}_{c,t}^{*b}(\kappa_{t}^{b}-q)\xrightarrow{d^{*}}\mathcal{N}(\bm{0},\mathbf{I}).\label{eq:mart-CLT}
\end{equation}

It remains to show the remainder terms in \eqref{eq: eit 1-1} are
negligible. For $\sum_{t=c+1}^{T}\bm{R}_{c,t}^{*b}\,q^{c}$, since
$\eta_{i}^{*b}$ are i.i.d.\ over $i$ with mean zero and bounded,
\[
E^{*}\left\Vert \frac{1}{\sqrt{N}}\sum_{i=1}^{N}\widehat{\bm{e}}_{it}\eta_{i}^{*b}\right\Vert ^{2}=O_{P}(1),
\]
uniformly in $t$. Hence 
\[
\sum_{t=c+1}^{T}\bm{R}_{c,t}^{*b}\,q^{c}=\frac{1}{\sqrt{T}}\sum_{t=c+1}^{T}\Bigl(\frac{1}{\sqrt{N}}\sum_{i=1}^{N}\widehat{\bm{e}}_{it}\eta_{i}^{*b}\eta_{0}^{*b}\Bigr)q^{c-1}\prod_{\iota=1}^{t-c}\kappa_{\iota}^{b}=O_{P^{*}}(T^{1/2}q^{c-1})=o_{P^{*}}(1),
\]
since $T^{1/2}q^{c-1}\to0$ by the choice of $c$. Likewise, 
\[
\sum_{t=1}^{c}\frac{1}{\sqrt{NT}}\sum_{i=1}^{N}\widehat{\bm{e}}_{it}\eta_{i}^{*b}\eta_{0}^{*b}q^{t}=O_{P^{*}}(cT^{-1/2})=o_{P^{*}}(1),
\]
provided $-(\ln q)^{-1}=o(T^{1/2})$. Combining with \eqref{eq:mart-CLT}
yields a conditional Gaussian limit for $\frac{1}{\sqrt{NT}}\sum_{i=1}^{N}\sum_{t=1}^{T}\widehat{\bm{e}}_{it}\cdot\eta_{i}^{*b}\eta_{t}^{*b}$
after self-normalization.

Moreover, its conditional variance matches the target variance: 
\begin{align}
Var^{*}\left(\frac{1}{\sqrt{NT}}\sum_{i=1}^{N}\sum_{t=1}^{T}\widehat{\bm{e}}_{it}\cdot\eta_{i}^{*b}\eta_{t}^{*b}\right)= & \frac{1}{NT}\sum_{i=1}^{N}\sum_{i'=1}^{N}\sum_{t=1}^{T}\sum_{t'=1}^{T}\mathcal{K}\!\left(\frac{\mathfrak{d}_{ii'}}{\mathfrak{d}_{N}}\right)q^{\left|t-t'\right|}\widehat{\bm{e}}_{it}\widehat{\bm{e}}_{i't'}^{\top}\nonumber \\
= & \frac{1}{NT}\sum_{i=1}^{N}\sum_{i'=1}^{N}\sum_{t=1}^{T}\sum_{t'=1}^{T}\left(\frac{\mathfrak{d}_{ii'}}{\mathfrak{d}_{N}}\right)q^{\left|t-t'\right|}\bm{e}_{it}\bm{e}_{i't'}^{\top}+o_{P}\left(1\right)\nonumber \\
= & \frac{1}{NT}\sum_{i=1}^{N}\sum_{i'=1}^{N}\sum_{t=1}^{T}\sum_{t'=1}^{T}q^{\left|t-t'\right|}E\left(\bm{e}_{it}\bm{e}_{i't'}^{\top}\vert\left\{ \bm{\xi}_{t}\right\} _{t}\right)+o_{P}\left(1\right)\nonumber \\
= & \bm{\sigma}_{e,f}^{2}+o_{P}\left(1\right).\label{eq:boot-var-e}
\end{align}
where the second equality follows from the definition of $\widehat{\bm{e}}_{it}$,
the third equality uses a similar argument as that for \eqref{eq:key_kern_conv}, conditional on $\{\bm{\xi}_{t}\}_{t}$, and the
final equality follows from the application of Lemma~\ref{lemma: convergence in covariance}
and law of iterated expectation. By Slutsky's Theorem, we obtain the
desirable result. 
\end{proof}
\clearpage{}

\section*{Appendix IB: Propositions and Proofs for Propositions}

\setcounter{equation}{0} \setcounter{assumption}{0}
\setcounter{figure}{0} \setcounter{table}{0}
\setcounter{subsection}{0}\setcounter{proposition}{0}

 \renewcommand{%
\theequation}{IB.\arabic{equation}}
\renewcommand{\thelemma}{IB.\arabic{lemma}}
\renewcommand{\theassumption}{IB.\arabic{assumption}} \renewcommand{%
\thetheorem}{IB.\arabic{theorem}}
\renewcommand{%
\theproposition}{IB.\arabic{proposition}}
\renewcommand{\thetable}{IB.\arabic{table}}
\renewcommand{\thefigure}{IB.\arabic{figure}}

\renewcommand{\thesubsection}{IB.\arabic{subsection}}

\subsection{Discriminant Factors}

\begin{proposition}[Discriminant Factors] \label{prop: two indicator factors}

(a) For D, I\&G, V\&G, $P^{*}\!\left(D_{k}^{*}=0\right)\overset{P}{\to}1$;
for I\&N and V\&N, $P^{*}\!\left(D_{k}^{*}=1\right)\overset{P}{\to}1$.

(b) Suppose $T\sigma_{ak,f}^{2}$ grows at a rate not slower than
$N\sigma_{dk,f}^{2}$. (i) For D, I\&G, V\&G, $T\widehat{\sigma}_{ak}^{2}$
diverges, converges, and vanishes in probability whenever $T\sigma_{ak,f}^{2}$
diverges, converges, and vanishes, respectively. (ii) For I\&N and
V\&N, $T\widehat{\sigma}_{ak}^{2}=O_{P}(1)$.

(c) (i) For D, if strengthened to $T\sigma_{ak,f}^{2}>2\log T$ or
$N\sigma_{dk,f}^{2}>2\log N$, then $T\widehat{\sigma}_{ak}^{2}$
diverges in probability; (ii) For V\&G, if strengthened to $T\sigma_{ak,f}^{2}=o\!\left(1/\log T\right)$,
$N\sigma_{dk,f}^{2}=o\!\left(1/\log N\right)$, and $\sigma_{vk,f}^{2}=o\!\left(1/\log N\right)$,
then $P\!\left(T\widehat{\sigma}_{ak}^{2}<1/\log T\right)\to1$. \end{proposition}
\begin{proof}[\textbf{Proof of Proposition \protect\ref{prop: two indicator factors}}]
Following the proofs of Theorems \ref{thm: main-2} and \ref{thm: main-3},
part (a) and part (c)(i) hold.

For (b), Lemma~\ref{lemma:limit form-1} implies that, in regimes
D, I\&G, and V\&G, the leading contribution to $T\widehat{\sigma}_{ak}^{2}$
is $T\sigma_{ak,f}^{2}$, establishing (i). For (ii), the same lemma
shows that, in regime I\&N, the leading terms in $T\widehat{\sigma}_{ak}^{2}$
are $T\sigma_{ak,f}^{2}$ and $\sigma_{vf,k}^{2}$; hence $T\widehat{\sigma}_{ak}^{2}=O_{P}(1)$.

Finally, under the strengthened V\&G conditions in (c)(ii), $T\sigma_{ak,f}^{2}=o\!\left(\frac{1}{\log T}\right)$,
$N\sigma_{dk,f}^{2}=o\!\left(\frac{1}{\log N}\right)$, and $\sigma_{vk,f}^{2}=o\!\left(\frac{1}{\log N}\right)$.
It remains to show that the $o_{P}\!\left(\bm{\sigma}_{e,f}^{2}\right)$
remainder term in \eqref{eq: sigma a order}---which enters $T\widehat{\sigma}_{ak}^{2}$---is
$o_{P}\!\left(1/\log T\right)$. This is natural because the HAC estimator
converges to the true variance at rate $T^{-1/3}$ under the optimal
geometric-kernel choice $-(\ln q)^{-1}\asymp T^{1/3}$, and $T^{-1/3}=o\!\left(1/\log T\right)$
as $T\to\infty$. By Lemma \ref{lemma:limit form-1}, it follows
that 
\[
P\!\left(T\widehat{\sigma}_{ak}^{2}<\frac{1}{\log T}\right)\to1.
\]
\end{proof}

\subsection{Impossibility Results}

\label{subsec:impossibility}

This subsection records three related lower bounds when the DGP is
left fully unspecified. Throughout, for each DGP $f$ and sample size
$(N,T)$, let $P_{NT,f}$ denote the probability law of the observed
array $\{(\bm{y}_{it}^{(f)},\bm{X}_{it}^{(f)})\}_{i=1,\dots,N;\,t=1,\dots,T}$.

\begin{proof}[\textbf{Proof of Proposition~\ref{prop: impossibility heter score}}]
It is enough to exhibit a two-point submodel contained in \(\mathcal B_0\)
for which the limiting laws are different, but the two DGPs are not
asymptotically distinguishable from the observed data. Consider the model
\[
y_{it}=\beta_{0}+\beta_{1}x_{it}+u_{it},\qquad
x_{it}=\alpha_i,\qquad
u_{it}=\xi_t+\varepsilon_{it},
\]
where
\[
\alpha_i\stackrel{i.i.d.}{\sim}N(0,1),\qquad
\xi_t\stackrel{i.i.d.}{\sim}N(0,1),
\]
and \(\{\alpha_i\}\), \(\{\xi_t\}\), and \(\{\varepsilon_{it}\}\) are mutually
independent. For a fixed constant \(m>0\), let
\[
\varepsilon_{it}=
\begin{cases}
m, & \text{with probability } \dfrac{1}{2NT},\\[6pt]
-m, & \text{with probability } \dfrac{1}{2NT},\\[6pt]
0, & \text{with probability } 1-\dfrac{1}{NT}.
\end{cases}
\]
Let \(f_m\) denote the corresponding DGP. Under \(f_m\),
\[
E(\varepsilon_{it})=0,\qquad
E(\varepsilon_{it}^{2})=\frac{m^2}{NT}.
\]
Let \(\widehat\beta_1\) be the OLS slope estimator from the regression of
\(y_{it}\) on an intercept and \(x_{it}\). Since the regression contains an
intercept, with
\(
\bar\alpha=\frac{1}{N}\sum_{i=1}^N \alpha_i,
\)
we have the exact finite-sample representation
\(
\widehat\beta_1-\beta_1
=
\frac{
\sum_{i=1}^N\sum_{t=1}^T(\alpha_i-\bar\alpha)u_{it}
}{
T\sum_{i=1}^N(\alpha_i-\bar\alpha)^2
}.
\)
Substituting \(u_{it}=\xi_t+\varepsilon_{it}\) gives
\[
\sum_{i=1}^N\sum_{t=1}^T(\alpha_i-\bar\alpha)u_{it}
=
\sum_{t=1}^T\xi_t\sum_{i=1}^N(\alpha_i-\bar\alpha)
+
\sum_{i=1}^N\sum_{t=1}^T(\alpha_i-\bar\alpha)\varepsilon_{it}.
\]
The first term is zero because
\(
\sum_{i=1}^N(\alpha_i-\bar\alpha)=0.
\)
Therefore, multiplying by \(NT\), we obtain
\[
NT(\widehat\beta_1-\beta_1)
=
\frac{
N
}{
\sum_{i=1}^N(\alpha_i-\bar\alpha)^2
}
\sum_{i=1}^N\sum_{t=1}^T(\alpha_i-\bar\alpha)\varepsilon_{it}.
\]
Since
\(
\frac{1}{N}\sum_{i=1}^N(\alpha_i-\bar\alpha)^2\overset{p}{\to}1,
\)
it remains to characterize
\(
\sum_{i=1}^N\sum_{t=1}^T(\alpha_i-\bar\alpha)\varepsilon_{it}.
\)
Let
\(
K_{NT}=\sum_{i=1}^N\sum_{t=1}^T 1\{\varepsilon_{it}\neq 0\}.
\)
Then
\(
K_{NT}\sim \operatorname{Binomial}\left(NT,\frac{1}{NT}\right)
\overset{d}{\to} K,
\) where \(
K\sim \operatorname{Poisson}(1).
\)
Conditional on the event that \(\varepsilon_{it}\neq 0\), its sign is equally
likely to be positive or negative. Moreover, the corresponding
\(\alpha_i-\bar\alpha\) is asymptotically standard normal, and the random sign of $\varepsilon_{it}$ does not affect the limit by symmetry, so we have
\(
\sum_{i=1}^N\sum_{t=1}^T(\alpha_i-\bar\alpha)\varepsilon_{it}
\overset{d}{\to}
m\sum_{\ell=1}^{K}Z_{\ell},
\)
where \(K\sim \operatorname{Poisson}(1)\), \(\{Z_{\ell}\}_{\ell\geq1}\) are
i.i.d. \(N(0,1)\), and \(K\) is independent of \(\{Z_{\ell}\}\). Therefore,
under \(f_m\),
\[
NT(\widehat\beta_1-\beta_1)
\overset{d}{\to}
L_m
\equiv
m\sum_{\ell=1}^{K}Z_{\ell}.
\]
Let \(G_m\) denote the distribution function of \(L_m\). If \(m_1\neq m_2\),
then \(G_{m_1}\neq G_{m_2}\). For example,
\(
\operatorname{Var}(L_m)
=
m^2 E(K)
=
m^2,
\)
so the limiting distributions differ whenever \(m_1\neq m_2\).

We now show that the two cases \(m=m_1\) and \(m=m_2\) cannot be
consistently distinguished. Let \(P_{m,NT}\) denote the joint distribution
of the observed data under \(f_m\). Consider the event
\[
A_{NT}
=
\left\{
\varepsilon_{it}=0\ \text{for all } i=1,\ldots,N,\ t=1,\ldots,T
\right\}.
\]
For every \(m>0\),
\[
P_{m,NT}(A_{NT})
=
\left(1-\frac{1}{NT}\right)^{NT}
\to e^{-1}.
\]
On \(A_{NT}\), the model reduces to
\[
y_{it}=\beta_0+\beta_1\alpha_i+\xi_t,
\]
which does not depend on \(m\). Hence the two DGPs \(f_{m_1}\) and
\(f_{m_2}\) share the same equation with asymptotic probability \(e^{-1}\).
Consequently, their total variation distance is bounded away from one:
\[
\limsup_{N,T\to\infty}
\left\|P_{m_1,NT}-P_{m_2,NT}\right\|_{\mathrm{TV}}
\leq 1-e^{-1}<1.
\]
By Le Cam's testing bound, for any test \(\varphi_{NT}\),
\[
P_{m_1,NT}(\varphi_{NT}=1)
+
P_{m_2,NT}(\varphi_{NT}=0)
\geq
1-
\left\|P_{m_1,NT}-P_{m_2,NT}\right\|_{\mathrm{TV}}.
\]
Therefore,
\[
\liminf_{N,T\to\infty}
\left\{
P_{m_1,NT}(\varphi_{NT}=1)
+
P_{m_2,NT}(\varphi_{NT}=0)
\right\}
\geq e^{-1}>0.
\]
Thus no test can consistently distinguish \(m=m_1\) from \(m=m_2\).

Suppose, toward a contradiction, that some data-dependent procedure
\(\widehat D\in\mathcal D\) estimated the limiting law uniformly over
\(\mathcal B_0\). Since \(G_{m_1}\neq G_{m_2}\), choose
\[
0<\varepsilon<
\frac14\|G_{m_1}-G_{m_2}\|_{\infty}.
\]
Define the test
\[
\varphi_{NT}
=
1\left\{
\|\widehat D-G_{m_2}\|_{\infty}
<
\|\widehat D-G_{m_1}\|_{\infty}
\right\}.
\]
Uniform consistency of \(\widehat D\) would imply
\[
P_{m_1,NT}(\varphi_{NT}=1)\to 0,
\qquad
P_{m_2,NT}(\varphi_{NT}=0)\to 0.
\]
This contradicts the lower bound above. Hence no data-dependent procedure
can uniformly consistently estimate the limiting distribution of
\(NT(\widehat\beta_1-\beta_1)\) over this class of DGPs.

\end{proof}

\begin{proposition}[Impossibility due to the infeasible Regime] \label{prop: impossibility}
Suppose the null hypothesis $\mathcal{H}_{0}:\bm\varrho^\top\bm{\beta}=\bm\varrho^\top\bm{\beta}_{0}$
holds and the DGP is fully unspecified. Let $\mathcal{D}$ be the
collection of all measurable maps from the observed data $\{(\bm{y}_{it},\bm{X}_{it})\}_{i,t}$.

Let $\mathcal{B}_{1}$ and $\mathcal{B}_{2}$ denote classes of DGPs
satisfying condition~\eqref{eq: converge non-gaussian} for some
$k$ and condition~\eqref{eq: vanish nongaussian} all $k$, respectively.
Moreover, all functions in $\mathcal{B}_{1}$ and $\mathcal{B}_{2}$
satisfy Assumptions~$\ref{as: AHS representation}$-$\ref{as: same rate}$.

\begin{enumerate}[label=(\alph*),leftmargin=2.2em]
\item \textbf{(Infeasible regime for uniform consistency).} There exists
$\varepsilon_{1}>0$ such that 
\[
\liminf_{N,T\to\infty}\;\inf_{\widehat{D}_{1}\in\mathcal{D}}\;\sup_{f\in\mathcal{B}_{1}}P_{NT,f}\!\left(\bigl\|\widehat{D}_{1}(\{(\bm{y}_{it}^{(f)},\bm{X}_{it}^{(f)})\}_{i,t})-\mathcal{L}_{0}\!\left(\bm{\nu}_{f},\{\bm{c}_{ll',f}\}_{l,l'=1}^{\infty}\right)\bigr\|_{\infty}>\varepsilon_{1}\right)>0.
\]
In particular, no procedure can estimate the target limit object $\mathcal{L}_{0}(\bm{\nu}_{f},\{\bm{c}_{l,f}\}_{l\ge1})$
uniformly over $\mathcal{B}_{1}$.
\item \textbf{(Impossible to distinguish an infeasible regime from a feasible
regime).} Let $\{(\bm{y}_{it}^{(f_{1})},\bm{X}_{it}^{(f_{1})})\}_{i,t}$
and $\{(\widetilde{\bm{y}}_{it}^{(f_{2})},\widetilde{\bm{X}}_{it}^{(f_{2})})\}_{i,t}$
be two independent observed samples generated under $f_{1}$ and $f_{2}$,
respectively. Let $\mathcal{D}$ be the class of all (possibly randomized)
tests $\widehat{D}$ measurable with respect to $\{(\bm{y}_{it},\bm{X}_{it})\}_{i,t}$.
Then 
\[
\liminf_{N,T\to\infty}\;\inf_{\widehat{D}\in\mathcal{D}}\;\sup_{f_{1}\in\mathcal{B}_{1},\,f_{2}\in\mathcal{B}_{2}}P_{NT,f_{1}f_{2}}\!\left(\widehat{D}\!\left(\{(\bm{y}_{it}^{(f_{1})},\bm{X}_{it}^{(f_{1})})\}_{i,t}\right)=\widehat{D}\!\left(\{(\widetilde{\bm{y}}_{it}^{(f_{2})},\widetilde{\bm{X}}_{it}^{(f_{2})})\}_{i,t}\right)\right)=1.
\]
Equivalently, even the best test yields (asymptotically) the same
decision under the worst cases in $\mathcal{B}_{1}$ and $\mathcal{B}_{2}$.
\item \textbf{(No uniform conservativeness if uniformly exact over $\mathcal{B}_{2}$).}
Let $\widehat{D}_{2}\in\mathcal{D}_2$ index a confidence interval $CI_{NT,f}(\widehat{D}_2,\alpha)$
with a significance level $\alpha$ for $\bm{\beta}_{0}$ that is
uniformly asymptotically exact for all function $f\in\mathcal{B}_{2}$.
Then there exists $\varepsilon_{2}>0$ such that 
\[
\liminf_{N,T\to\infty}\sup_{\widehat{D}_2\in\mathcal{D}_2}\inf_{f\in\mathcal{B}_{1}}P_{NT,f}\!\left(\bm{\beta}_{0}\in CI_{NT,f}\bigl(\widehat{D}_2,\alpha\bigr)\right)\;\le\;1-\alpha-\varepsilon_{2}.
\]
Hence, any procedure that is uniformly asymptotically exact over $\mathcal{B}_{2}$
must undercover somewhere in $\mathcal{B}_{1}$ by a non-vanishing
amount.
\end{enumerate}
\end{proposition}

\begin{proof}[\textbf{Proof of Proposition~\ref{prop: impossibility}}]
We first construct a counterexample, in the spirit of Menzel~(\citeyear{menzel2021bootstrap};
cf.\ Proposition~4.1).

\medskip{}
\textbf{A parametric model.} Consider the scalar model (suppressing
$k$ for simplicity) 
\[
y_{it}=x_{it}\beta_{0}+u_{it},\qquad x_{it}=\alpha_{i}^{x}\xi_{t}^{x},\qquad u_{it}=\alpha_{i}^{u}\xi_{t}^{u},
\]
where $\bm{\alpha}_{i}\equiv(\alpha_{i}^{x},\alpha_{i}^{u})^{\top}\overset{i.i.d.}{\sim}\mathcal{N}(\bm{0},\mathbf{I}_{2})$
and $\bm{\xi}_{t}\equiv(\xi_{t}^{x},\xi_{t}^{u})^{\top}\overset{i.i.d.}{\sim}\mathcal{N}((1,T^{-1/2}c)^{\top},\mathbf{I}_{2})$
for a fixed constant $c\ge0$. This model satisfies Assumptions~\ref{as: AHS representation}--\ref{as: same rate}.

\noindent Let $\widehat{\beta}$ denote the OLS estimator. A direct
calculation yields 
\[
\sqrt{NT}\,(\widehat{\beta}-\beta_{0})=\left(\frac{1}{NT}\sum_{i=1}^{N}\sum_{t=1}^{T}x_{it}^{2}\right)^{-1}\left(\frac{1}{\sqrt{N}}\sum_{i=1}^{N}\alpha_{i}^{x}\alpha_{i}^{u}\right)\left(\frac{1}{\sqrt{T}}\sum_{t=1}^{T}(\xi_{t}^{x}\xi_{t}^{u}-T^{-1/2}c)+c\right).
\]
By LLN and CLT, as $N,T\to\infty$, 
\[
\frac{1}{NT}\sum_{i,t}x_{it}^{2}=\left(\frac{1}{N}\sum_{i=1}^{N}(\alpha_{i}^{x})^{2}\right)\left(\frac{1}{T}\sum_{t=1}^{T}(\xi_{t}^{x})^{2}\right)\xrightarrow{P}2,
\]
and 
\[
\frac{1}{\sqrt{N}}\sum_{i=1}^{N}\alpha_{i}^{x}\alpha_{i}^{u}\xrightarrow{d}Z_{\alpha},\qquad\frac{1}{\sqrt{T}}\sum_{t=1}^{T}(\xi_{t}^{x}\xi_{t}^{u}-T^{-1/2}c)\xrightarrow{d}Z_{\xi},
\]
where $Z_{\alpha}$ and $Z_{\xi}$ are independent $\mathcal{N}(0,1)$
random variables. Consequently, 
\begin{equation}
\sqrt{NT}\,(\widehat{\beta}-\beta_{0})\xrightarrow{d}\frac{1}{2}Z_{\alpha}\,(Z_{\xi}+c).\label{eq:limit_model}
\end{equation}
When $c=0$, the limit in \eqref{eq:limit_model} corresponds to a
V\&N regime ($\mathcal{B}_{2}$); when $c>0$, the limit is in I\&N
regime ($\mathcal{B}_{1}$) and its dispersion increases with $c$
since $Var\!\bigl(\frac{1}{2}Z_{\alpha}(Z_{\xi}+c)\bigr)=\frac{1}{4}(1+c^{2})$.

Notice that $T^{-1/2}c=E(\xi_{t}^{u})$ cannot be estimated at a better
rate than by directly observing $\left\{ \xi_{t}^{u}\right\} _{t=1}^{T}$,
but the estimation error of the expectation is of order $O_{P}\left(T^{-1/2}\right)$.
Hence, there exists no consistent test that separates $c=0$ from
$c=1$ (or $c=1$ from $c=2$) based on the observed data $\{(y_{it},x_{it})\}_{i,t}$.

\medskip{}
\textbf{Proof of (a).} Pick two fixed constants $c_{1}\neq c_{2}$
(e.g.\ $c_{1}=1$, $c_{2}=2$), and let $f^{(c)}$ denote the DGP
in the model above. Then \eqref{eq:limit_model} implies that the
target limit object $\mathcal{L}_{0}(\bm{\nu}_{f^{(c)}},\{\bm{c}_{l,f^{(c)}}\}_{l\ge1})$
depends on $c$ through the law of $\frac{1}{2}Z_{\alpha}(Z_{\xi}+c)$,
and hence differs under $c_{1}$ and $c_{2}$.

\noindent If $\widehat{D}_{1}$ were uniformly consistent over $\mathcal{B}_{1}$,
one could construct a consistent test between $c=c_{1}$ and $c=c_{2}$
by comparing $\widehat{D}_{1}$ to the two distinct limits (e.g.\ in
$\|\cdot\|_{\infty}$), a contradiction. Therefore, uniform consistency
fails, which establishes (a).

\medskip{}
\textbf{ Proof of (b).} Take the same model and consider two DGPs
that differ only in the value of $c$, namely $c=0$ and $c=1$. The
case $c=0$ belongs to $\mathcal{B}_{2}$, whereas $c=1$ belongs
to $\mathcal{B}_{1}$, yet these two DGPs remain asymptotically not
separable. If there existed a procedure that could distinguish them
with probability tending to one, then it would immediately yield a
consistent test for $H_{0}:c=c_{0}$ (in particular $c_{0}=0$) against
$H_{1}:c=1$, contradicting the non-separability result. Hence, no
test can reliably distinguish these two DGPs, which is exactly the
statement in~(b).

\noindent \medskip{}
\textbf{Proof of (c).} Recall that $\widehat{D}_{2}$ is a procedure
that is uniformly asymptotically exact over $\mathcal{B}_{2}$. Hence,
under $c=0$ the implied asymptotic law/critical value is calibrated
to the $c=0$ limit in \eqref{eq:limit_model}. By part~(b), for
any data-based calibration rule the induced critical value under $c>0$
cannot be reliably distinguished from the one under $c=0$. Therefore,
if the procedure is asymptotically exact at $c=0$, then when $c>0$
it must produce a critical value that is asymptotically not separable
from its $c=0$ counterpart. Since $Var(\frac{1}{2}Z_{\alpha}Z_{\xi})=\frac{1}{4}$
but $Var(\frac{1}{2}Z_{\alpha}(Z_{\xi}+c))=\frac{1}{4}(1+c^{2})$,
for sufficiently large fixed $c$ the $c=0$ calibration yields rejection
probability strictly exceeding $\alpha$, equivalently coverage strictly
below $1-\alpha$ by a non-vanishing margin. This delivers (c). 
\end{proof}
\clearpage{}

\section*{Appendix IC: Heterogeneous
Intersection Sizes}

\setcounter{equation}{0} \setcounter{assumption}{0}
\setcounter{figure}{0} \setcounter{table}{0}
\setcounter{subsection}{0}\setcounter{proposition}{0}

 \renewcommand{%
\theequation}{IC.\arabic{equation}}
\renewcommand{\thelemma}{IC.\arabic{lemma}}
\renewcommand{\theassumption}{IC.\arabic{assumption}} \renewcommand{%
\thetheorem}{IC.\arabic{theorem}}
\renewcommand{%
\theproposition}{IC.\arabic{proposition}}
\renewcommand{\thetable}{IC.\arabic{table}}
\renewcommand{\thefigure}{IC.\arabic{figure}}

\renewcommand{\thesubsection}{IC.\arabic{subsection}}

In this section, we allow for heterogeneous numbers of observations
and missingness across intersections. Recall that $M_{it}$ denotes
the number of observations in cell $(i,t)$, and define the average
cell size 
\[
\bar{M}\equiv\frac{1}{NT}\sum_{i=1}^{N}\sum_{t=1}^{T}M_{it},\qquad M_{i\cdot}\equiv\frac{1}{T}\sum_{t=1}^{T}M_{it},\qquad M_{\cdot t}\equiv\frac{1}{N}\sum_{i=1}^{N}M_{it}.
\]

\begin{assumption}\label{as: heterogeneous inter size} (i) $\liminf_{N,T\to\infty}\bar{M}\ >\ 0$
(ii) there exist $c_{1},c_{2}<\infty$ such that $\frac{1}{N}\sum_{i=1}^{N}({M_{i\cdot}}/{\bar{M}})^{2}\to c_{1}$
and $\frac{1}{T}\sum_{t=1}^{T}({M_{\cdot t}}/{\bar{M}})^{2}\to c_{2}.$
(iii) The function $f$ and the cluster effects $\{\bm{\alpha}_{i}\}_{i\le N}$,
$\{\bm{\xi}_{t}\}_{t\le T}$, and $\{\bm{\varepsilon}_{it}\}_{i\le N,\ t\le T}$
are independent of the array $\{M_{it}\}_{i\le N,\ t\le T}$. \end{assumption}

Part (i) allows $\bar{M}$ to diverge, while ruling out a vanishing
effective sample size by requiring the average number of observations
per intersection to stay bounded away from zero. Hence, some intersections
may be empty ($M_{it}=0$), but not an asymptotically overwhelming
fraction of them. Part (ii) controls imbalance through row and column
totals in an $L^{2}$ sense. In particular, bounding $\frac{1}{N}\sum_{i=1}^{N}(M_{i\cdot}/\bar{M})^{2}$
and $\frac{1}{T}\sum_{t=1}^{T}(M_{\cdot t}/\bar{M})^{2}$ rules out
dominant rows or columns and implies that no single intersection $(i,t)$
can account for a non-negligible fraction of the total sample size.
Hence, extreme outliers in cluster size are excluded at the level
relevant for the CLT and variance estimation (cf.\ Chiang, Sasaki,
and Wang~\citeyearpar{chiang2023genuinely}). This condition places
no restriction on the magnitude of $\bar{M}$ itself: $\bar{M}$ may
be close to one or diverge with $(N,T)$, as long as it does not vanish.
Part (iii) is an exogeneity condition ensuring that the pattern of
missingness or multiplicity across intersections is independent of
the underlying cluster effects.

PWB-H accommodates heterogeneous intersection sizes automatically.
Under Assumption~\ref{as: heterogeneous inter size} and the assumptions
of Theorems~\ref{thm: main}-\ref{thm: main-3}, the analogous conclusions in the main
text continue to hold with minor modifications.

\clearpage{}

\section*{Appendix ID: Additional Simulation Results}

\setcounter{equation}{0} \setcounter{assumption}{0}
\setcounter{figure}{0} \setcounter{table}{0}
\setcounter{subsection}{0}\setcounter{proposition}{0}
\setcounter{assumption}{0}

 \renewcommand{%
\theequation}{ID.\arabic{equation}}
\renewcommand{\thelemma}{ID.\arabic{lemma}}
\renewcommand{\theassumption}{ID.\arabic{assumption}} \renewcommand{%
\thetheorem}{ID.\arabic{theorem}}
\renewcommand{%
\theproposition}{ID.\arabic{proposition}}
\renewcommand{\thetable}{ID.\arabic{table}}
\renewcommand{\thefigure}{ID.\arabic{figure}}

\renewcommand{\thesubsection}{ID.\arabic{subsection}}

\subsection{Results for the Heteroskedasticity Design}

We examine the finite-sample performance of the proposed procedures
under heteroskedasticity. The design is based on the simulation DGP in
\eqref{eq: simulation DGP}, with the covariates \(x_{it,k}\) generated as in
\eqref{eq: dgp 1}. To introduce conditional heteroskedasticity, we generate
the regression disturbance as
\[
u_{it}
=
(1+0.5x_{it,5})
\left(\alpha_i^u+\xi_t^u+\varepsilon_{it}^u\right),
\]
where the scale of the error term depends on the fifth covariate. This
specification preserves the two-way clustered dependence structure in
\eqref{eq: dgp 1}, while allowing the conditional variance of \(u_{it}\) to
vary with \(x_{it,5}\).

Table~\ref{tab: simulation hetero} reports the corresponding rejection
frequencies. Relative to the homoskedastic benchmark in \eqref{eq: dgp 1},
the finite-sample performance is somewhat weaker, reflecting the additional
difficulty introduced by heteroskedasticity. Nevertheless, the overall
pattern remains similar to that in the baseline design. In particular, the
three procedures deliver the same rejection frequencies across
the sample sizes considered. Moreover, as \(N\) and \(T\) increase, the
rejection frequencies move closer to the nominal level.

Overall, these results suggest that the proposed inference procedures are
reasonably robust to heteroskedasticity of this form, although the finite-sample distortions are somewhat larger than in the corresponding
homoskedastic design.

\begin{table}[h!]
{\centering{{%
\begin{tabular}{lcccccccc}
\hline 
\hline 
$N,T$  & 16  & 25  & 36  & 64  & 100  & 144  & 196 \tabularnewline
PWB-V  &0.211 &0.169 &0.141 &0.122 &0.112 &0.099 &0.089 \\
PWB-D  &0.211 &0.169 &0.141 &0.122 &0.112 &0.099 &0.089 \\
PWB-H  & 0.211 &0.169 &0.141 &0.122 &0.112 &0.099 &0.089 \\
\hline\hline
\end{tabular}} \caption{\textbf{Rejection Frequency for heteroskedasticity design.}
 For each bootstrap
method, $B=999$. Results are based on 5,000 Monte Carlo replicates.
The predetermined significance level is 5\%.}
\label{tab: simulation hetero} } }
\end{table}

\subsection{Results for the Nonseparable Panel Model}

We next consider a nonseparable panel model. As in the previous designs, we
base the simulation on \eqref{eq: simulation DGP}, with the covariates
\(x_{it,k}\) generated according to \eqref{eq: dgp 1}. The difference is that
the regression disturbance is now generated through a nonlinear function of
the individual and time effects. Specifically, following the Gaussian-kernel
specification used in Fern{\'a}ndez-Val, Freeman, and Weidner
\citeyearpar{fernandez2021low} and Chen, Fern{\'a}ndez-Val, and Weidner
\citeyearpar{chen2021nonlinear}, we set
\[
u_{it}
=
\frac{1}{\sqrt{2\pi}\sigma}
\exp\left\{-\frac{(\alpha_i^u-\xi_t^u)^2}{\sigma^2}\right\}
+\varepsilon_{it}^u.
\]
This specification generates a nonlinear and nonseparable dependence
structure between the individual effect \(\alpha_i^u\) and the time effect
\(\gamma_t^u\). The parameter \(\sigma\) controls the smoothness of the
kernel component. Smaller values of \(\sigma\) generate a less smooth kernel
surface and imply a slower decay of the singular values, whereas larger
values of \(\sigma\) correspond to a smoother structure with faster singular
value decay. To examine the sensitivity of the procedures to this form of
smoothness, we consider \(\sigma\in\{0.1,1,10\}\).

Table~\ref{tab: simulation nonseparable} reports the corresponding rejection
frequencies. The results show that varying the smoothness parameter has only
a limited effect on the finite-sample performance of the procedures. Across
the three values of \(\sigma\), the overall pattern is close to that observed
under \eqref{eq: dgp 3}: PWB-V and PWB-H deliver rejection frequencies close
to the nominal level as the sample size increases, whereas PWB-D remains
distorted. This indicates that, in the present nonseparable design, the
cluster dependence along the two dimensions is relatively weak but still
non-negligible. Consequently, procedures that account for the relevant
two-way dependence structure, such as PWB-V and PWB-H, continue to perform
well, while PWB-D fails to provide reliable inference.

\begin{table}[h!]
{\centering{{%
\begin{tabular}{lccccccc}
\hline 
\hline 
$N,T$  & 16  & 25  & 36  & 64  & 100  & 144  & 196 \tabularnewline
\multicolumn{8}{c}{Panel A: $\sigma=0.1$}\\
PWB-V  &0.114 &0.085 &0.077 &0.070&0.055 &0.053 &0.052 \\
PWB-D  &0.201 &0.162 &0.147 &0.124 &0.110 &0.112 &0.105 \\
PWB-H &0.114 &0.085 &0.077 &0.070 &0.055 &0.053 &0.052 \\
\multicolumn{8}{c}{Panel B: $\sigma=1.0$}\\
PWB-V  &0.114 &0.079 &0.075 &0.064 &0.055 &0.056 &0.053 \\
PWB-D  &0.213 &0.162 &0.141 &0.126 &0.094 &0.102 &0.108 \\
PWB-H  & 0.114 &0.079 &0.075 &0.064 &0.056 &0.056 &0.053 \\
\multicolumn{8}{c}{Panel C: $\sigma=10$}\\
PWB-V  &0.114 &0.080 &0.068 &0.054 & 0.055 & 0.050 &0.050\\
PWB-D  &0.212 &0.155 &0.126 &0.108 & 0.093 &0.097 &0.098 \\
PWB-H  &0.114 &0.080 &0.068 &0.054 & 0.055 & 0.050 &0.050\\
\hline\hline
\end{tabular}} \caption{\textbf{Rejection Frequency for the nonseparable panel model.}
 For each bootstrap
method, $B=999$. Results are based on 5,000 Monte Carlo replicates.
The predetermined significance level is 5\%.}
\label{tab: simulation nonseparable} } }
\end{table}

\subsection{Results for Varying Levels of Spatial Dependence}

We also examine the sensitivity of the simulation results to different
levels of spatial dependence. Table~\ref{tab: simulation spatial dependence}
reports rejection frequencies as the spatial-dependence parameter
\(\rho_{\mathfrak d}\) varies. As an additional benchmark, we compare the
projection-based wild bootstrap (PWB) procedures with the cluster-robust
variance estimator proposed by Chiang, Hansen, and Sasaki
\citeyearpar{chiang2023standard}, denoted by CHS-CRVE.

\begin{table}[h!]
{\centering{{%
\begin{tabular}{lccccccccc}
\hline 
\hline 
$\rho_{\mathfrak{d}}$  & 0.00  & 0.05  & 0.10 & 0.15  & 0.20  & 0.25 & 0.30 & 0.35 & 0.40 \tabularnewline
PWB-V  &0.077 &0.068 &0.077 &0.081 &0.096 &0.115 &0.136 &0.161 &0.194 \\
PWB-D  &0.077 &0.068 &0.077 &0.081 &0.096 &0.115 &0.136 &0.161 &0.194 \\
PWB-H  & 0.077 &0.068 &0.077 &0.081 &0.096 &0.115 &0.136 &0.161 &0.194 \\
CHS-CRVE & 0.070 &0.060 &0.073 &0.086 &0.120 &0.171 &0.239 &0.306 &0.391 \\
\hline\hline
\end{tabular}} \caption{\textbf{Rejection Frequency for varying levels of spatial dependence.}
For each bootstrap
method, $B=999$. Results are based on 5,000 Monte Carlo replicates.
The predetermined significance level is 5\%.}
\label{tab: simulation spatial dependence} } }
\end{table}

When spatial dependence is weak, specifically when
\(\rho_{\mathfrak d}\leq 0.10\), the PWB procedures perform slightly worse
than CHS-CRVE. This reflects the fact that the additional robustness of the
PWB procedures may come with a modest finite-sample cost when spatial
dependence is negligible or very weak. However, as spatial dependence
becomes stronger, the relative performance changes substantially. When
\(\rho_{\mathfrak d}\geq 0.15\), the PWB procedures begin to outperform
CHS-CRVE, and the difference becomes increasingly pronounced as
\(\rho_{\mathfrak d}\) increases. For example, when
\(\rho_{\mathfrak d}=0.40\), the PWB procedures have a rejection frequency
of approximately \(0.194\), whereas CHS-CRVE has a rejection frequency of
\(0.391\), which is almost twice as large. These results suggest that the
PWB procedures are more robust to general spatial dependence. Although this
robustness can slightly deteriorate finite-sample performance when spatial
dependence is absent or weak, it becomes especially valuable when spatial
dependence is non-negligible.

In the main text, however, we continue to use a relatively weak level of
spatial dependence. This choice is intentional. The purpose of the main
simulation design is to illustrate the asymptotic regimes predicted by the
theory as clearly as possible, without confounding the interpretation of the
results. In particular, PWB-V is not valid in the V\&G regime and tends to be
undersized in that case. If we were to choose a large value of
\(\rho_{\mathfrak d}\), the additional spatial dependence could increase the
rejection frequency of PWB-V and make it appear closer to the nominal level,
which would be misleading. With stronger spatial dependence, one can still
recover the asymptotic pattern by using sufficiently large numbers of
clusters \(N\) and \(T\), but doing so would substantially increase the
computational burden. For this reason, the main text adopts a weak spatial
dependence design, while Table~\ref{tab: simulation spatial dependence}
separately documents the robustness of the PWB procedures to stronger forms
of spatial dependence.

 \clearpage

\end{document}